\documentclass[1 leqno,11pt,psfig]{amsart}
\usepackage{amssymb, amsmath,amsmath,latexsym,amssymb,amsfonts,amsbsy, amsthm,mathtools,graphicx,CJKutf8,CJKnumb,CJKulem}
\usepackage{graphicx,epsfig}
\usepackage{graphicx}
\usepackage{amssymb}
\usepackage{graphicx}
\usepackage{graphicx,epsfig}
\usepackage{amsmath}
\usepackage{mathrsfs}
\usepackage{amssymb}
\usepackage[colorlinks=true,citecolor=blue,bookmarks=true]{hyperref}

\usepackage{amssymb,mathrsfs,amsfonts,amsmath,amsbsy,tikz}
\usepackage{fontenc}
\usepackage{textcomp}

\numberwithin{equation}{section}

\allowdisplaybreaks

\let\g=\gamma

\newcommand{\beq}{\begin{equation}}
\newcommand{\eeq}{\end{equation}}
\newcommand{\ben}{\begin{eqnarray}}
\newcommand{\een}{\end{eqnarray}}
\newcommand{\beno}{\begin{eqnarray*}}
\newcommand{\eeno}{\end{eqnarray*}}

\newtheorem{theorem}{Theorem}[section]

\newtheorem{lemma}[theorem]{Lemma}

\newtheorem{corollary}{Corollary}[section]

\newtheorem{remark}[theorem]{Remark}

\begin{document}

\title{Turning point principle for the stability of viscous gaseous stars}
\author{Ming Cheng}
\address{College of Mathematics, Jilin University, Changchun, 130012, China.}
\email{mcheng314@jlu.edu.cn}

\author{Zhiwu Lin}
\address{School of Mathematical Sciences, Fudan University, Shanghai, 200433, China}
\email{zwlin@fudan.edu.cn}

\author{Yucong Wang}
\address{School of Mathematics and Computational Science, Xiangtan University,  411105, China}
\email{yucongwang666@163.com}

\date{}
\maketitle

\begin{abstract}
We investigate the stability of non-rotating viscous gaseous stars, which are modeled by the Navier-Stokes-Poisson system. Based on general hypotheses concerning the equation of states, we demonstrate that the count of unstable modes in the linearized Navier-Stokes-Poisson system matches that of the linearized Euler-Poisson system modeling inviscid gaseous stars.  In particular, the turning point
principle holds  for the non-rotating stars with viscosity. This principle asserts that the stability of these stars  is determined by the mass-radius curve parameterized by
the center density. The transition of  stability only occurs at the extrema of the total mass. To substantiate our claims, we  formulate an infinite-dimensional Kelvin-Tait-Chetaev Theorem
for a class of abstract second-order linear equations with dissipation. Moreover, we prove
that linear stability implies  nonlinear asymptotic stability and  linear
instability implies   nonlinear instability for the Navier-Stokes-Poisson system under spherically symmetric perturbations.
\end{abstract}

\section{Introduction}

Consider the Navier-Stokes-Poisson system modeling Newtonian self-gravitating
viscous gaseous stars
\cite{DBZA2005,FDZT2009,JangJ2010,JJTI2013,LX2018,LX2019,LuoTao2019,LTXZZH2016,LTXZZH20162}
\begin{equation}
\left\{
\begin{array}
[c]{ll}%
\rho_{t}+\nabla\cdot(\rho\mathbf{u})=0 & \ \text{in}\ \Omega(t),\\
(\rho\mathbf{u})_{t}+\nabla\cdot(\rho\mathbf{u}\otimes\mathbf{u}%
)+\nabla\cdot\mathbf{S}=-\rho\nabla V & \ \text{in}\ \Omega(t),\\
\Delta V=4\pi\rho & \ \text{in}\ \mathbb{R}^{3},\\
\lim_{\left\vert \mathbf{x}\right\vert \rightarrow\infty}V\left(
t,\mathbf{x}\right)  =0. &
\end{array}
\right.  \label{NSP}%
\end{equation}
Here, $(\mathbf{x},t)\in\mathbb{R}^{3}\times\lbrack0,+\infty)$, $\rho
(\mathbf{x},t)\geq0$ is the density with support $\Omega(t)\subset
\mathbb{R}^{3}$, $\mathbf{u}(\mathbf{x},t)\in\mathbb{R}^{3}$ is the velocity
vector, $V(\mathbf{x},t)$ is the self-consistent gravitational potential. The stress tensor is
\[
\mathbf{S}=PI_{3\times3}-\nu_{1}\left(\nabla\mathbf{u}+\nabla\mathbf{u}^{T}%
-\frac{2}{3}(\nabla\cdot\mathbf{u})I_{3\times3}\right)-\nu_{2}(\nabla\cdot
\mathbf{u})I_{3\times3},
\]
where $I_{3\times3}$ is the $3\times3$ identical matrix, $P=P\left(  \rho\right)  $ is the pressure, $\nu_{1}>0$ is the
shear viscosity and $\nu_{2}>0$ is the bulk viscosity. We consider the general
equation of states $P\left(  \rho\right)  \ $satisfying\\
(P1) $P\left(  s\right)  \in C^{1}\left(  0,\infty\right) ,\ P^{\prime}>0,$ and $ P(0)=0$.
 There exist $\gamma_1\in\left(  \frac{6}{5},2\right)  $ and $K_1>0$ such that
\begin{equation*}
\lim_{s\rightarrow0^+}s^{1-\gamma_1}P^{\prime}\left(  s\right)
=K_{1}>0.
\end{equation*}
It implies that
the pressure $P\left(  \rho\right)  \thickapprox K_1\rho^{\gamma_1}$ for
$\rho$ near  $0$.
When the viscosity is ignored, we get the Euler-Poisson system for inviscid
gaseous stars \cite{GR1998,HU2003,J2008,R2003,SS2004}
\begin{equation}
\left\{
\begin{array}
{ll}
\rho_{t}+\nabla\cdot(\rho\mathbf{u})=0 & \ \text{in}\ \Omega(t),\\
(\rho\mathbf{u})_{t}+\nabla\cdot(\rho\mathbf{u}\otimes\mathbf{u}%
)+\nabla P=-\rho\nabla V & \ \text{in}\ \Omega(t),\\
\Delta V=4\pi\rho & \ \text{in}\ \mathbb{R}^{3},\\
\lim_{\left\vert \mathbf{x}\right\vert \rightarrow\infty}V\left(
t,\mathbf{x}\right)  =0. &
\end{array}
\right.  \label{EP}%
\end{equation}

Both the Navier-Stokes-Poisson and the Euler-Poisson systems have the same
static equilibrium configurations with $\textbf{u}=\textbf{0}$ which is called the non-rotating stars. In \cite{BWL1981}, it is shown that the density function of a non-rotating star with compact support is spherically symmetric. By Lemma 3.1 in \cite{LZ2019}, there exists $\mu_{max}\in(0,+\infty]$ such that for any center density $\rho_\mu(0)=\mu\in(0,\mu_{max})$, there is a unique non-rotating star with the density $\rho_\mu(|\mathbf{x}|)$ and the radius of the support $R_\mu<\infty$ satisfying
\[
\nabla P(\rho_{\mu})+\rho_{\mu}\nabla V_{\mu}=0,\ |\mathbf{x}|<R_\mu,
\]
where $\Delta V_{\mu}=4\pi\rho_{\mu}$. In particular, by \cite{HU2003}, we know $\mu_{max}=+\infty$ when $\gamma_1\geq \frac{4}{3}$. For the polytropic case $P\left(  \rho\right)  =K\rho^{\gamma}$  $\left(  K>0,\ \gamma>\frac{6}{5}\right)  $, $\mu_{max}=+\infty$ and the non-rotating stars are called the Lane-Emden stars in the literature.

The stability of the non-rotating stars of the Euler-Poisson system has been well
studied. For the polytropic case $P\left(  \rho\right)  =K\rho^{\gamma}$, S.-S. Lin \cite{LSS} established the linear instability of the Lane-Emden stars when $\gamma\in\left(\frac{6}{5},\frac{4}{3}\right)$ and linear stability when $\gamma\in\left[\frac{4}{3},2\right)$. In \cite{J2008}, J. Jang proved the nonlinear instability of the Lane-Emden stars for $\gamma\in\left(\frac{6}{5},\frac{4}{3}\right)$. G. Rein \cite{R2003} construct the non-rotating stars of the Euler-Poisson system for $\gamma>\frac{4}{3}$ as minimizers of a suitably defined energy functional. The minimizing property implies the nonlinear stability of such states against general.  In particular, a turning point principle for the Euler-Poisson system was proved in \cite{LZ2019}
for  very general equations of states. That is, the transition of stability only
occurs at the extrema of the total mass, for the curve of non-rotating stars
parameterized by the center density. For further  works on the turning point principle for other models, we refer to \cite{HL,LW,LWZ}.

In this paper, we consider the stability and instability of the  non-rotating stars of the Navier-Stokes-Poisson system under spherically symmetric perturbations. Consider the following linearized Navier-Stokes-Poisson system under spherically symmetric perturbations
at a non-rotating star $\left(  \rho_{\mu}(|\mathbf{x}|),\mathbf{0}\right) $:
\begin{equation}
\left\{
\begin{array}
[c]{ll}
\rho_{t}+\frac{1}{r^{2}}(r^{2}\rho_{\mu}u)_{r}=0 &
\ \text{in}\ (0,R_\mu),\\
u_{t}+\left(\Phi^{\prime\prime}(\rho_{\mu})\rho\right)_{r}+
V_{r}-\frac{\nu}{\rho_{\mu}}\left(  \frac{1}{r^{2}}\left(r^{2}u\right)_{r}\right)_{r}  =0 & \ \text{in}\ (0,R_\mu),\\
V_{r}=\frac{4\pi}{r^{2}}\int_{0}^{r}\rho(s)s^{2}ds & \ \text{in}%
\ [0,\infty),
\end{array}
\right.  \label{L-NSP}%
\end{equation}
where $\nu=\frac{4}{3}\nu_{1}+\nu_{2},\ r=|\mathbf{x}|,\ \rho(|\mathbf{x}|,t)=\rho(r,t),\ \mathbf{u}(\mathbf{x},t)=u(r,t)\frac{\mathbf{x}}{r}$  and $\Phi(\rho)$ is the enthalpy function defined by
\[
\Phi^{\prime\prime}(\rho)=\frac{P^{\prime}(\rho)}{\rho},\quad\Phi
(0)=\Phi^{\prime}(0)=0.
\]
For comparison, the linearized Euler-Poisson system  under spherically symmetric perturbations
at $\left(  \rho_{\mu}(|\mathbf{x}|),\mathbf{0}\right) $ is
\begin{equation}
\left\{
\begin{array}
{ll}
\rho_{t}+\frac{1}{r^{2}}(r^{2}\rho_{\mu}u)_{r}=0 &
\ \text{in}\ (0,R_\mu),\\
u_{t}+\left(\Phi^{\prime\prime}(\rho_{\mu})\rho\right)_r+V_r=0 & \ \text{in}\ (0,R_\mu),\\
V_{r}=\frac{4\pi}{r^{2}}\int_{0}^{r}\rho(s)s^{2}ds & \ \text{in}%
\ [0,\infty).
\end{array}
\right.  \label{L-EP}%
\end{equation}


 Our main result is to give the turning point principle for the stability of the non-rotating
viscous stars with  very general equations of states and show the relation between  the stability criteria of viscous gaseous stars and inviscid gaseous stars.  Denote
\[
M_\mu  =\int_{\mathbb{R}^{3}}\rho_{\mu}dx=\int_{|x|\leq R_{\mu}  }\rho_{\mu}dx
\]
to be the total mass of the non-rotating  star. Let $n_{EP}^{u}\left(  \mu\right)  $ and
$n_{NSP}^{u}\left(  \mu\right)  $ be the number of unstable modes of the linearized
Euler-Poisson system \eqref{L-EP} and the linearized Navier-Stokes-Poisson system \eqref{L-NSP} respectively, namely the
total algebraic multiplicities of unstable eigenvalues.
The following theorem shows that the numbers of unstable modes for viscous and
inviscid cases are equal.

\begin{theorem}
\label{nEPeqnNSP} When $\frac{dM_\mu}{d\mu}  \neq0$, we have
$n_{EP}^{u}\left(  \mu\right)  $ $=n_{NSP}^{u}\left(  \mu\right)  .$
\end{theorem}

By the characterization of $n_{EP}^{u}(\mu)$ given in \cite{LZ2019}, we have
the following turning point principle of the Navier-Stokes-Poisson system.

\begin{corollary}\label{ttpforvis}
\label{ttp} The linear stability of the non-rotating star $\left(  \rho_{\mu}(|\mathbf{x}|),\mathbf{0}\right)  $ of the Navier-Stokes-Poisson system \eqref{NSP} is fully determined by the
mass-radius curve parametrized by $\mu$. Here, the mass-radius curve is
oriented in a coordinate plane where the horizontal and vertical axes
correspond to the support radius and mass of the star respectively. For
$\mu>0$ small enough, we have
\begin{equation}
n_{NSP}^{u}\left(  \mu\right)  =%
\begin{cases}
1 & \text{when }\gamma_{1}\in\left(  \frac{6}{5},\frac{4}{3}\right), \\
0 & \text{ when }\gamma_{1}\in\left(  \frac{4}{3},2\right).
\end{cases}
\label{formula-unstable-mode-small-mu}%
\end{equation}
The number $n_{NSP}^{u}\left(  \mu\right)  \ $can only change at the mass extrema
(i.e. maxima or minima of $M_\mu  $). With the increase of $\mu$, at a
mass extrema point, $n_{NSP}^{u}\left(  \mu\right)  $ increases by $1$ if
$\frac{dM_\mu}{d\mu}\frac{dR_\mu}{d\mu}  $ changes from $-$ to $+$ (i.e.
the mass-radius curve bends counterclockwise) and $n_{NSP}^{u}\left(
\mu\right)  $ decreases by $1$ if $\frac{dM_\mu}{d\mu}\frac{dR_\mu}{d\mu}$ changes from $+$ to $-\ $(i.e. the mass-radius curve bends clockwise).
\end{corollary}
\begin{remark}
In particular, for the polytropic case $P\left(  \rho\right)  =K\rho^{\gamma}$, we get
$n_{NSP}^{u}\left(  \mu\right)  =1$ when $\gamma\in\left(  \frac{6}{5}%
,\frac{4}{3}\right)  $ and $n_{NSP}^{u}\left(  \mu\right)  =0$ when $\gamma
\in\left(  \frac{4}{3},2\right) $. It shows that the non-rotating viscous stars are  linearly stable for $\gamma
\in\left(  \frac{4}{3},2\right) $ and linearly unstable for $\gamma\in\left(  \frac{6}{5}%
,\frac{4}{3}\right)$.
\end{remark}
Intuitively, one might expect that the viscosity is stabilizing and an
unstable non-rotating star could become stable when the viscosity is added.
The above conclusions show that this is not the case. However, the spectra
besides the unstable eigenvalues are very different for the Navier-Stokes-Poisson
and the Euler-Poisson systems. For the linearized Euler-Poisson system, the eigenvalues on the real axis are discrete and symmetric about the imaginary axis. The rest of the spectra are on the  imaginary axis.  For the linearized Navier-Stokes-Poisson system, the unstable eigenvalues are on the positive real axis. The rest of the spectra are on the left half plane. This suggests decaying behaviors
on the stable space of finite co-dimension.


Here, we discuss some important ideas in the proof of  Theorem \ref{nEPeqnNSP}. We rewrite the linearized
Navier-Stokes-Poisson system (\ref{L-NSP}) as a second-order linear PDE with dissipation:
\begin{equation}
u_{tt}+D_{\mu}u_{t}+L_{\mu}u=0, \label{LNSP-2nd order}%
\end{equation}
where the operators $D_{\mu},\ L_{\mu}$ are defined by (\ref{defn-D-mu}) and
(\ref{defn-L-mu}). The damping operator $D_{\mu}$ is positive and the
operator $L_{\mu}$ is self-adjoint with  finite number of the negative modes.
Let $n^{-}\left(  L_{\mu}\right)  $ be the number of negative modes of $L_{\mu}$.
The condition $\frac{dM_\mu}{d\mu}  \neq0$ implies that $\ker L_{\mu
}=\left\{  0\right\}  $. We note that $n^{-}\left(  L_{\mu}\right)  $ is
exactly the number of unstable modes of the linearized Euler-Poisson
system. That is $n^{-}\left(  L_{\mu}\right)  =n_{EP}^{u}\left(  \mu\right)
$. To prove Theorem \ref{nEPeqnNSP}, it suffices to show $n_{NSP}^{u}\left(
\mu\right)  =n^{-}\left(  L_{\mu}\right)  $.  In order to achieve this, we  study a class of
abstract second-order linear equations with dissipation
\begin{equation}
u_{tt}+Du_{t}+Lu=0, \label{L-2nd-order}%
\end{equation}
where the operators $D$ and $L$ satisfy assumptions (A1)-(A7) in Section \ref{s2}.
Roughly speaking, we assume that $D$ and $L$ are self-adjoint operators with
$D>0,\ker (L)=\left\{  0\right\},n^{-}\left(  L\right)  <\infty$ and a
certain compact embedding property of the space with the graph norm of $L$.
Under those assumptions, we can prove that $n^{u}=n^{-}(L)$, where $n^{u}$
denotes the total algebraic multiplicity of unstable eigenvalues of (\ref{L-2nd-order}). An unstable eigenvalue
$\lambda\in\mathbb{C}$ of (\ref{L-2nd-order}) is such that
$\operatorname{Re}\lambda>0$ and $\ker\left(  \lambda^{2}I+\lambda D+L\right)  \neq\left\{  0\right\}$. For the
finite-dimensional case (i.e. $D,\ L$ are matrices), this type of results is
usually called Kelvin-Tait-Chetaev Theorem in
the literature \cite{CSJ2002,KON2013,ZEE1964}. Results in the infinite-dimensional cases under various assumptions were also obtained
in
\cite{Miloslavskij1991,Pivovarchik1989}. However, these results cannot be applied to our current model. In Section \ref{s2}, we establish
a new infinite-dimensional Kelvin-Tait-Chetaev  Theorem (see Theorem
\ref{thm: KTC}) which is of independent interest. In Section \ref{s3}, we apply
Theorem \ref{thm: KTC} to
(\ref{LNSP-2nd order}) and obtain $n^{-}\left(  L_{\mu}\right)  =n_{NSP}%
^{u}\left(  \mu\right)  $.

In \cite{LTXZZH2016,LTXZZH20162}, the nonlinear asymptotic stability of the non-rotating viscous stars was
proved for the polytropic case $P\left(  \rho\right)  =K\rho^{\gamma}$ with
$\gamma\in\left(  \frac{4}{3},2\right)  $. Recently, in \cite{LTWYZH}, the
nonlinear asymptotic stability of the non-rotating stars was proved for a more
general polytrope $P\left(  \rho\right)  =K\rho^{\gamma}$ with $\gamma
>\frac{4}{3}$ and for a class of equations of states including the white dwarf stars. In
Theorem \ref{thm:asymptotic sta}, for a general equation of states satisfying
(P1) and (P2), we prove  nonlinear asymptotic
stability of the non-rotating viscous stars satisfying the sharp linear stability
condition $\frac{dM_\mu}{d\mu}  \neq0, n^{-}\left(  L_{\mu
}\right)  =0$. In particular, this implies the asymptotic stability of any
non-rotating white dwarf star  and non-rotating
star with center density up to the first mass maximum for a general equation of
states. Comparing with the equation of states in \cite{LTWYZH}, our assumptions are general. We only need more regularity of the pressure in addition and  no more assumptions for the center density of stars. The sharp linear stability condition in this paper is essentially different from the assumptions in \cite{LTWYZH}. Especially, the condition (1.9) in \cite{LTWYZH} is not obvious to verify.
For the proof, we adopt the ideas in \cite{LTXZZH20162}. However,
there is one significant difference. In \cite{LTWYZH} and  \cite{LTXZZH20162},
the pointwise behaviors of the equation of states are utilized to get the
positivity of leading order terms in the energy estimates. By contrast, we use
the positivity of the quadratic form $\langle L_{\mu
}u,u\rangle_{X_\mu}$, which is equivalent to the sharp stability condition, to control the higher order terms. For the
polytropic case $P\left(  \rho\right)  =K\rho^{\gamma}$ with $\gamma\in\left(
\frac{6}{5},\frac{4}{3}\right)  $, the nonlinear instability of the Lane-Emden
stars was shown in \cite{JJTI2013}. Similarly, we are able to prove the nonlinear
instability  of linearly unstable non-rotating viscous stars (i.e. $\frac{dM_\mu}{d\mu}  \neq0,n^{-}\left(  L_{\mu
}\right)  >0$) with a general equation of states. However,  the results in our paper show that  the number
of unstable modes is one for the polytropic case which is not proved in \cite{JJTI2013}. The turning
point principle is also true for the nonlinear stability of the non-rotating viscous  stars.

This paper is organized as follows. In Section \ref{s2}, we prove the
Kelvin-Tait-Chetaev Theorem for a class of abstract second-order equations with dissipation.
The work in Section \ref{s3} is to prove the turning point principle of  non-rotating
viscous stars by using the abstract theory in Section \ref{s2}. In Section
\ref{Nonlinearstability}, we obtain the nonlinear asymptotic stability of the
non-rotating viscous stars when $n^{-}\left(  L_{\mu
}\right)  =0$ and
$\frac{dM_\mu}{d\mu}\neq0$. In Section \ref{Nonlinearinstability}, the
nonlinear instability of the non-rotating viscous stars is proved when
$n^{-}\left(  L_{\mu
}\right)  >0$ and
$\frac{dM_\mu}{d\mu}\neq0$.

\textbf{Notations.}
We use $X \lesssim Y$ when $X \le CY$ for some constant $C> 0$. If the constant $C$ depends on $a,b,\ldots$, we set $X \lesssim_{a,b,\ldots} Y$.  $X \sim Y$ is used to denote $X \lesssim Y \lesssim X$. Further, $C,C',\ldots$ denote positive constants which may vary in different estimates.

\section{Kelvin-Tait-Chetaev Theorem in infinite dimension}\label{s2}

Before considering the linear stability of the gaseous stars modeled by the Navier-Stokes-Poisson system, we recall an old dynamical theorem first asserted by Kelvin and Tait and proved by Chetaev. The Kelvin-Tait-Chetaev Theorem is widely applied in the study of satellite attitude dynamics. We refer to \cite{ZEE1964} for details of the Kelvin-Tait-Chetaev Theorem and extensions in finite dimensions. In this section, we give the Kelvin-Tait-Chetaev Theorem for the infinite-dimensional case.
Consider the following equation:
\begin{align}\label{2.1}
u_{tt}+ D u_t+Lu=0.
\end{align}
We treat \eqref{2.1} as a first-order problem  on a Banach space $\mathfrak{X}$.
Let $v=u_t$ and  $\vec{u}=(u,v)^{T}$. Then, \eqref{2.1} is rewritten as
\begin{align}\label{2.2}
\vec{u}_t&=A\vec{u}(t),\\
A&=\left(
    \begin{array}{cc}
      0 & I \\
      -L & -D \\
    \end{array}
  \right)\nonumber
\end{align}
with initial data $\vec{u}|_{t=0}=(u_0,v_0)^T\in \mathfrak{X}$.

We assume that\\
(A1) $X$ is a Hilbert space with real inner product $\langle\cdot,\cdot\rangle_X$. Let $\mathfrak{X}=X^{D}_1\times X$, $\textrm{Dom}(L)\subset X$, $\textrm{Dom}(D)\subset X$, $\textrm{Dom}(A):=(\mathrm{Dom}(L)\cap \mathrm{Dom}(D))\times \mathrm{Dom}(D)$. Here, $X_1^D$ is the first Sobolev space with respect to the closed linear operator $D$:
\begin{align*}
X_1^D:=\textrm{Dom}(D),\|\cdot\|_{X_1^D}:=\|D(\cdot)\|_X.
\end{align*}
(A2) $D$ and $L$ are  closed and densely defined linear operators.\\
(A3) $D$ is a  positive definite operator, i.e., $\langle D u, u\rangle_X\geq C \langle  u, u\rangle_X$ for some constant $C>0$.\\
(A4) $L$ is a self-adjoint operator satisfying $\textrm{ker}(L)=\{0\}$. There exists a decomposition of $X$ such that $X=X_-\oplus X_+$ satisfying
\begin{align*}
n^{-}(L):=\textrm{dim}(X_-)<\infty,\ L|_{X_-\backslash \{0\}}<0,\ L|_{X_+\backslash \{0\}}\geq\delta>0,
\end{align*}
where $\delta$ is a positive constant.\\
(A5) $-D$ generates a strongly continuous semigroup on $X$.\\
(A6) $L$ is $D$-bounded. It implies that $\textrm{Dom}(D)\subseteq\textrm{Dom}(L)$ and $L\in\mathcal{L}(X^D_1,X)$. \\
(A7) There exists $m>0$ such that $\langle (L+m)u,u\rangle_X>0$, for all $0\neq u\in\textrm{Dom}(L)$. If the sequence $\{u_n\}_{n=1}^\infty$ satisfies $\langle (L+m)u_n,u_n\rangle_X\leq C$ for some positive constant $C$, then there exist $\bar{u}\in X$ and a subsequence of $\{u_n\}_{n=1}^\infty$ which is denoted by itself such that $u_n\rightarrow \bar{u}$ in $X$ as $n\rightarrow\infty$.
It means that the imbedding $Z\hookrightarrow X$ is compact where $Z$ is a Banach space with norm $\|\cdot\|_Z:=\langle (L+m)^\frac{1}{2}\cdot,(L+m)^\frac{1}{2}\cdot\rangle_X$.

By Corollary 3.4 of \cite{EN2000}, \eqref{2.2} is well-posed if and only if the closed
operator $A:\textrm{Dom}(A)=(\mathrm{Dom}(L)\cap \mathrm{Dom}(D))\times \mathrm{Dom}(D)\subset\mathfrak{X}\rightarrow\mathfrak{X} $ generates a strongly continuous semigroup $\{S(t)\}_{t\geq 0}$ on $\mathfrak{X}$. Moreover, if the initial value $(u_0,v_0)^T$ belonging to $\mathrm{Dom}(A)$ and $\vec{u}=(u,v)^T$ is the unique classical solution of \eqref{2.2}, then the first component $u$ of $\vec{u}$ is the unique classical solution of \eqref{2.1} with the initial data $u(0)=u_0$ and $u_t(0)=v_0$.  Therefore, by the assumptions (A5)-(A6) and Corollary 3.4  of \cite{EN2000}, we  have

\begin{lemma}\label{lem2.2}
Under the assumptions (A1), (A2), (A5) and (A6), the second-order Cauchy problem \eqref{2.1} with the initial data $u(0)=u_0$ and $u_t(0)=v_0$ is well-posed.
\end{lemma}

For $\tau\in[0,1],$ define
\begin{align*}
A_\tau=\left(
    \begin{array}{cc}
      0 & I \\
      -L & -\tau D \\
    \end{array}
  \right).
\end{align*}
\begin{lemma}\label{lem2.22}
For every $\tau\in [0,\infty)$, $\sigma_{ess}(A_\tau)$ (the essential spectrum of $A_\tau$) satisfies $\sigma_{ess}(A_\tau)\subset\{\lambda\in\mathbb{C}|\Re\lambda\leq0\}$. Moreover, any $\lambda_\tau\in \sigma(A_\tau)\cap \{\lambda\in\mathbb{C}|\Re\lambda > 0\}$ is a discrete spectrum satisfying
$\Im\lambda_{\tau}=0$ and there exists an constant $m_0>0$ such that $m_0<\lambda_\tau<\sqrt{m}$ on $\tau\in [0,\tau']$ for any $\tau'> 0.$
\end{lemma}
\begin{proof}
By the assumption (A4), define the projection $\Pi_+:X\mapsto X_+$.
Let $L_1=\Pi_+L\Pi_+$ and $L_2=L-L_1$. We have $\textrm{dim}(\textrm{Dom}(L_2))<\infty$ and $\textrm{dim}(\textrm{Ran}(L_2))<\infty$. Let $D_+=\Pi_+D\Pi_+$ and
\begin{align*}
A_+=\left(
    \begin{array}{cc}
      0 & \Pi_+ \\
      -L_1 & -\tau D_+ \\
    \end{array}
  \right).
\end{align*}
For any given $\Re\lambda>0$, there holds
\begin{align*}
&\left\langle (\lambda I-A_+)\begin{pmatrix} u\\v
\end{pmatrix},\begin{pmatrix} u\\v
\end{pmatrix}\right\rangle_{X\times X}\\
=&\langle\lambda u,u\rangle_X+\langle\lambda v,v\rangle_X+\langle\tau D_+ v_+,v_+\rangle_X-\langle v_+,u_+\rangle_X+\langle L_1u_+,v_+\rangle_X,
\end{align*}
where $u_+=\Pi_+u,v_+=\Pi_+v$.
Without loss of generality, we assume that $-\langle v_+,u_+\rangle_X+\langle L_1u_+,v_+\rangle_X\geq 0$. Otherwise, we consider
\begin{align*}
\left\langle \left(
    \begin{array}{cc}
      \frac{I}{C} &, 0 \\
      0 &, I \\
    \end{array}
  \right)(\lambda I-A_+)\begin{pmatrix} u\\v
\end{pmatrix},\begin{pmatrix} u\\v
\end{pmatrix}\right\rangle_{X\times X}
\end{align*}
for $C>0$ large enough. Thus, we have
\begin{align*}
&\left\| (\lambda I-A_+)\begin{pmatrix} u\\v
\end{pmatrix}\right\|_{X\times X}
\geq \Re\lambda \left\| \begin{pmatrix} u\\v
\end{pmatrix}\right\|_{X\times X}.
\end{align*}
It implies that $\lambda\in \rho(A_+)$, where $\rho(A_+)$ is the resolvent set of $A_+$.

Note that  $\textrm{dim}\left(\textrm{Dom}(A_\tau-A_+)\right)<\infty$ and $\textrm{dim}\left(\textrm{Ran}(A_\tau-A_+)\right)<\infty$, therefore, $\sigma_{ess}(A_\tau)=\sigma_{ess}(A_+)\subset\left\{\lambda\in\mathbb{C}|\Re\lambda\leq 0\right\}$.

By the assumption (A4), the unstable eigenvalues of $A_0$ are on the real axis and away from $0$.
 One can obtain that  the unstable eigenvalues $\lambda_\tau$ of $A_\tau$ with $\tau>0$ are on the real axis. Indeed,  let $u_{\tau}\neq 0$ be the first component of  an eigenfunction of $A_{\tau}$
with respect to the unstable eigenvalue $\lambda_{\tau}$ satisfying
\begin{align*}
  (A-\lambda_{\tau}I)\vec{u}_1=0.
\end{align*}
Note that
\begin{align*}
2\Re \lambda_{\tau}\Im \lambda_{\tau}\langle u_{\tau},u_{\tau}\rangle_X+\tau \Im \lambda_{\tau}\langle Du_{\tau},u_{\tau}\rangle_X=0.
\end{align*}
 Thus, it obtains that $\Im\lambda_{\tau}=0.$

Moreover, for any $\tau_1\in[0,\tau_0]$, we choose an sequence  $\{\tau_n\}_{n=1}^\infty$ such that $\tau_n>0,$ $\lim_{n\rightarrow\infty}\tau_n=\tau_1$ and $\lim_{n\rightarrow\infty}\lambda_{\tau_n}=0$, where $\lambda_{\tau_n}$ is the unstable eigenvalue of $A_{\tau_n}$. Let $u_{\tau_n}$ be the first component of an  eigenfunction of $A_{\tau_n}$ with respect to the unstable eigenvalue $\lambda_{\tau_n}$ with  $\|u_{\tau_n}\|_X=1$. There holds
\begin{align*}
(\lambda^2_{\tau_n}+m)\langle u_{\tau_n},u_{\tau_n}\rangle_X+\tau_n\lambda_{\tau_n}\langle Du_{\tau_n},u_{\tau_n}\rangle_X+\langle (L+m)u_{\tau_n},u_{\tau_n}\rangle_X=2m.
\end{align*}
Hence, $\lambda^2_{\tau_n}+m\leq 2m$ by the assumption (A3) and (A7). Furthermore, by $\langle (L+m)u_{\tau_n},u_{\tau_n}\rangle_X\leq 2m$ and the assumption (A7), there exists $\bar{u}\neq0$  such that
\begin{align*}
 u_{\tau_n}\rightarrow\bar{u}\ \textrm{in}\ X,\ \textrm{as}\ n\rightarrow \infty.
\end{align*}
Note that
\begin{align}\label{3}
\lambda^2_{\tau_n}\langle u_{\tau_n},\phi\rangle_X+\tau_n\lambda_{\tau_n}\langle Du_{\tau_n},\phi\rangle_X+\langle Lu_{\tau_n},\phi\rangle_X=0,\ \forall \phi\in \text{Dom}(D).
\end{align}
Taking $n\rightarrow\infty$ in \eqref{3}, we have $\langle \bar{u},L\phi\rangle_X=0$ which yields $\bar{u}\in \overline{\text{Ran}(L)}^\perp=\text{ker}(L)$.
It is a contradiction to $\textrm{ker}(L)=\{0\}$. It obtains that the unstable eigenvalues of $A_{\tau_n}$ can't tend to zero as $n\rightarrow \infty$.
\end{proof}

Let  $n^u$ be the total algebraic multiplicity of eigenvalues of $A$ on the right half plane, that is, the dimension of the  unstable
space of \eqref{2.1}. First, we have

\begin{lemma}\label{lem2.3}
$n^u\leq n^-(L)$.
\end{lemma}
\begin{proof}
Let $\lambda_i(i=1,\ldots,M)$  be the distinct eigenvalues of $A$ on the right half plane with algebraic multiplicity $n_i<\infty$ and $\sum_{i=1}^M n_i=n^u.$

For $1\leq i\leq M$, let $\vec{h}_{ij},$ $j=1,\ldots,n_i$ be linearly independent functions such that
\begin{align*}
(A-\lambda_i I)^{n_i}\vec{h}_{ij}=0.
\end{align*}
Hence, the Cauchy problem \eqref{2.2} with the initial data $\vec{u}(0)=\vec{h}_{ij}$ has the solution
\begin{align}\label{ex}
\vec{u}_{ij}(t)=\vec{P}_{ij}(t)e^{\lambda_it},
\end{align}
where $\vec{P}_{ij}(t)$ is the polynomial of $t$ with the degree less than $n_i$ and $\vec{P}_{ij}(0)=\vec{h}_{ij}$.
Since  $\vec{h}_{ij},i=1,\ldots,M,j=1,\ldots,n_i$ are linearly independent,  we have that the first components $u_{ij}(t):=P_{ij}(t)e^{\lambda_it}$ of  $\vec{u}_{ij}(t)$, which are the solutions of \eqref{2.1}, are linearly independent.

Then, by \eqref{2.1}, we have
\begin{align*}
\frac{1}{2}\frac{d}{dt}\left\langle \frac{du_{ij}}{dt},\frac{du_{ij}}{dt}\right\rangle_X+\left\langle D \frac{du_{ij}}{dt},\frac{du_{ij}}{dt}\right\rangle_X+\frac{1}{2}\frac{d}{dt}\left\langle Lu_{ij},u_{ij}\right\rangle_X =0.
\end{align*}
By the assumption (A3), there holds
\begin{align}\label{2.12}
\frac{1}{2}\frac{d}{dt}\left(\left\langle \frac{du_{ij}}{dt},\frac{du_{ij}}{dt}\right\rangle_X+\langle Lu_{ij},u_{ij}\rangle_X\right) =-\left\langle D \frac{du_{ij}}{dt},\frac{du_{ij}}{dt}\right\rangle_X<0.
\end{align}
By $\Re\lambda_i>0$ and \eqref{ex}, we easily have
\begin{align*}
\left\langle \frac{du_{ij}}{dt},\frac{du_{ij}}{dt}\right\rangle_X+\langle Lu_{ij},u_{ij}\rangle_X\rightarrow 0,\ \text{as}\ t\rightarrow -\infty.
\end{align*}
Integrating \eqref{2.12} from $-\infty$ to $0$, we obtain
\begin{align}\label{contr}
\frac{1}{2}\left(\left\langle \frac{du_{ij}}{dt}(0),\frac{du_{ij}}{dt}(0)\right\rangle_X+\langle Lu_{ij}(0),u_{ij}(0)\rangle_X\right) =-\int_{-\infty}^{0}\left\langle D \frac{du_{ij}}{dt},\frac{du_{ij}}{dt}\right\rangle_X dt<0.
\end{align}
This implies that
\begin{align*}
\langle Lu_{ij}(0),u_{ij}(0)\rangle_X <0.
\end{align*}
It should be noted that $u_{ij}(0),i=1,\ldots,M,j=1,\ldots,n_i$ are not linearly dependent. Indeed, if there exists some constants $C_{ij}$ such that $\sum_{i=1}^M\sum_{j=1}^{n_i}C_{ij}u_{ij}(0)=0$, we can replace $u_{ij}$ by $\sum_{i=1}^M\sum_{j=1}^{n_i}C_{ij}u_{ij}$ in the above argument. By \eqref{contr}, it is a contradiction to  $u_{ij},i=1,\ldots,M,j=1,\ldots,n_i$  are linearly independent.  Then, we have
\begin{align*}
L\big|_{\textrm{span}\{u_{ij}(0),i=1,\ldots,M,j=1,\ldots,n_i\}\setminus\{0\}}<0,
\end{align*}
which completes the proof.
\end{proof}

We use a homotopy method to prove $n^u\geq n^-(L)$ which implies that $n^u= n^-(L)$. 
Denote the resolvent of $A_\tau$ by
\begin{align*}
R_{\lambda,\tau}:=(A_\tau-\lambda I)^{-1},\ \lambda\in\rho(A_\tau),
\end{align*}
where $\rho(A_\tau)$ is the resolvent set of $A_\tau$. Let $R_{\lambda,\tau}^*$ be the conjugate operator of $R_{\lambda,\tau}$.

\begin{lemma}\label{lem2.4}
$R_{\lambda,\tau}$ and $R_{\lambda,\tau}^*$ are strongly continuous with respect to $\tau\in [0,1]$.
\end{lemma}
\begin{proof}
For $\tau_0\in [0,1]$, let $\{\tau_n\}_{n=1}^\infty\subset [0,1]$ be any sequence satisfying $\lim_{n\rightarrow\infty}\tau_n=\tau_0$.
For any fixed $(u,v)^T\in \textrm{Dom}(A)\subset\mathfrak{X}$, there holds
\begin{align*}
(A_{\tau_n}-A_{\tau_0})\begin{pmatrix} u\\
v
\end{pmatrix}=
\begin{pmatrix} 0, &0,\\
0,&-(\tau_n-\tau_0) D
\end{pmatrix}\begin{pmatrix} u\\
v
\end{pmatrix}.
\end{align*}
We obtain that $A_{\tau_n}\rightarrow A_{\tau_0}$ strongly as $n\rightarrow\infty$. Also, it  obtains that  $A^*_{\tau_n}\rightarrow A^*_{\tau_0}$ strongly  as $n\rightarrow\infty$.

Let $\lambda_{\tau_0}>0$ be an unstable eigenvalue of $A_{\tau_0}$. There exists $r>0$ small enough such that the ball $B(\lambda_{\tau_0},r)\subset\rho(A_{\tau_0})\cap\{x\in\mathbb{C}|\Re x>0\}\cup\{\lambda_{\tau_0}\}$  contains no other spectrum expect $\lambda_{\tau_0}$. We claim that
 there exists $N>0$ such that $A_{\tau_n}-\lambda I$ is invertible for all $n>N$ and   $\lambda\in  \partial B(\lambda_{\tau_0},r)$.
In fact, we need to prove that $\lambda_{\tau_0}\pm r$ aren't eigenvalues of $A_{\tau_n}$ for $n>N$.

By contradiction, assume that there exists an subsequence of $\{\tau_n\}_{n=N+1}^\infty$ which is denoted by itself such that  $ u_{\tau_n}$ is the first component of an eigenfunction of $A_{\tau_n}$ with respect to $\lambda_{\tau_0}+ r$ satisfying $\| u_{\tau_n}\|_X=1$ and
\begin{align}\label{e2}
  (\lambda_{\tau_0}+ r)^2\langle u_{\tau_n},\phi\rangle_X+\tau_n(\lambda_{\tau_0}+ r)\langle Du_{\tau_n},\phi\rangle_X+\langle Lu_{\tau_n},\phi\rangle_X=0,\ \forall\phi\in \text{Dom}(D).
\end{align}
Choosing $\phi=u_{\tau_n}$, we have $\langle (L+m)u_{\tau_n},u_{\tau_n}\rangle_X\leq m$. By the assumption (A7), there exists $\bar{u}\neq0$  such that
\begin{align*}
 u_{\tau_n}\rightarrow\bar{u}\ \textrm{in}\ X,\ \textrm{as}\ n\rightarrow \infty.
\end{align*}
Taking $n\rightarrow\infty$ in \eqref{e2}, we obtain
\begin{align*}
   (\lambda_{\tau_0}+ r)^2\langle \bar{u},\phi\rangle_X+\tau_0(\lambda_{\tau_0}+ r)\langle \bar{u},D\phi\rangle_X+\langle\bar{ u},L\phi\rangle_X=0,\ \forall\phi\in \text{Dom}(D),
\end{align*}
which shows that $\bar{u}\in\overline{\text{Ran}((\lambda_{\tau_0}+ r)^2I+\tau_0(\lambda_{\tau_0}+ r)D+L)}^\perp=\text{ker}((\lambda_{\tau_0}+ r)^2I+\tau_0(\lambda_{\tau_0}+ r)D+L)$.
Since $\lambda_{\tau_0}+ r$ is not an eigenvalue of $A_{\tau_0}$, it is a contradiction. The same argument holds for the case $\lambda_{\tau_0}-r$.

We separate the proof into two cases.\\
\textbf{Case 1}: $\tau_n,\tau_0\in (0,1]$, $\lim_{n\rightarrow\infty}\tau_n=\tau_0$.

Note that there exists some constant $C>0$ independent of $n$ such that
\begin{align*}
\|R_{\lambda,\tau_0}\|,\|R_{\lambda,\tau_n}\|\leq C,\ \forall\tau_n,\tau_0\in(0,1],\ \lambda\in \partial B(\lambda_{\tau_0},r).
\end{align*}
Hence, for any fixed $(f,g)^T$, there holds
\begin{align*}
\left\|(R_{\lambda,\tau_n}-R_{\lambda,\tau_0})\begin{pmatrix} f\\g
\end{pmatrix}\right\|^2_{\mathfrak{X}}&=\left\|R_{\lambda,\tau_n}(A_{\tau_0}-A_{\tau_n})R_{\lambda,\tau_0}\begin{pmatrix} f\\g
\end{pmatrix}\right\|^2_{\mathfrak{X}}\lesssim\left\|(A_{\tau_0}-A_{\tau_n})R_{\lambda,\tau_0}\begin{pmatrix} f\\g
\end{pmatrix}\right\|^2_{\mathfrak{X}}.
\end{align*}
Since $A_{\tau_n}\rightarrow A_{\tau_0}$ strongly as $n\rightarrow\infty$,
 we obtain
\begin{align*}
\left\|(R_{\lambda,\tau_n}-R_{\lambda,\tau_0})\begin{pmatrix} f\\g
\end{pmatrix}\right\|^2_{\mathfrak{X}}\rightarrow0,\ \textrm{as}\ n\rightarrow\infty.
\end{align*}
It implies that $R_{\lambda,\tau}$ is strongly continuous at $\tau=\tau_0$.
\\
\textbf{Case 2}:  $\lim_{n\rightarrow\infty}\tau_n=\tau_0=0$.

Let $R_{\lambda,0}=(A_{0}-\lambda I)^{-1}$ and $A_0=\left(
    \begin{array}{cc}
      0 & I \\
      -L & 0 \\
    \end{array}
  \right).$ For any fixed $(f,g)^T$, define
\begin{align}\label{2.26}
&\begin{pmatrix} u_1\\u_2
\end{pmatrix}=R_{\lambda,0}\begin{pmatrix} f\\g
\end{pmatrix},\\
&\begin{pmatrix} v_{1,n}\\v_{2,n}
\end{pmatrix}=R_{\lambda,\tau_n}\begin{pmatrix} f\\g
\end{pmatrix},\label{2.27}
\end{align}
where $\lambda\in \partial B(\lambda_{0},r)$ and $\lambda_0>0$ is the unstable eigenvalue of $A_0$.

To obtain the result, it suffices to prove
\begin{align}\label{2.28}
\lim_{n\rightarrow \infty}\left\|\begin{pmatrix} v_{1,n}\\v_{2,n}
\end{pmatrix}-\begin{pmatrix}u_1\\u_2
\end{pmatrix}\right\|_{X\times X}=0.
\end{align}
By \eqref{2.26} and \eqref{2.27}, we have
\begin{align*}
&-\lambda u_1+u_2=f,\ -Lu_1-\lambda u_2=g,\\
&-\lambda v_{1,n}+v_{2,n}=f,\ -Lv_{1,n}-(\lambda+\tau_n D)v_{2,n}=g.
\end{align*}
That is
\begin{align}\label{2.33}
&(\lambda^2+\tau_n\lambda D+L)v_{1,n}=-g-(\lambda+\tau_n D)f,\\
&(\lambda^2+L)u_1=-g-\lambda f.\label{2.34}
\end{align}
If we prove
\begin{align*}
\lim_{n\rightarrow\infty}\left\|v_{1,n}-u_1\right\|_X=0,
\end{align*}
then \eqref{2.28} holds.

Here, we claim that
\begin{align}\label{cl1}
\left(\lambda^2+\tau_n\lambda D+L\right)^{-1}\ \text{is bounded for all}\ \tau_n\in [0,1],\ \lambda\in \partial B(\lambda_0,r).
\end{align}

First, by the assumption (A7), there exists $m>0$ such that $L+m>0$. We consider the operator $\lambda^2+\tau\lambda D+L+2m$ for $\tau\in[0,1]$ and $\Re\lambda>0$. For every $f_n$ with $\sup_{n\in\mathbb{Z}^+}\|f_n\|_X\lesssim 1$, let $u_n$ satisfy
\begin{align*}
(\lambda^2+\tau\lambda D+L+2m)u_n=f_n.
\end{align*}
Hence, we have
\begin{align*}
\left(\Re(\lambda^2)+m\right)\langle u_n,u_n\rangle_X+\tau\Re\lambda\langle Du_n,u_n\rangle_X+\langle (L+m)u_n,u_n\rangle_X=\langle f_n,u_n\rangle_X.
\end{align*}
Choosing $m>0$ large enough such that $\Re(\lambda^2)+m>0$,  we have
\begin{align*}
\langle (L+m)u_n,u_n\rangle_X\lesssim\|f_n\|_X^2.
\end{align*}
By the assumption (A7), there exists $\bar{u}\in X$ and a subsequence of $\{u_n\}_{n=1}^\infty$ which is denoted by itself such that $u_n\rightarrow\bar{u}$ in $X$. Therefore, $\left(\lambda^2+\tau\lambda D+L+2m\right)^{-1}: X\rightarrow X$ is compact. We obtain the spectra of $\lambda^2+\tau\lambda D+L$ are eigenvalues for $\tau\in[0,1]$ and $\Re\lambda>0$.

Next, we prove \eqref{cl1} by contradiction. If for $\epsilon_n\rightarrow0$, there exists a sequence $\{u_n\}_{n=1}^\infty$ with $\|u_n\|_X=1$ such that
\begin{align}\label{e1}
(\lambda^2+\tau_n\lambda D+L)u_n=\epsilon_n u_n,
\end{align}
then we have
\begin{align*}
(\Re(\lambda^2)-\epsilon_n)+\tau_n\Re\lambda\langle Du_n,u_n\rangle_X+\langle Lu_n,u_n\rangle_X=0.
\end{align*}

Choose $m>0$ large enough such that
\begin{align*}
\Re(\lambda^2)-\epsilon_n+m>0,\ \langle (L+m)u_n,u_n\rangle_X>0.
\end{align*}
Note that
\begin{align*}
\tau_n\Re\lambda\langle Du_n,u_n\rangle_X\geq 0.
\end{align*}
Then, from
\begin{align*}
(\Re(\lambda^2)-\epsilon_n+m)+\tau_n\Re\lambda\langle Du_n,u_n\rangle_X+\langle (L+m)u_n,u_n\rangle_X=2m,
\end{align*}
we have
\begin{align}
\notag&\langle (L+m)u_n,u_n\rangle_X\leq 2m,\\
&\tau_n\Re\lambda\langle Du_n,u_n\rangle_X\leq 2m.\label{2.46}
\end{align}
By the assumption (A7),  there exist $\bar{u}\in Z$ and a subsequence of $\{u_n\}_{n=1}^\infty$ which is denoted by itself such that
\begin{align*}
u_n\rightharpoonup\bar{u}\ \textrm{in}\ Z,\ u_n\rightarrow\bar{u}\ \textrm{in}\ X,\ \textrm{as}\ n\rightarrow \infty.
\end{align*}

By \eqref{e1}, there holds
\begin{align}\label{2.49}
\langle \lambda^2 u_n,\phi\rangle_X+\tau_n\langle \lambda Du_n,\phi\rangle_X+\langle Lu_n,\phi\rangle_X=\epsilon_n\langle u_n,\phi\rangle_X,\ \forall \phi\in \text{Dom}(D).
\end{align}
By \eqref{2.46} and $\lim_{n\rightarrow\infty}\tau_n=0$, we have
\begin{align*}
\tau_n\langle \lambda Du_n,\phi\rangle_X\leq \tau_n|\lambda|\|D^\frac{1}{2}u_n\|_X\|D^\frac{1}{2}\phi\|_X\rightarrow 0,\ \textrm{as}\ n\rightarrow\infty.
\end{align*}
Therefore, by taking $n\rightarrow\infty$ in \eqref{2.49}, we obtain
\begin{align*}
\langle \lambda^2\bar{u},\phi\rangle_X+\langle \bar{u},L\phi\rangle_X=0,\ \forall \phi\in \text{Dom}(D).
\end{align*}
Since $\lambda^2+L$ is invertible, we have $\bar{u}=0$ which is a contradiction.

By \eqref{2.33}, we have
\begin{align*}
v_{1,n}=\left(\lambda^2+\tau_n\lambda D+L\right)^{-1}\left(-g-(\lambda+\tau_n D)f\right).
\end{align*}
Hence, by \eqref{cl1}, there holds
\begin{align}\label{2.53}
\|v_{1,n}\|_X\lesssim_{\|g\|_X,\|f\|_X,\|Df\|_X}1.
\end{align}
There exists $v_1\in X$ and  a subsequence of $\{v_{1,n}\}_{n=1}^\infty$ which is denoted by itself such that $v_{1,n}\rightharpoonup v_1$ in $X$ as $n\rightarrow\infty$.
Note that
\begin{align*}
\langle(\lambda^2+\tau_n\lambda D+L)v_{1,n},\phi\rangle_X=\langle -g-(\lambda+\tau_n D)f,\phi\rangle_X,\ \forall\phi\in \text{Dom}(D).
\end{align*}
Taking $n\rightarrow\infty$, we have
\begin{align*}
\langle v_1,(\overline{\lambda^2}+L)\phi\rangle_X=\langle -g-\lambda f,\phi\rangle_X,\ \forall\phi\in \text{Dom}(D).
\end{align*}
By \eqref{2.34} and  $-\lambda^2\in\rho(L)$, it obtains that $v_1=u_1$.

Moreover, by
\begin{align*}
\left(\lambda^2+\tau_n\lambda D+L+2m\right)v_{1,n}=-g-\left(\lambda+\tau_n D\right)f+2mv_{1,n},
\end{align*}
we have
\begin{align*}
&(\Re(\lambda^2)+m)\langle v_{1,n},v_{1,n}\rangle_X+\tau_n\Re\lambda\langle Dv_{1,n},v_{1,n}\rangle_X+\langle (L+m)v_{1,n},v_{1,n}\rangle_X\\
=&\langle -g-(\lambda+\tau_n D)f,v_{1,n}\rangle_X+2m\langle v_{1,n},v_{1,n}\rangle_X.
\end{align*}
Choosing $m$ large enough such that $\langle (L+m)v_{1,n},v_{1,n}\rangle_X>0$ and $\Re(\lambda^2)+m>0$, we have
\begin{align*}
\langle (L+m)v_{1,n},v_{1,n}\rangle_X\lesssim\|-g-(\lambda+\tau_n D)f\|_X\|v_{1,n}\|_X+2m\|v_{1,n}\|^2_X.
\end{align*}
By \eqref{2.53}, there holds
\begin{align*}
\langle (L+m)v_{1,n},v_{1,n}\rangle_X\lesssim_{m,\|g\|_X,\|f\|_X,\|Df\|_X}1.
\end{align*}
Hence, by the assumption (A7), we obtain
\begin{align*}
v_{1,n}\rightarrow v_1\ \textrm{in}\ X,\ \textrm{as}\ n\rightarrow \infty.
\end{align*}
It implies that $v_{1,n}\rightarrow u_1$ in $X$. It completes the proof.
\end{proof}

\begin{theorem}\label{thm: KTC}
 The number of unstable modes
of the  operator $A$ equals to the number of negative modes of
the operator $L$, i.e., $n^{u}=n^{-}(L)$.
\end{theorem}
\begin{proof}
If $n^{-}(L)=0$, then by Lemma \ref{lem2.3}, we know $n^{u}=n^{-}(L)=0$.

If $n^{-}(L)>0$, we divide the proof into three steps.
Let $$\Lambda:=\{\tau\in [0,+\infty)| n^u(A_\tau)=n^-(L) \}$$
where $n^u(A_\tau)$ is the total algebraic multiplicity of the unstable eigenvalues of  $A_\tau$.
We want to prove  $\Lambda=[0,+\infty)$.

\textbf{(Step 1)}.
It is clear that $0\in \Lambda$. By Lemma \ref{lem2.22}, the unstable eigenvalues of $A_\tau$ are real and away from $0$ when $\tau\in[0,\tau']$ for any $\tau'>0$. 

Define the Riesz projection operator:
\begin{align*}
&P_{0}=\oint_{\Gamma} R_{\lambda,0}d\lambda,
\end{align*}
where $\Gamma=\partial B(\lambda_{0},r_{0})$, $\lambda_{0}$ is the eigenvalue of $A_0$ on the right half plane and $r_{0}>0$ is small enough such that the ball $B(\lambda_{0},r_{0})\subset \rho(A_0)\cap\left\{x\in\mathbb{C}|\Re x>0\right\}\cup\{\lambda_0\}$ contains no other spectrum expect $\lambda_{0}$. The integral curve is in the counterclockwise sense.

For $\tau>0$ near to $0$, there exists the resolvent $R_{\lambda,\tau}$ for $\lambda\in \Gamma$. Define
\begin{align*}
P_{\tau}=\oint_{\Gamma} R_{\lambda,\tau}d\lambda.
\end{align*}
By Lemma \ref{lem2.4}, we have $P_{\tau}\rightarrow P_{0}$ and $P^*_{\tau}\rightarrow P^*_{0}$ strongly as $\tau\rightarrow 0^+$.
Therefore,  $\text{dim}(\textrm{Ran}(P_{\tau}))\geq\text{dim}(\textrm{Ran}(P_{0}))$.
It obtains that there exists $\tilde{\tau}>0$ such that  $n^u(A_{\tau})\geq n^u(A_0)$  for $\tau\in[0,\tilde{\tau})$. On the other hand, by the proof of Lemma \ref{lem2.3},
we have $n^u(A_{\tau})\leq n^-(L)=n^u(A_0)$. Therefore, $n^u(A_{\tau})=n^-(L)$ for every $\tau\in[0,\tilde{\tau})$.
This implies that $\Lambda$ is a relative open set at point $0$.

\textbf{(Step 2)}.
For every $\tau_0\in \Lambda$, we consider $\tau\in (\tau_0-\delta,\tau_0+\delta)$ where  $\delta>0$ is small enough. By the proof of Lemma \ref{lem2.4}, there exists the resolvent $R_{\lambda,\tau}$ for $\lambda\in \Gamma=\partial B(\lambda_{\tau_0},r_{0})$ where $\lambda_{\tau_0}$ is the eigenvalue of $A_{\tau_0}$ on the right half plane and $r_{0}>0$ is small enough such that the ball $B(\lambda_{\tau_0},r_{0})\subset \rho(A_{\tau_0})\cap\left\{x\in\mathbb{C}|\Re x>0\right\}\cup\{\lambda_{\tau_0}\}$ contains no other spectrum expect $\lambda_{\tau_0}$. Moreover, we have $P_{\tau}\rightarrow P_{\tau_0}$ strongly as $\tau\rightarrow \tau_0$.
Therefore,  $\text{dim}(\textrm{Ran}(P_{\tau}))\geq\text{dim}(\textrm{Ran}(P_{\tau_0}))$.
It obtains that $n^u(A_{\tau})\geq n^u(A_{\tau_0})$  for $\tau\in (\tau_0-\delta,\tau_0+\delta)$. On the other hand, by the proof of Lemma \ref{lem2.3} and $\tau_0\in \Lambda$,
we have $n^u(A_{\tau})\leq n^-(L)=n^u(A_{\tau_0})$. Therefore, $n^u(A_{\tau})=n^-(L)$ for $\tau\in (\tau_0-\delta,\tau_0+\delta)$.
This implies that for every $\tau_0\in \Lambda$, $\tau_0$ is a inner point of $\Lambda$.

\textbf{(Step 3)}.
Now, we will prove $\Lambda$ is closed. That is if there exists $\{\tau_n\}_{n=1}^\infty\subset\Lambda$ such that  $\lim_{n\rightarrow\infty}\tau_n=\tau_0$, then $\tau_0\in \Lambda$. 

Denote
$$[\cdot,\cdot]=\langle L\cdot,\cdot\rangle_X.$$
Choose  $\{w_i\}_{i=1}^{n^-(L)}\subset \text{Dom}(D)\subset X$ such that
\begin{align}
X_-=\text{span} \left\{w_{1},w_{2},...,w_{n^-(L)}\right\},\quad X_{+}=\left\{\psi\in X|\langle L\psi,\phi\rangle=0,\forall \phi\in X_-\right\}
\end{align}
and $[w_i,w_j]=-\delta_{ij}$ for $1\leq i,j\leq n^-(L)$, where $\delta_{ij}$ is Kronecker symbol.

We denote the unstable eigenspace of $A_{\tau_n}$ by $\mathfrak{X}_-^{(n)}\subset \mathfrak{X}$.
Since $\{\tau_n\}_{n=1}^\infty\subset \Lambda$, we have $n^u(A_{\tau_n})=n^-(L)$. Let $W^{(n)}:=$\{the first component of $\mathfrak{X}_-^{(n)}$\}. Denote $\Pi_\pm: X\rightarrow X_\pm$ to be the projection operators
with $\text{ker}(\Pi_\pm) = X_\mp$. By \eqref{contr} and $n^u(A_{\tau_n})=n^-(L)$, we have $\text{dim}(W^{(n)})=n^-(L)$ and $\Pi_-(W^{(n)})=X_-$.
Then we can choose a basis $\{\vec{\xi}_{\tau_n,1},\ldots,\vec{\xi}_{\tau_n,n^-(L)}\}$ of $\mathfrak{X}_-^{(n)}$
such that $\vec{\xi}_{\tau_n,i}\in\text{Dom}(A)$, $\zeta_{\tau_n,i}=w_i+w^{(n)}_{i+}$ for $i=1,2,...,n^-(L)$ where $\zeta_{\tau_n,i}$ is the first component of $\vec{\xi}_{\tau_n,i}$ and
$w^{(n)}_{i+}\in X_+$.


By the same computation as \eqref{contr}, we have for $i=1,\ldots,n^-(L)$,
\begin{align}
0>\langle L \zeta_{\tau_n,i},\zeta_{\tau_n,i}\rangle_X =\langle L w_i,w_i\rangle_X+\langle L w^{(n)}_{i+},w^{(n)}_{i+}\rangle_X\geq -1+\delta\|w^{(n)}_{i+}\|_X.
\end{align}
So $\|\zeta_{\tau_n,i}\|_X\leq C$ for some positive constant $C$ independent of $n$.
Hence, subject to a subsequence, we have $\zeta_{\tau_n,i}\rightharpoonup \zeta_{\tau_0,i}$  weakly in  $X$
and $\Pi_-(\zeta_{\tau_0,i}) =w_i$.

By using \eqref{contr} and $-1\leq\langle L \zeta_{\tau_n,i},\zeta_{\tau_n,i}\rangle_X<0$, we have
$\tilde{\zeta}_{\tau_n,i}\rightharpoonup \tilde{\zeta}_{\tau_0,i}$ weakly in $X$ for $i=1,\ldots,n^-(L)$ where $\tilde{\zeta}_{\tau_n,i}$ are the second component of $\vec{\xi}_{\tau_n,i}$.
Therefore, $\vec{\xi}_{\tau_n,i}\rightharpoonup \vec{\xi}_{\tau_0,i}$ weakly in $X\times X$ for $i=1,\ldots,n^-(L)$, where $\vec{\xi}_{\tau_0,i}=(\zeta_{\tau_0,i},\tilde{\zeta}_{\tau_0,i})^T$.
This implies that $\{\vec{\xi}_{\tau_0,1},\ldots,\vec{\xi}_{\tau_0,n^-(L)}\}$ generates a $n^-(L)$-dimension space.

For every $j=1,\ldots,n^{-}(L)$, there exist $a^{(i,j)}_{\tau_n},i=1,\ldots,n^{-}(L)$ such that
\begin{align}\label{A}
A_{\tau_n}\vec{\xi}_{\tau_n,j}=\sum_{i=1}^{n^-(L)}a^{(i,j)}_{\tau_n}\vec{\xi}_{\tau_n,i}.
\end{align}
We claim that $a^{(i,j)}_{\tau_n}$ are uniformly bounded for $n\in\mathbb{Z^+}, i,j=1,\ldots,n^-(L)$. Suppose otherwise, there exists a subsequence $\{n_k\}_{k=1}^\infty$ such that
\begin{align*}
|a^{(i_0,j_0)}_{\tau_{n_k}}|=\max_{i,j=1,\ldots,n^-(L)} |a^{(i,j)}_{\tau_{n_k}}|\rightarrow\infty,\ \text{and}\ \forall  i,j=1,\ldots,n^-(L),\ c^{(i,j)}=\lim_{k\rightarrow\infty}\frac{a^{(i,j)}_{\tau_{n_k}}}{a^{(i_0,j_0)}_{\tau_{n_k}}}\ \text{exists}.
\end{align*}
Then, by
\begin{align*}
\frac{1}{a^{(i_0,j_0)}_{\tau_{n_k}}}\langle\vec{\xi}_{\tau_{n_k},j_0},A^*_{\tau_{n_k}}\vec{\phi}\rangle_{X\times X}=&\frac{1}{a^{(i_0,j_0)}_{\tau_{n_k}}}\langle A_{\tau_{n_k}}\vec{\xi}_{\tau_{n_k},j_0},\vec{\phi}\rangle_{X\times X}\\
=&\left\langle \sum_{i=1}^{n^-(L)}\frac{a^{(i,j_0)}_{\tau_{n_k}}}{a^{(i_0,j_0)}_{\tau_{n_k}}}\vec{\xi}_{\tau_{n_k},i},\vec{\phi}\right\rangle_{X\times X},\ \forall \vec{\phi}\in\text{Dom}(A),
\end{align*}
and taking $k\rightarrow\infty$, we have
\begin{align*}
0=\sum_{i=1}^{n^-(L)}\langle c^{(i,j_0)}\vec{\xi}_{\tau_{0},i},\vec{\phi}\rangle_{X\times X},\ \forall \vec{\phi}\in\text{Dom}(A).
\end{align*}
By a density argument, it concludes that $\sum_{i=1}^{n^-(L)}c^{(i,j_0)}\vec{\xi}_{\tau_{0},i}=0.$ Note that $c^{(i_0,j_0)}=1$.  This is in contradiction to the independency of $\{\vec{\xi}_{\tau_0,1},\ldots,\vec{\xi}_{\tau_0,n^-(L)}\}$.

Subject to a subsequence, there exist $a^{(i,j)}_{\tau_0}$ such that $a^{(i,j)}_{\tau_0}:=\lim_{n\rightarrow\infty}a^{(i,j)}_{\tau_n}$ for $i,j=1,\ldots,n^-(L)$.
By \eqref{A}, there holds
\begin{align*}
\left\langle A_{\tau_n}\vec{\xi}_{\tau_n,j},\vec{\phi}\right\rangle_{X\times X}=\left\langle \sum_{i=1}^{n^-(L)}a^{(i,j)}_{\tau_n}\vec{\xi}_{\tau_n,i},\vec{\phi}\right\rangle_{X\times X},\ \forall \vec{\phi}\in\text{Dom}(A),\ j=1,\ldots,n^-(L).
\end{align*}
Taking $n\rightarrow\infty$,  we obtain that
\begin{align*}
\left\langle\vec{\xi}_{\tau_0,j},A^*_{\tau_0}\vec{\phi}\right\rangle_{X\times X}=\left\langle \sum_{i=1}^{n^-(L)}a^{(i,j)}_{\tau_0}\vec{\xi}_{\tau_0,i},\vec{\phi}\right\rangle_{X\times X},\ \forall \vec{\phi}\in\text{Dom}(A),\ j=1,\ldots,n^-(L).
\end{align*}
By a density argument, it implies that
\begin{align*}
A_{\tau_0}\vec{\xi}_{\tau_0,j}= \sum_{i=1}^{n^-(L)}a^{(i,j)}_{\tau_0}\vec{\xi}_{\tau_0,i},\ j=1,\ldots,n^-(L).
\end{align*}
So $\text{span}\{\vec{\xi}_{\tau_0,1},\ldots,\vec{\xi}_{\tau_0,n^-(L)}\}$ is invariant under $A_{\tau_0}$.

Let the $n^{-}(L)\times n^-(L)$ matrices $H_{\tau_n}:=\left(a^{(i,j)}_{\tau_n}\right)$ and $H_{\tau_0}:=\left(a^{(i,j)}_{\tau_0}\right)$. From Lemma \ref{lem2.22}, the upper bound and the lower bound of the eigenvalues of $H_{\tau_n}$ are both positive. Since $\lim_{n\rightarrow\infty}a^{(i,j)}_{\tau_n}=a^{(i,j)}_{\tau_0}$, it obtains that the eigenvalues of $H_{\tau_0}$ are positive bounded below. It shows that $\text{span}\{\vec{\xi}_{\tau_0,1},\ldots,\vec{\xi}_{\tau_0,n^-(L)}\}$ is a subspace of the unstable space of $A_{\tau_0}$, i.e. $n^u(A_{\tau_0})\geq n^-(L)$. Together with Lemma \ref{lem2.3}, we obtain $n^u(A_{\tau_0})=n^-(L)$ and $\tau_0\in\Lambda$.

In all, it concludes that  $\Lambda=[0,+\infty)$.  Thus, we obtain $n^u=n^-(L)$.
\end{proof}


\section{The turning point principle for non-rotating viscous stars.}\label{s3}
In this section, we consider the stability  of the non-rotating stars of \eqref{NSP} under spherically symmetric perturbations. In the spherically symmetric setting, the support $\Omega(t)$ is a ball with radius $R(t)$. For $r=|\textbf{x}|\in (0,R(t))$, set
\begin{align*}
\rho(\textbf{x},t)=\rho(r,t),\quad \textbf{u}(\textbf{x},t)=u(r,t)\frac{\textbf{x}}{r}.
\end{align*}
Then, the system \eqref{NSP} can be rewritten as
\begin{align}\label{1.4}
&\rho_t+\frac{1}{r^2}(r^2\rho u)_r=0&\ \textrm{in}\ (0,R(t)),\\
&\rho(u_t+uu_r)+P_r+\frac{4\pi\rho}{r^2}\int^r_0\rho(s,t)s^2ds=\nu\left(\frac{(r^2u)_r}{r^2}\right)_r&\ \textrm{in}\ (0,R(t)),\label{1.5}
\end{align}
where $\nu=\frac{4\nu_1}{3}+\nu_2>0$ is the viscosity constant.

Since $R(t)$ is the moving interface of fluids and vacuum states, we have
\begin{align}
\rho|_{r=R(t)}=0.
\end{align}
At the free boundary, we impose the kinematic condition:
\begin{align}
\frac{d}{dt}R(t)=u(R(t),t).
\end{align}
The continuity of the  normal stress implies $\textbf{S}\cdot\textbf{n}=0$ at the surface  where $\textbf{n}$ represents the exterior unit normal vector. In the spherically symmetric setting, this condition reduces to
\begin{align}\label{1.8}
P-\frac{4\nu_1}{3}\left(u_r-\frac{u}{r}\right)-\nu_2\left(u_r+\frac{2u}{r}\right)=0
\end{align}
for $r=R(t),t\geq 0.$ Due to the symmetry, it is known that $u(0,t)=0$.

The linearized system of \eqref{1.4}-\eqref{1.5} at a given non-rotating star $(\rho_{\mu}(r), 0)$ with center density $\mu$ is
\begin{align}\label{3.1}
&\rho_t+\frac{1}{r^2}(r^2 \rho_{\mu}u)_r=0,\\
&u_t+ (\Phi''(\rho_{\mu})\rho)_r+V_r(\rho) -\frac{\nu}{\rho_{\mu}}\left(\frac{\left(r^2 u\right)_r}{r^2}\right)_r =0,\label{3.2}
\end{align}
where
\begin{align}\label{3.3}
V_r(\rho) =\frac{4\pi}{r^2}\int^r_0\rho(s)s^2ds.
\end{align}

We can reasonably  take  the density perturbation and the velocity perturbation with  the the same support of $\rho_\mu$. That is $\mathrm{supp}\rho=\mathrm{supp}u= \left\{|\textbf{x}|\leq R_\mu\right\}$. Indeed, the Lagrangian version of the linearized Navier-Stokes-Poisson system is equivalent to \eqref{3.1}-\eqref{3.3} which is shown in the appendix.

To study the stability, we differentiate \eqref{3.2} with respect to $t$ and use \eqref{3.1}. There holds
\begin{align*}
u_{tt}&= -\left(\Phi''(\rho_\mu)\rho_t\right)_r - V_r (\rho_t)+\frac{\nu}{\rho_\mu}\left(\frac{(r^2 u_t)_r}{r^2}\right)_r\\
&=\left(\Phi''(\rho_\mu)\frac{\left(r^2 \rho_\mu u\right)_r}{r^2}\right)_r + V_r \left(\frac{\left(r^2 \rho_\mu u\right)_r}{r^2}\right)+ \frac{\nu}{\rho_\mu}\left(\frac{\left(r^2 u_t\right)_r}{r^2}\right)_r.
\end{align*}
Denote it by
\begin{align}\label{3.4}
u_{tt}=-L_{\mu}u-D_{\mu}u_t,
\end{align}
where
\begin{equation}
L_{\mu}:=-\left(  \Phi^{\prime\prime}(\rho_{\mu})\frac{\left(r^{2}\rho_{\mu}\cdot\right)_{r}}{r^{2}}\right)_{r}  -V_{r}\left(  \frac
{\left(r^{2}\rho_{\mu}\cdot\right)_{r}}{r^{2}}\right)  , \label{defn-L-mu}%
\end{equation}%
\begin{equation}
D_{\mu}:=-\frac{\nu}{\rho_{\mu}}\left(  \frac{\left(r^{2}\cdot\right)_{r}}{r^{2}}\right)_{r}. \label{defn-D-mu}%
\end{equation}
One can see the spectrum problem corresponding to \eqref{3.4} is
\begin{align}\label{3.7}
\lambda^2u+\lambda D_{\mu}u+L_{\mu}u=0.
\end{align}


To apply  Theorem \ref{thm: KTC} obtained in Section \ref{s2}, we should check the assumptions (A1)-(A7). Choose the state space $\mathfrak{X}_{\mu}=X^{D_{\mu}}_1\times X_{\mu}$, where $X_{\mu}$ is a weighted Hilbert space:
\begin{align*}
X_{\mu}:=L^2_{\rho_\mu}=\left\{u(r),r\in [0,R_\mu]\Big|\int_0^{R_\mu}\rho_\mu r^2|u|^2dr<\infty\right\}
\end{align*}
with norm
\begin{align*}
\|u\|_{X_{\mu}}=\int^{R_\mu}_0 \rho_\mu r^2|u(r)|^2dr
\end{align*}
and inner product
\begin{align*}
\langle u,v\rangle_{X_{\mu}}=\Re\int^{R_\mu}_0 \rho_\mu r^2 u\bar{v}dr.
\end{align*}

Rewrite \eqref{3.4} by
\begin{align*}
\vec{u}_t=A_\mu\vec{u}(t)=\left(
    \begin{array}{cc}
      0 & I \\
      -L_\mu & -D_\mu \\
    \end{array}
  \right)\vec{u}(t) ,
\end{align*}
where $\vec{u}=(u,u_t)^{T}$, $\textrm{Dom}(A_{\mu}):=\left(\mathrm{Dom}(L_{\mu})\cap \mathrm{Dom}(D_{\mu})\right)\times \mathrm{Dom}(D_{\mu})$.
The domain of unbounded operator $D_{\mu}:\textrm{Dom} (D_{\mu}) \subset X_{\mu}\rightarrow X_{\mu}$ consists of all functions $u\in L^2_{\rho_\mu}$ with $\frac{\nu}{\rho_\mu}\partial_r \left(\frac{1}{r^2}\partial_r(r^2 u)\right)\in L^2_{\rho_\mu}$ in the distribution sense and   $u(0)=0$, $-\frac{4\nu_1}{3}( \partial_r u-\frac{u}{r})-\nu_2( \partial_r u+ \frac{2u}{r})|_{r=R_\mu}=0$ in the sense of $W^{1,\infty}$-trace and $H^1$-trace respectively.
The domain of  unbounded operator $L_{\mu}:\textrm{Dom} (L_{\mu}) \subset X_{\mu}\rightarrow X_{\mu}$ consists of all functions $u\in L^2_{\rho_\mu}$ with $-\partial_r \left(\Phi''(\rho_\mu)\frac{1}{r^2}\partial_r(r^2 \rho_\mu u)\right)-\partial_r V\left(\frac{1}{r^2}\partial_r(r^2 \rho_\mu u)\right)\in L^2_{\rho_\mu}$ in the distribution sense and $u(0)=0$ in the sense of $W^{1,\infty}$-trace.

Firstly, the assumption (A1) is clear. Secondly, the assumption (A2) is true by the following lemma.
\begin{lemma}\label{lem4.2}
The operators $L_{\mu}$ and  $D_{\mu}$ are self-adjoint, closed and densely defined.
\end{lemma}
\begin{proof}
The property of $L_\mu$ is obtained in \cite{LZ2019}. By the definition of $D_{\mu}$, we can easily obtain the results.
\end{proof}

Thirdly, the assumption (A3) holds due to
\begin{align}\label{D}
\langle D_{\mu}u,u\rangle_{X_{\mu}}=\int_0^{R_\mu}   \left(\frac{\nu_2}{r^2}|(r^2u)_r|^2+\frac{4\nu_1}{3}r^4\left|\left(\frac{u}{r}\right)_r\right|^2\right)dr> 0,\ \forall u\neq 0.
\end{align}

 The following two lemmas are used to check the assumption (A4)-(A6).

\begin{lemma}\label{lem3.3}
The operator $L_{\mu}$ is $D_{\mu}$-bounded.
\end{lemma}
\begin{proof}
Let
\begin{align*}
D_{\mu1}&:=\int^{R_\mu}_0  \frac{1}{r^2}|(r^2u)_r|^2 dr
=\int^{R_\mu}_0\left(r^2|u_r |^2+4r\Re (\bar{u}u_r)+4|u|^2\right)dr,\\
D_{\mu2}&:=\int^{R_\mu}_0  r^4\left|\left(\frac{u}{r}\right)_r\right|^2dr
=\int^{R_\mu}_0\left( r^2|u_r|^2-2r\Re (\bar{u}u_r )+|u|^2\right)dr.
\end{align*}
There holds
\begin{align*}
D_{\mu1}+2D_{\mu2}=\int^{R_\mu}_0 \left( 3r^2|u_r |^2+6|u|^2\right)dr
\end{align*}
and
\begin{align}\label{3.21}
\int^{R_\mu}_0 \left(r^2|u_r|^2+|u|^2\right)dr\lesssim \langle D_{\mu}u,u\rangle_{X_{\mu}}\lesssim \|D_{\mu}u\|^2_{X_{\mu}}.
\end{align}

To obtain the result, we shall show
\begin{align*}
\|L_{\mu}u\|_{X_{\mu}}=\left\|\left(\Phi''(\rho_\mu )\frac{1}{r^2}\left(r^2 \rho_\mu u\right)_r\right)_r +V_r\left(\frac{1}{r^2}(r^2 \rho_\mu u)_r\right)\right\|_{X_{\mu}}
\lesssim\|u\|_{X_{\mu}}+\left\|D_\mu u\right\|_{X_{\mu}}.
\end{align*}
Indeed, we only need to estimate
\begin{align*}
&\left\|\left(\Phi''(\rho_\mu )\frac{1}{r^2}(r^2 \rho_\mu u)_r\right)_r \right\|_{X_{\mu}}\\
=&\int_0^{R_\mu}\rho_\mu  r^2\left|\left(\Phi''(\rho_\mu )\frac{1}{r^2}(r^2 \rho_\mu u)_r\right)_r\right|^2dr\\
=&\int_0^{R_\mu}\rho_\mu r^2\left|\Phi''(\rho_\mu )\left(\frac{(r^2\rho_\mu  u)_r}{r^2}\right)_r+(\Phi''(\rho_\mu ))_r\frac{(r^2\rho_\mu  u)_r}{r^2}\right|^2dr\\
\lesssim&\int_0^{R_\mu}\Bigg(\rho_\mu r^2|\Phi''(\rho_\mu )|^2\left|\left(\frac{\rho_\mu (r^2 u)_r+r^2u\rho'_\mu }{r^2}\right)_r\right|^2+\rho_\mu r^2|(\Phi''(\rho_\mu ))_r|^2\left|\frac{r^2u\rho'_\mu +\rho_\mu (r^2u)_r}{r^2}\right|^2\Bigg)dr\\
\lesssim&\int_0^{R_\mu}\Bigg(\rho_\mu r^2|\Phi''(\rho_\mu )|^2\left(|\rho'_\mu |^2\left|\frac{(r^2u)_r}{r^2}\right|^2+|\rho_\mu |^2\left|\left(\frac{(r^2u)_r}{r^2}\right)_r\right|^2+|u_r |^2|\rho'_\mu |^2+|u|^2|\rho''_\mu |^2\right)\\
&+\rho_\mu r^2|(\Phi''(\rho_\mu ))_r|^2\left(|\rho'_\mu |^2|u|^2+\left|\frac{\rho_\mu }{r^2}(r^2u)_r\right|^2\right)\Bigg)dr\\
=:&\sum_{i=1}^6I_i.
\end{align*}

Consider
\begin{align*}
I_4=\int^{R_\mu}_0\rho_\mu r^2|\Phi''(\rho_\mu )|^2|u|^2|\rho''_\mu |^2dr.
\end{align*}
By the Hardy's inequality \cite{Jang2013CPAM,KML2007} and $\gamma_1<2$, for $r_2<R_\mu$ close to $R_\mu$, we have
\begin{align*}
&\int^{R_{\mu}}_{r_2}\rho_\mu r^2|\Phi''(\rho_\mu )|^2|u|^2|\rho''_\mu |^2dr\\
\lesssim&\int^{R_{\mu}}_{r_2}(R_\mu-r)^{\frac{1}{\gamma_1-1}}(R_\mu-r)^{\frac{2(\gamma_1-2)}{\gamma_1-1}}(R_\mu-r)^{\frac{2}{\gamma_1-1}-4}|ru|^2dr\\
\lesssim&\int^{R_{\mu}}_{r_2}(R_\mu-r)^{\frac{3-2\gamma_1}{\gamma_1-1}}|u|^2dr\\
\lesssim&\int^{R_{\mu}}_{r_2}(R_\mu-r)^{\frac{1}{\gamma_1-1}}\left(|u|^2+|u_r|^2\right)dr\\
\lesssim&\|D_{\mu}u\|^2_{X_{\mu}}.
\end{align*}
Here, we use $\rho_\mu(r)\thicksim (R_\mu-r)^\frac{1}{\gamma_1-1}$ for $r<R_\mu$ close to $R_\mu$.
For $r_1>0$ small enough, we get
\begin{align*}
\int^{r_1}_0\rho_\mu r^2|\Phi''(\rho_\mu )|^2|u|^2|\rho''_\mu |^2dr
\lesssim\int^{r_1}_0\rho_\mu r^2|u|^2dr.
\end{align*}
It concludes that $I_4\lesssim\|D_{\mu}u\|^2_{X_{\mu}}.$ The estimate of $I_5=\int^{R_\mu}_0\rho_\mu r^2|(\Phi''(\rho_\mu ))_r|^2|u|^2|\rho'_\mu |^2dr$ is similar to $I_4$.

Note that for $r<R_\mu$ close to $R_\mu$, there holds
\begin{align*}
\rho_\mu ^4|\Phi''(\rho_\mu )|^2\lesssim(R_\mu-r)^\frac{4}{\gamma_1-1}(R_\mu-r)^\frac{2(\gamma_1-2)}{\gamma_1-1}\leq C.
\end{align*}
Therefore,
$$I_2=\int_0^{R_\mu}\rho_\mu r^2|\Phi''(\rho_\mu )|^2|\rho_\mu |^2\left|\left(\frac{(r^2u)_r}{r^2}\right)_r\right|^2dr\lesssim \|D_{\mu}u\|^2_{X_{\mu}}.$$

The arguments about the estimates of $$I_1=\int_0^{R_\mu}\rho_\mu r^2|\Phi''(\rho_\mu )|^2|\rho'_\mu |^2\left|\frac{(r^2u)_r}{r^2}\right|^2dr$$ and $$I_6=\int^{R_\mu}_0\rho_\mu r^2|(\Phi''(\rho_\mu ))_r|^2\left|\frac{\rho_\mu }{r^2}(r^2u)_r\right|^2dr$$ are almost the same. Taking $I_1$ for example, note that
\begin{align*}
I_1\lesssim\int^{R_\mu}_0 \rho_\mu r^2|\Phi''(\rho_\mu )|^2|\rho'_\mu |^2\left(|u_r|^2+\frac{|u|^2}{r^2}\right)dr
=:I_{11}+I_{12}.
\end{align*}
Obviously, for $r<R_\mu$ close to $R_\mu$, we have
\begin{align*}
\rho_\mu |\Phi''(\rho_\mu )|^2|\rho'_\mu |^2\lesssim(R_\mu-r)^{\frac{1}{\gamma_1-1}+\frac{2(\gamma_1-2)}{\gamma_1-1}+\frac{2}{\gamma_1-1}-2}\lesssim(R_\mu-r)^\frac{1}{\gamma_1-1}\leq C.
\end{align*}
Hence, there holds
\begin{align*}
I_{12}=\int^{R_\mu}_0 \rho_\mu |\Phi''(\rho_\mu )|^2|\rho'_\mu |^2|u|^2dr
\lesssim\|D_{\mu}u\|^2_{X_{\mu}}.
\end{align*}
Similarly, we have
\begin{align*}
I_{11}=\int^{R_\mu}_0  \rho_\mu r^2|\Phi''(\rho_\mu )|^2|\rho'_\mu |^2|u_r|^2dr
\lesssim\int^{R_\mu}_0  r^2|u_r|^2dr.
\end{align*}
It obtains that $I_1\lesssim \|D_{\mu}u\|^2_{X_{\mu}}.$ The estimate of $I_3=\int_0^{R_\mu}\rho_\mu r^2|\Phi''(\rho_\mu )|^2|u_r|^2|\rho'_\mu|^2dr$ is similar to $I_{11}$. It completes the proof.
\end{proof}

\begin{lemma}
The operator $-D_{\mu}$ generates a strongly continuous semigroup on $X_{\mu}$.
\end{lemma}
\begin{proof}
For  $\lambda\in\mathbb{C}$ with $\Re\lambda>0$, by \eqref{D}, we have
\begin{align*}
\|(\lambda+D_{\mu})u\|_{X_{\mu}}^2=|\lambda|^2\langle u,u\rangle_{X_{\mu}}+2\Re\lambda\langle D_{\mu}u,u\rangle_{X_{\mu}}+\langle D_{\mu}u,D_{\mu}u\rangle_{X_{\mu}}\geq (\Re\lambda)^2\|u\|^2_{X_{\mu}}.
\end{align*}
By the Hille-Yosida Theorem in \cite{EN2000}, it completes the proof.
\end{proof}

The following lemmas can help us to check the assumptions (A4) and (A7).
\begin{lemma}\label{lem4.1}
There exists a positive constant $m>0$ such that $(L_{\mu}+m)^{-1}$ is compact in $X_\mu$. Hence, all nonzero eigenvalues $\lambda_k^2\in\mathbb{R}$ of $L_{\mu}$ are isolated with finite multiplicity satisfying
\begin{align*}
\lim_{k\rightarrow+\infty} \lambda^2_k=+\infty.
\end{align*}
Moreover, there exists a decomposition $X_{\mu}=X_{\mu-}\oplus X_{\mu+}$ such that
 $$n^{-}(L_{\mu}):=\dim X_{\mu-}<\infty,\ L|_{X_{\mu-}\backslash \{0\}}<0,\ L|_{X_{\mu+}\backslash \{0\}}\geq\delta>0.$$
\end{lemma}
\begin{proof}
Define the function space $Z_{\mu}$ to be the closure of $\textrm{Dom} (L_{\mu}) \subset X_\mu$ equipped with norm
\begin{align*}
\|u\|_{Z_{\mu}}=\left(\int_0^{R_\mu}\rho_\mu r^2|u|^2dr\right)^{\frac{1}{2}}+\left(\int_0^{R_\mu} \Phi''(\rho_\mu)r^2|\frac{1}{r^2}(r^2\rho_\mu u)_r|^2dr\right)^\frac{1}{2}.
\end{align*}
Similar to the proof of Lemma 3.13 in \cite{LZ2019}, we obtain the embedding $Z_{\mu}\hookrightarrow X_\mu$ is compact.
By \eqref{3.3}, we have
\begin{align*}
\int_0^{R_\mu}V_r\left(\frac{(r^2\rho_\mu u)_r}{r^2}\right)r^2\rho_\mu \bar{u} dr=\int_0^{R_\mu} 4\pi r^2\rho_\mu^2|u|^2dr\lesssim \int_0^{R_\mu} \rho_\mu|u|^2r^2dr.
\end{align*}
Therefore, there exists a positive constant $m>4\pi\sup_{r\in[0,R_\mu]}\rho_\mu$ such that
\begin{align*}
&\langle (L_{\mu}+m)u,u\rangle_{X_{\mu}}\\
=&\int_0^{R_\mu} \Phi''(\rho_\mu)\left|\frac{(r^2\rho_\mu u)_r}{r^2}\right|^2r^2dr-\int_0^{R_\mu} 4\pi r^2\rho_\mu^2|u|^2dr+m\int_0^{R_\mu} \rho_\mu r^2 |u|^2dr\\
\geq& \int_0^{R_\mu} \Phi''(\rho_\mu)\left|\frac{(r^2\rho_\mu u)_r}{r^2}\right|^2r^2dr.
\end{align*}

For every $f\in X_\mu$, consider the following equation:
\begin{align}\label{4.6}
(L_{\mu}+m)u=f.
\end{align}
Multiplying \eqref{4.6} with $r^2\rho_\mu \bar{u}$ and integrating from $0$ to $R_\mu$, we have
\begin{align*}
\int_0^{R_\mu} \left(\frac{1}{r^2}\Phi''(\rho_\mu)|(r^2 \rho_\mu u)_r|^2-4\pi r^2\rho_\mu^2|u|^2+mr^2\rho_\mu |u|^2\right)dr=\int_0^{R_\mu} f r^2\rho_\mu \bar{u}dr.
\end{align*}
Hence,
\begin{align*}
\int_0^{R_\mu} \frac{1}{r^2}\Phi''(\rho_\mu)|(r^2 \rho_\mu u)_r|^2dr\lesssim \|f\|^2_{X_\mu}.
\end{align*}
By the compactness of embedding $Z_{\mu}\hookrightarrow X_\mu$, we conclude that  $(L_{\mu}+m)^{-1}$ is compact.
\end{proof}

Denote $L_{in}^\mu:=\Phi''(\rho_\mu)-4\pi(-\Delta)^{-1}$.
By Lemma 3.7 and Lemma 3.10 in \cite{LZ2019} and $\frac{dM_\mu}{d\mu}\neq0$,  we know that
\begin{align*}
\begin{cases}
\ker (L_{in}^\mu)=\text{span}\left\{\frac{d\rho_\mu}{d\mu}\right\},\text{ if } \frac{d}{d\mu}\left(\frac{M_\mu}{R_{\mu}}\right)=0,\\
\ker (L_{in}^\mu)=\text{span}\left\{0\right\},\text{ if } \frac{d}{d\mu}\left(\frac{M_\mu}{R_{\mu}}\right)\neq0.\\
\end{cases}
\end{align*}
For any $u\in \ker(L_\mu)\subset X_\mu$, we have $L_\mu u=\left(L_{in}^\mu\left(\frac{(r^2\rho_\mu u)_r}{r^2}\right)\right)_r=0$.
Hence, there holds $L_{in}^\mu\left(\frac{(r^2\rho_\mu u)_r}{r^2}\right)=C$ which yields
\begin{align*}
0=\left\langle L_{in}^\mu\left(\frac{(r^2\rho_\mu u)_r}{r^2}\right),\frac{d\rho_\mu}{d\mu}\right\rangle_{L^2(0,R_\mu)}=C\frac{dM_\mu}{d\mu}.
\end{align*}
It shows that $C=0$ and $\frac{(r^2\rho_\mu u)_r}{r^2}\in \ker \left(L_{in}^\mu\right)$. If $\frac{d}{d\mu}\left(\frac{M_\mu}{R_{\mu}}\right)\neq0$, we have $u=0$.
If $\frac{d}{d\mu}\left(\frac{M_\mu}{R_{\mu}}\right)=0$,  we have $\frac{(r^2\rho_\mu u)_r}{r^2}=\tilde{C}\frac{d\rho_\mu}{d\mu}$ for some constant $\tilde{C}$.
By the fact that $0=\int_0^{R_\mu}\frac{(r^2\rho_\mu u)_r }{r^2}dr=\tilde{C}\frac{dM_\mu}{d\mu}$ and $\frac{dM_\mu}{d\mu}\neq0$, we know that
$u=\tilde{C}=0$. In all, one has $\ker(L_\mu)=\{0\}$ when $\frac{dM_\mu}{d\mu}\neq0$. Together with Lemma \ref{lem4.1}, it obtains that the assumption (A4) holds.

\begin{lemma}\label{lem4.3}
There holds
\begin{align*}
\langle L_{\mu}u,u\rangle_{X_{\mu}}\lesssim \langle D_{\mu}u,u\rangle_{X_{\mu}}+\langle u,u\rangle_{X_{\mu}}\lesssim\langle D_{\mu}u,u\rangle_{X_{\mu}}.
\end{align*}
\end{lemma}
\begin{proof}
Note that
\begin{align*}
\langle L_{\mu}u,u\rangle_{X_{\mu}}=&\int_0^{R_\mu}  \frac{\Phi''(\rho_\mu )}{r^2}|(r^2\rho_\mu u)_r|^2dr-\int_0^{R_\mu} 4\pi r^2\rho_\mu^2u^2dr\\
\leq&\int_0^{R_\mu}  \frac{\Phi''(\rho_\mu )}{r^2}|(r^2\rho_\mu u)_r|^2dr\\
\lesssim&\int_0^{R_\mu}  \frac{\Phi''(\rho_\mu )}{r^2}\left(|\rho'_\mu |^2|r^2u|^2+|\rho_\mu |^2|(r^2u)_r|^2\right)dr\\
=:&I_{L1}+I_{L2}.
\end{align*}
By the assumption (P1), we have
\begin{align*}
\Phi''(\rho_\mu )|\rho_\mu |^2\sim\rho_\mu ^{\gamma_1-2}\rho_\mu ^2\sim\rho_\mu ^{\gamma_1}\leq C,
\end{align*}
for $r<R_\mu$ close to $R_\mu$. Hence, it obtains that $I_{L2}\lesssim \langle D_{\mu}u,u\rangle_{X_{\mu}}$.

For $r_2<R_\mu$ close to $R_\mu$, we have
\begin{align*}
\int^{R_\mu}_{r_2}\frac{\Phi''(\rho_\mu )}{r^2}|\rho'_\mu |^2|r^2u|^2dr
\lesssim&\int^{R_\mu}_{r_2}(R_\mu-r)^\frac{\gamma_1-2}{\gamma_1-1}(R_\mu-r)^{\frac{2}{\gamma_1-1}-2}|u|^2r^2dr\\
\lesssim&\int^{R_\mu}_{r_2}(R_\mu-r)^\frac{2-\gamma_1}{\gamma_1-1}|u|^2r^2dr.
\end{align*}
For $r_1>0$ close to $0$, we get
\begin{align*}
\int^{r_1}_0\frac{\Phi''(\rho_\mu )}{r^2}|\rho'_\mu |^2|r^2u|^2dr
\lesssim\int^{r_1}_0 r^2|u|^2dr.
\end{align*}
By \eqref{3.21}, it obtains that
$
I_{L1}\lesssim\langle D_{\mu}u,u\rangle_{X_\mu}
$
which completes the proof.
\end{proof}

After  checking  (A1)-(A7), we  apply
Theorem \ref{thm: KTC} to \eqref{3.1}-\eqref{3.3} and obtain Theorem \ref{nEPeqnNSP} and Corollary \ref{ttpforvis}.

\section{Nonlinear stability under the spherically
symmetric perturbation}\label{Nonlinearstability}

In this section, we assume that  $n^-(L_{\mu})=0$ and $\frac{dM_{\mu}}{d\mu}\neq0$. By Theorem \ref{nEPeqnNSP}, the corresponding non-rotating gaseous star modeled by Navier-Stokes-Poisson equations  is linearly stable. We will study the nonlinear stability in this part.

\subsection{Formulation in Lagrangian coordinates}
In order to compare with \cite{LTXZZH20162} more conveniently, we redefine the symbols of each variable. Let $x$ be the radius of the spherical coordinate system in 3-dimensional Euclidean space.
Recall the system \eqref{1.4}-\eqref{1.8}:
\begin{align}\label{5.5}
&x^2\rho_t+(x^2\rho u)_x=0,&\ \textrm{in}\ (0,R(t)),\\
&\rho(u_t+uu_x)+P_x+\frac{4\pi\rho}{x^2}\int^x_0\rho s^2ds=\nu(x^{-2}(x^2u)_x)_x,&\ \textrm{in}\ (0,R(t)),\label{5.6}\\
\notag&\frac{dR(t)}{dt}=u(R(t),t),&\\
\notag&\frac{4\nu_1}{3}\left(u_x-\frac{u}{x}\right)+\nu_2\left(u_x+\frac{2u}{x}\right)=0, \rho=0,&\ \text{at}\ x=R(t),\\
\notag&u(x,t)=0,&\ \text{at}\ x=0.
\end{align}

The particle trajectory Lagrangian formulation is stated as follows. Let  $r(x,t)$ be the  Lagrangian variable satisfying 
\begin{align*}
&r_t(x,t)=u(r(x,t),t),\\
&\int^{r(x,t)}_0\rho(s,t)s^2ds=\int^{x}_0\rho_\mu(s)s^2ds,\quad 0\leq x\leq R_\mu.
\end{align*}
Define the Lagrangian density and velocity  by
\begin{align*}
f(x,t)=\rho(r(x,t),t),\ v(x,t)=u(r(x,t),t).
\end{align*}
Then, \eqref{5.5}-\eqref{5.6} can be rewritten as
\begin{align}\label{5.13}
&(r^2f)_t+r^2f\frac{v_x}{r_x}=0,&\ \textrm{in}\ (0,R_\mu),\\
&fv_t+\frac{(P(f))_x}{r_x}+\frac{4\pi f}{r^2}\int^{x}_0\rho_\mu s^2ds=\frac{\nu}{r_x}\left(\frac{(r^2v)_x}{r^2r_x}\right)_x,&\ \textrm{in}\ (0,R_\mu).\label{5.14}
\end{align}
By \eqref{5.13}, we have
\begin{align}\label{5.15}
f(x,t)=\frac{x^2\rho_\mu}{r^2r_x}.
\end{align}
Hence, \eqref{5.14} becomes
\begin{align}\label{5.16}
\rho_\mu\left(\frac{x}{r}\right)^2v_t-\nu\left(\frac{(r^2v)_x}{r^2r_x}\right)_x
=-\left(P\left(\frac{x^2\rho_\mu}{r^2r_x}\right)\right)_x+(P(\rho_\mu))_x\frac{x^4}{r^4},\ \textrm{in}\ (0,R_\mu) .
\end{align}
Here, we use
\begin{align}\label{P}
(P(\rho_\mu))_x=-\frac{4\pi \rho_\mu}{x^2}\int^x_0\rho_\mu s^2ds.
\end{align}

Denote
\begin{align}\label{B}
B(x,t)=\nu\frac{v_x}{r_x}+\left(2\nu_2-\frac{4}{3}\nu_1\right)\frac{v}{r}=\frac{4}{3}\nu_1\frac{r}{r_x}\left(\frac{v}{r}
\right)_x+\nu_2\frac{(r^2v)_x}{r_xr^2}.
\end{align}
Then,  we have
\begin{align*}
\nu\left(\frac{(r^2v)_x}{r^2r_x}\right)_x=B_x+4\nu_1\left(\frac{v}{r}\right)_x.
\end{align*}
The boundary value conditions become
\begin{align}\label{7.31}
v(0,t)=0,\ B(R_\mu,t)=0.
\end{align}

Set $w=\frac{r}{x}-1$. Then,
\begin{align}\label{5.19}
r=x+xw,r_x=1+w+xw_x,v=xw_t,v_x=w_t+xw_{tx},v_t=xw_{tt}.
\end{align}
Linearizing the above transforms around $r=x$ and $v=0$, we obtain  the relation between the linear perturbations:
\begin{align}
&\delta r=x\delta w,\delta r_x=\delta w+x\delta w_x,\delta v=x\delta w_t,\delta v_x=\delta w_t+x\delta w_{tx},\delta v_t=x\delta w_{tt},\label{ltran}\\
&x\delta r_x-\delta r=x^2\delta w_x, x^2\delta w_{tx}=x\delta v_x-\delta v.\label{ltran2}
\end{align}

By \eqref{5.15} and \eqref{5.19},
the linear perturbation of density  satisfies
\begin{align}\label{pd}
\delta\rho&=\left(-2\frac{\rho_\mu}{x}-\rho'_\mu\right)\delta r-\rho_\mu\delta r_x
=\left(-2\frac{\rho_\mu}{x}-\rho'_\mu\right)x\delta w-\rho_\mu\left(\delta w+x\delta w_x\right).
\end{align}
The linear perturbation of velocity satisfies
\begin{align}\label{pv}
\delta u=\delta v=x\delta w_t.
\end{align}

Note that the linearized equation around $r=x,v=0$ of \eqref{5.16} is
\begin{align}
&\rho_\mu \delta v_t-3\left(\rho_\mu P'(\rho_\mu)\right)_x\frac{\delta r}{x}-(\rho_\mu P'(\rho_\mu))_x\left(\delta r_x-\frac{\delta r}{x}\right)\nonumber\\&-\rho_\mu P'(\rho_\mu)\frac{(x^2(x\delta r_x-\delta r))_x}{x^3}
+4\frac{(P(\rho_\mu))_x}{x}\delta r=\nu\left(\delta v_x+\frac{2\delta v}{x}\right)_x.\label{5.18}
\end{align}
This equation is equivalent to \eqref{3.1}-\eqref{3.3} which is shown in the Part II of Appendix.

Without confusion, we omit the perturbation symbol $\delta$. Using \eqref{ltran} and \eqref{ltran2}, \eqref{5.18} becomes
\begin{align}\label{5.23}
x\rho_\mu  w_{tt}-\left(3(\rho_\mu P'(\rho_\mu))_x-4(P(\rho_\mu))_x\right)w-\frac{\left(\rho_\mu P'(\rho_\mu)x^4 w_x\right)_x}{x^3}=\nu(xw_{xt}+3 w_t)_x.
\end{align}
The linearized boundary condition is
\begin{align}\label{5.24}
xw_{xt}+ w_t=0,\ \text{at}\ x=R_\mu.
\end{align}
One can see the system \eqref{5.23}-\eqref{5.24} is coincide with (A.41a)-(A.41b) in \cite{LTXZZH20162} for the polytropes.

%

\subsection{Linear stable estimates}
Multiplying \eqref{5.23} with $x^3w_t$ and $x^3w$ respectively, integrating from $0$ to $R_\mu$ and using \eqref{5.24}, we obtain
\begin{align}
&\frac{1}{2}\frac{d}{dt}\int_0^{R_\mu} \left(x^4\rho_\mu |w_t |^2+x^4P'(\rho_\mu)\rho_\mu |w_x |^2+x^3\left((P(\rho_\mu))_x-3(P'(\rho_\mu))_x\rho_\mu\right)w^2\right)dx\nonumber\\
&+\nu\int_0^{R_\mu}  x^2\left((w_t +xw_{tx})^2+2|w_t |^2\right)dx=0\label{5.25}
\end{align}
and
\begin{align}
\int_0^{R_\mu}  x^4\rho_\mu |w_t|^2dx=&\frac{1}{2}\frac{d}{dt}\int_0^{R_\mu} \left(2x^4\rho_\mu ww_t +\nu(x^2(w+x w_x)^2+2x^2w^2)\right)dx\nonumber\\
&+\int_0^{R_\mu} \left(x^4P'(\rho_\mu)\rho_\mu |w_x |^2+x^3\left((P(\rho_\mu))_x-3(P'(\rho_\mu))_x\rho_\mu\right)w^2\right) dx.\label{5.26}
\end{align}

Now, multiplying \eqref{5.25} with $(1+t)$ and integrating from $0$ to $t$, we have
\begin{align*}
&(1+t)\int_0^{R_\mu} \left( \ x^4\rho_\mu |w_t |^2+x^4P'(\rho_\mu)\rho_\mu |w_x|^2+x^3\left((P(\rho_\mu))_x-3(P'(\rho_\mu))_x\rho_\mu\right)w^2\right)dx\\
&-\int^t_0\int_0^{R_\mu} \left( x^4\rho_\mu |w_s |^2+x^4P'(\rho_\mu)\rho_\mu | w_x |^2+x^3((P(\rho_\mu))_x-3(P'(\rho_\mu))_x\rho_\mu)w^2\right)dxds\\
&+2\nu\int^t_0(1+s)\int_0^{R_\mu}   x^2\left((w_s+xw_{sx})^2+2|w_s |^2\right)dxds\\
=&\int_0^{R_\mu} \left( x^4\rho_\mu |w_t|^2+x^4P'(\rho_\mu)\rho_\mu |w_x |^2+x^3((P(\rho_\mu))_x-3(P'(\rho_\mu))_x\rho_\mu)w^2\right)dx\Big|_{t=0}.
\end{align*}
Integrating \eqref{5.25} and \eqref{5.26} from $0$ to $t$ respectively, we have
\begin{align*}
&\int_0^{R_\mu} \left( x^4\rho_\mu |w_t |^2+x^4P'(\rho_\mu)\rho_\mu |w_x|^2+x^3((P(\rho_\mu))_x-3(P'(\rho_\mu))_x\rho_\mu)w^2\right)dx\\
&+2\nu\int^t_0\int_0^{R_\mu}   x^2\left((w_s +xw_{sx})^2+2|w_s |^2\right)dxds\\
=&\int_0^{R_\mu} \left( x^4\rho_\mu |w_t |^2+x^4P'(\rho_\mu)\rho_\mu |w_x |^2+x^3\left((P(\rho_\mu))_x-3(P'(\rho_\mu))_x\rho_\mu\right)w^2\right)dx\Big|_{t=0}
\end{align*}
and
\begin{align*}
&\int_0^{R_\mu}  \left(2x^4\rho_\mu w w_t+\nu(x^2(w+xw_x )^2+2x^2w^2)\right)dx-\int^t_0\int_0^{R_\mu}   x^4\rho_\mu |w_s|^2dxds\\
&+\int^t_0\int_0^{R_\mu} \left( x^4P'(\rho_\mu)\rho_\mu | w_x |^2+x^3((P(\rho_\mu))_x-3(P'(\rho_\mu))_x\rho_\mu)w^2\right)dxds\\
=&\int_0^{R_\mu}  \left(2x^4\rho_\mu ww_t +\nu(x^2(w+xw_x)^2+2x^2w^2)\right)dx\Big|_{t=0}.
\end{align*}
Combining these, there holds
\begin{align}
\notag&\int_0^{R_\mu}  \left(x^2(w+xw_x)^2+2x^2w^2\right)dx+(1+t)\left(\int_0^{R_\mu}  x^4\rho_\mu |w_t|^2dx+\langle L_{\mu}(xw),(xw)\rangle_{X_\mu}\right)\\
\notag&+\int^t_0\langle L_{\mu}(xw),(xw)\rangle_{X_\mu} ds+\int^t_0(1+s)\int_0^{R_\mu}  \left(x^2(w_s +xw_{sx})^2+2x^2|w_{s}|^2\right)dxds\\
\lesssim &\int_0^{R_\mu} \left( x^4\rho_\mu | w_t|^2+x^4P'(\rho_\mu)\rho_\mu |w_x|^2+x^3((P(\rho_\mu))_x-3(P'(\rho_\mu))_x\rho_\mu)w^2\right)dx\Big|_{t=0}.\label{5.50}
\end{align}

By \eqref{P}, we have
\begin{align*}
-4\pi\rho_\mu^2x^4
=&\rho_\mu x^2\left(\frac{P'(\rho_\mu)(\rho'_\mu)^2}{\rho_\mu}\right)_x\\
=&-\frac{P'(\rho_\mu)(\rho'_\mu)^2 x^4}{\rho_\mu}+(P'(\rho_\mu))_x\rho'_\mu x^4+P'(\rho_\mu)\rho''_\mu x^4+2P'(\rho_\mu)\rho'_\mu x^3.
\end{align*}
Therefore,
\begin{align}
\langle L_{\mu}(xw),(xw)\rangle_{X_\mu}&=\int_0^{R_\mu}   \left( \frac{P'(\rho_\mu)}{\rho_\mu}\frac{1}{x^2}|(x^3\rho_\mu w)_x|^2-4\pi x^4\rho_\mu^2w^2\right)dx\nonumber\\
&=\int_0^{R_\mu}   \left( x^3(4( P(\rho_\mu))_x-3(\rho_\mu P'(\rho_\mu))_x)w^2+x^4P'(\rho_\mu)\rho_\mu w^2_x\right)dx\nonumber\\
&=\int_0^{R_\mu}   \left( x^4P'(\rho_\mu)\rho_\mu w^2_x+x^3((P(\rho_\mu))_x-3\rho_\mu(P'(\rho_\mu))_x)w^2\right)dx. \label{5.52}
\end{align}

By Lemma 3.13 in \cite{LZ2019}, there holds
\begin{align*}
\int_0^{R_\mu} \frac{\Phi''(\rho_\mu )}{r^2}|(r^2\rho_\mu u)_r|^2dr\gtrsim\int_0^{R_\mu} r^2\rho_\mu |u|^2dr,\ \forall u\in \textrm{Dom}(L_\mu).
\end{align*}
Hence, by \eqref{a}, there exists a constant $C>0$ small enough such that
\begin{align}
\langle L_{\mu}u,u\rangle_{X_\mu}=&\int_0^{R_\mu} \left(\frac{\Phi''(\rho_\mu )}{r^2}|(r^2\rho_\mu u)_r|^2-4\pi r^2\rho^2_\mu|u|^2\right)dr\nonumber\\
\geq& C\int_0^{R_\mu}\left(\frac{\Phi''(\rho_\mu )}{r^2}|(r^2\rho_\mu u)_r|^2+ r^2\rho_\mu |u|^2\right)dr.\label{5.37}
\end{align}
Replacing $u$ and $r$   by $xw$ and $x$, there holds
\begin{align*}
\langle L_{\mu}(xw),(xw)\rangle_{X_\mu}&\gtrsim\int_0^{R_\mu}  P'(\rho_\mu )\rho_\mu x^2|(xw)_x |^2dx+\int_0^{R_\mu}  x^4\rho_\mu |w|^2dx\\
&\gtrsim\int_0^{R_\mu}  P'(\rho_\mu )\rho_\mu x^4|w_x |^2dx+\int_0^{R_\mu}  x^4\rho_\mu |w|^2dx.
\end{align*}
Therefore, by \eqref{5.50}, we have
\begin{align}
&(1+t)\int_0^{R_\mu} x^4\rho_\mu|w|^2dx+(1+t)\int_0^{R_\mu}  P'(\rho_\mu )\rho_\mu x^4|w_x |^2dx+(1+t)\int x^4\rho_\mu |w_t|^2dx\nonumber\\
\lesssim & \int_0^{R_\mu} \left( x^4\rho_\mu |w_t |^2+x^4P'(\rho_\mu)\rho_\mu |w_x |^2+x^3( (P(\rho_\mu))_x-3(P'(\rho_\mu))_x\rho_\mu)w^2\right)dx\Big|_{t=0}.\label{5.55}
\end{align}

By the Hardy's inequality, for $r_1>0$ small enough, we have
\begin{align*}
\int^{r_1}_0x^2|w|^2dx\lesssim \int^{r_1}_0 (x^4|w|^2+x^4|w_x |^2)dx.
\end{align*}
It obtains that
\begin{align*}
&\int^{r_1}_0\frac{P'(\rho_\mu)}{\rho_\mu}x^2|\rho_\mu w|^2dx\lesssim\int^{r_1}_0 P'(\rho_\mu)\rho_\mu x^2|w|^2dx\lesssim \int^{r_1}_0 \left(x^4|w|^2+x^4|w_x |^2\right)dx,\\
&\int^{r_1}_0\frac{P'(\rho_\mu)}{\rho_\mu}x^2|\rho'_\mu xw|^2dx\lesssim\int^{r_1}_0 x^4|w|^2dx.
\end{align*}
For $r_2<R_\mu$ close to $R_\mu$, we get
\begin{align*}
\int^{R_\mu}_{r_2}\frac{P'(\rho_\mu)}{\rho_\mu}x^2|\rho_\mu w|^2dx\lesssim&\int^{R_\mu}_{r_2} \rho^{\gamma_1}_\mu x^2|w|^2dx\lesssim \int^{R_\mu}_{r_2} \rho_\mu x^4|w|^2dx,\\
\int^{R_\mu}_{r_2}\frac{P'(\rho_\mu)}{\rho_\mu}x^2|\rho'_\mu xw|^2dx
\lesssim&\int^{R_\mu}_{r_2}(R_\mu-r)^{\frac{2-\gamma_1}{\gamma_1-1}}|w|^2dx\\
\lesssim&\int^{R_\mu}_{r_2}(R_\mu-r)^{\frac{2-\gamma_1}{\gamma_1-1}+2}(|w_x |^2+|w|^2)dx\\
\lesssim&\int^{R_\mu}_{r_2} \left(P'(\rho_\mu)\rho_\mu x^4|w_x |^2+x^4\rho_\mu|w|^2\right)dx.
\end{align*}
Then, by \eqref{pd}, there holds
\begin{align*}
\int_0^{R_\mu} \frac{P'(\rho_\mu)}{\rho_\mu}|\rho|^2x^2dx
\lesssim&\int_0^{R_\mu} \frac{P'(\rho_\mu)}{\rho_\mu}x^2\left(|\rho_\mu x w_x|^2+|\rho_\mu w|^2+|\rho'_\mu xw|^2\right)dx\\
\lesssim& \int_0^{R_\mu}\left(P'(\rho_\mu)\rho_\mu x^4| w_x|^2+x^4\rho_\mu|w|^2\right)dx.
\end{align*}

Combining with \eqref{pv}, \eqref{5.52} and \eqref{5.55}, we have
\begin{align*}
&(1+t)\int_0^{R_\mu} \left(x^2\rho_\mu|r|^2+x^2\rho_\mu|u|^2+\frac{P'(\rho_\mu)}{\rho_\mu}x^2|\rho|^2\right)dx\\
\lesssim& \left(\int_0^{R_\mu}  x^4\rho_\mu |w_t|^2dx+\langle L_\mu(xw),(xw)\rangle_{X_\mu}\right)\Big|_{t=0}\\
\lesssim&\int_0^{R_\mu} \left(x^2\rho_\mu|r|^2+x^2\rho_\mu|u|^2+\frac{P'(\rho_\mu)}{\rho_\mu}x^2|\rho|^2\right)dx\Big|_{t=0}.
\end{align*}

Summarizing the above arguments, we state
\begin{theorem}
Assume that $n^-(L_\mu)=0$ and $\frac{dM_{\mu}}{d\mu}\neq0$, i.e., the  non-rotating star of the corresponding Navier-Stokes-Possion system \eqref{1.4}-\eqref{1.5} is  linearly stable. Then, we have
\begin{align}
&(1+t)\int_0^{R_\mu} \left(x^2\rho_\mu|r|^2+x^2\rho_\mu|u|^2+\frac{P'(\rho_\mu)}{\rho_\mu}x^2|\rho|^2\right)dx\nonumber\\
\lesssim&\int_0^{R_\mu} \left(x^2\rho_\mu|r|^2+x^2\rho_\mu|u|^2+\frac{P'(\rho_\mu)}{\rho_\mu}x^2|\rho|^2\right)dx\Big|_{t=0},\label{ldecay}
\end{align}
where $\rho,u$ are the linear spherically symmetric perturbation of density and velocity respectively, $x\in[0,R_\mu]$ is the reference variable and $r(x,t)$ is the Lagrangian variable.
\end{theorem}

\begin{remark}
By the assumptions $n^-(L_{\mu})=0$, $\frac{dM_{\mu}}{d\mu}\neq0$, Lemma \ref{lem4.1} and the fact that $L_{\mu}$ satisfies the assumption (A4), there exists a constant $\lambda'>0$ such that
\begin{align}\label{a}
\langle L_{\mu}u,u\rangle_{X_\mu}\geq \lambda'\|u\|_{X_\mu}>0,\ \forall u\neq 0.
\end{align}

By \eqref{3.7}, we have
\begin{align}\label{b}
\langle \lambda^2 u,u\rangle_{X_\mu}+\langle \lambda D_{\mu}u,u\rangle_{X_\mu}+\langle L_{\mu}u,u\rangle_{X_\mu}=0.
\end{align}
There holds
\begin{align*}
2\Im\lambda\Re\lambda\langle u,u\rangle_{X_\mu}+\Im\lambda\langle D_{\mu}u,u\rangle_{X_\mu}=0.
\end{align*}
If $\Re\lambda\geq 0$, by \eqref{D}, we obtain $\Im\lambda=0$. Therefore, by \eqref{a} and \eqref{b}, we get $u\equiv 0$ for $\lambda\geq 0$. It obtains that all of the spectra are on the left half plane which is consistent with  that the  non-rotating viscous gaseous star  are linearly stable. However, it is unknown whether the decay estimate \eqref{ldecay} is sharp, which is very interesting to study in the future.
\end{remark}

\subsection{Nonlinear estimates for stability}
In  this subsection, we consider the nonlinear stability of the Lane-Emden stars of the
Navier-Stokes-Poisson system under spherically symmetric perturbations. We extend the results of \cite{LTWYZH,LTXZZH20162} by  assuming  the sharp linear stability condition $n^-(L_{\mu})=0$ and $\frac{dM_\mu}{d\mu}\neq 0$. In  \cite{LTXZZH20162}, the authors establish the global regularity uniformly up to the vacuum boundary. It ensures the long time asymptotic uniform convergence of the evolving vacuum boundary, density and velocity to those of the Lane-Emden stars with detailed convergence rates, and  long time behaviors of solutions near the vacuum boundary.
The difficulty is caused by the degeneracy and singular behavior near the vacuum free boundary and coordinates singularity at the symmetry center. Their proof depends on the polytropic case of $\gamma\in(\frac{4}{3},2)$ strongly. We face the same difficulty. Although the estimates in our proof are similar to the ones in \cite{LTXZZH20162},  we mainly use the positivity of the quadratic form $\langle L_{\mu}\cdot,\cdot\rangle_{X_\mu}$ to absorb the higher order rest term which is very different from \cite{LTWYZH,LTXZZH20162}.  Compared with the equation of states in \cite{LTWYZH}, we need more regularity of the pressure in addition and  no more assumptions for the center density of stars. The model in our paper is more general than the polytropic case including the white dwarf stars.

Additionally, in this subsection, we assume that $P(\rho)$ satisfies\\
(P2) $P(s)\in C^3(0,+\infty)$ and $P'(0)=0$. There exists $K_2>0$ such that $$\lim_{s\rightarrow 0^+}P''(s)s^{2-\gamma_1}=\lim_{s\rightarrow 0^+}|P'''(s)|s^{3-\gamma_1}=K_2.$$

\subsubsection{The positive  quadratic form}\label{sub421}
The coordinate transform \eqref{5.19} is needed in the argument about the positivity of the quadratic form $\langle L_{\mu}\cdot,\cdot\rangle_{X_\mu}$. To prepare for the next, we prove that $\langle L_{\mu}(xw),(xw)\rangle_{X_\mu}$ can absorb $\int_0^{R_\mu} x^2Q_0(w,(xw)_x)dx$ where $Q_0(w,(xw)_x)$ is the generic term  bounded by
\begin{align}\label{Q}
|Q_0(w,(xw)_x)|\lesssim\epsilon P(\rho_\mu)\left(|w|^2+|(xw)_x|^2\right)
\end{align}
provided that $|w|, |(xw)_x|\leq \epsilon$ small enough.

\begin{lemma}\label{lemL}
For $|w|, |(xw)_x|\leq \epsilon$ small enough, assume that $Q_0(w,(xw)_x)$ satisfies \eqref{Q}. There holds
\begin{align*}
&\int_0^{R_\mu}   \left( x^4P'(\rho_\mu)\rho_\mu w^2_x+x^3((P(\rho_\mu))_x-3\rho_\mu(P'(\rho_\mu))_x)w^2\right)dx+\int_0^{R_\mu} x^2Q_0(w,(xw)_x)dx\\
\gtrsim&\int_0^{R_\mu}  \left(  P'(\rho_\mu )\rho_\mu x^2|(xw)_x|^2+x^2\rho_\mu |w|^2\right)dx.
\end{align*}
\end{lemma}

\begin{proof}
Define
\begin{align*}
A_1^2:=\int_0^{R_\mu} \frac{\Phi''(\rho_\mu )}{r^2}|r^2\rho_\mu u_r|^2dr,\ A_2^2:=\int_0^{R_\mu} \frac{\Phi''(\rho_\mu )}{r^2}|(r^2\rho_\mu )_r u|^2dr,\ A_3^2:=\int_0^{R_\mu}  \rho_\mu |u|^2dr.
\end{align*}
For $r_1>0$ close to $0$, there holds
\begin{align*}
\int^{r_1}_0\frac{\Phi''(\rho_\mu )}{r^2}r^4\rho_\mu ^2|u_r |^2dr&=\int^{r_1}_0P'(\rho_\mu )\rho_\mu r^2|u_r|^2dr\gtrsim\int^{r_1}_0 r^2|u_r|^2dr,\\
\int^{r_1}_0\rho_\mu |u|^2dr&\gtrsim\int^{r_1}_0|u|^2dr.
\end{align*}
On the other hand, by  the Hardy's inequality, we have
\begin{align*}
&\int^{r_1}_0 \frac{\Phi''(\rho_\mu )}{r^2}|(r^2\rho_\mu )_r u|^2dr=\int^{r_1}_0\frac{\Phi''(\rho_\mu )}{r^2}|u|^2|2r\rho_\mu +r^2\rho'_\mu |^2dr\\
\sim&\int^{r_1}_0\frac{|u|^2}{r^2}\left(4r^2\rho_\mu ^2+4r^3\rho_\mu \rho'_\mu +r^4(\rho'_\mu )^2\right)dr
\lesssim\int^{r_1}_0|u|^2dr\\
\lesssim&\int^{r_1}_0 r^2\left(|u|^2+|u_r|^2\right)dr.
\end{align*}
For $r_2<R_\mu$ close to $R_\mu$, there holds
\begin{align*}
&\int^{R_{\mu}}_{r_2}P'(\rho_\mu )\rho_\mu r^2|u_r|^2dr\sim \int^{R_{\mu}}_{r_2}\rho_\mu ^{\gamma_1}|u_r|^2dr\sim \int^{R_{\mu}}_{r_2}(R_\mu-r)^\frac{\gamma_1}{\gamma_1-1}|u_r|^2dr,\\
&\int^{R_{\mu}}_{r_2}\rho_\mu |u|^2dr\sim\int^{R_{\mu}}_{r_2}(R-r)^\frac{1}{\gamma_1-1}|u|^2dr.
\end{align*}
Moreover, we have
\begin{align*}
&\int^{R_{\mu}}_{r_2}\frac{\Phi''(\rho_\mu )}{r^2}|(r^2\rho_\mu )_r u|^2dr\\
\sim&\int^{R_{\mu}}_{r_2}\frac{\rho^{\gamma_1-2}_\mu}{r^2}|u|^2\left|(r^2(R_\mu-r)^{\frac{1}{\gamma_1-1}})_r\right|^2dr\\
\sim&\int^{R_{\mu}}_{r_2}|u|^2(R_\mu-r)^{\frac{\gamma_1-2}{\gamma_1-1}}\left(4(R_\mu-r)^{\frac{2}{\gamma_1-1}}-4r(R_\mu-r)^{\frac{2}{\gamma_1-1}-1}+r^2(R_\mu-r)^{\frac{2}{\gamma_1-1}-2}\right)dr\\
\sim&\int^{R_{\mu}}_{r_2}|u|^2(R_\mu-r)^{\frac{2-\gamma_1}{\gamma_1-1}}\left(4(R_\mu-r)^2-4r(R_\mu-r)+r^2\right)dr\\
\sim&\int^{R_{\mu}}_{r_2}|u|^2(R_\mu-r)^{\frac{2-\gamma_1}{\gamma_1-1}}dr\\
\lesssim&\int^{R_{\mu}}_{r_2}(R_\mu-r)^{\frac{\gamma_1}{\gamma_1-1}}\left(|u|^2+|u_r|^2\right)dr.
\end{align*}
Hence, we get
\begin{align*}
A_2^2\lesssim\left(A_1^2+A_3^2\right).
\end{align*}

Lemma 3.13 in \cite{LZ2019} shows that
\begin{align*}
A^2_1+A_2^2\gtrsim\int_0^{R_\mu}  \frac{\Phi''(\rho_\mu )}{r^2}\left|(r^2\rho_\mu  u)_r\right|^2dr\gtrsim A_3^2.
\end{align*}
Therefore, combining with \eqref{5.37}, there holds
\begin{align*}
\langle L_{\mu}u,u\rangle_{X_\mu}\gtrsim(A_1+A_2)^2\gtrsim\left(A_1^2+A_3^2\right).
\end{align*}

Note that $P(\rho_\mu)\sim P'(\rho_\mu)\rho_\mu$ for $0<r<R_\mu$ near to $R_\mu$.
Replacing $u,r$ by $xw,x$ respectively,
we conclude that
\begin{align*}
&\langle L_{\mu}(xw),(xw)\rangle_{X_\mu}+\int_0^{R_\mu} x^2Q_0(w,(xw)_x)dx\\
\gtrsim&\frac{1}{2}\langle L_{\mu}(xw),(xw)\rangle_{X_\mu}\nonumber\\
 \gtrsim& \int_0^{R_\mu}  P'(\rho_\mu )\rho_\mu x^2|(xw)_x|^2dx+\int_0^{R_\mu} \rho_\mu x^2|w|^2dx
\end{align*}
provided that $|w|,|(xw)_x|\leq \epsilon$ small enough.

Together with \eqref{5.52}, we have
\begin{align}
&\int_0^{R_\mu}   \left( x^4P'(\rho_\mu)\rho_\mu w^2_x+x^3((P(\rho_\mu))_x-3\rho_\mu(P'(\rho_\mu))_x)w^2+x^2Q_0(w,(xw)_x)\right)dx\nonumber\\
\gtrsim&\int_0^{R_\mu}  \left(  P'(\rho_\mu )\rho_\mu x^2|(xw)_x|^2+x^2\rho_\mu |w|^2\right)dx\nonumber\\
=&\int_0^{R_\mu}   \left( P'(\rho_\mu )\rho_\mu x^2|r_x-1|^2+x^2\rho_\mu \left|\frac{r}{x}-1\right|^2\right)dx \label{7.27}
\end{align}
for $|w|,|(xw)_x|\leq \epsilon$ small enough. The last equality \eqref{7.27} holds by $w=\frac{r}{x}-1, (xw)_x=r_x-1$.
\end{proof}

\subsubsection{Lower-order estimates}

By the continuity arguments, we will show that for small perturbation of the Lane-Emden stars with the same mass, there is a unique global strong solution of the  problem \eqref{5.16} with boundary condition \eqref{7.31}.

Assume that for any fixed $T>0$, there exist suitably small  constants $\epsilon_1,\epsilon_2>0$ such that
\begin{align*}
&\left|\frac{r}{x}-1\right|\leq\sup_{(x,t)\in [0,R_\mu]\times[0,T]}\left|r_x-1\right|\leq \epsilon_1,\\
&\left|\frac{v}{x}\right|\leq\sup_{(x,t)\in [0,R_\mu]\times[0,T]}|v_x|\leq\epsilon_2.
\end{align*}
Denote $\|\cdot\|^2_2:=\int^{R_\mu}_0 |\cdot|^2dx$, $\|\cdot\|_\infty:=\sup_{x\in[0,R_\mu]}|\cdot|$.
\begin{lemma}\label{lem7.1}
For $0\leq t\leq T$, we have
\begin{align*}
&\|x\rho_\mu^{\frac{1}{2}}v\|_2^2+\left\|x(P(\rho_\mu))^{\frac{1}{2}}\left(\frac{r}{x}-1,r_x-1\right)\right\|^2_2+\sigma\int^t_0\|(v,xv_x)\|^2_2ds\\
\lesssim& \left(\|x\rho_\mu^{\frac{1}{2}}v\|^2_2+\left\|x(P(\rho_\mu))^{\frac{1}{2}}\left(\frac{r}{x}-1,r_x-1\right)\right\|^2_2\right)\Bigg|_{t=0},
\end{align*}
where $\sigma=\min\left\{\frac{2\nu_1}{3},\nu_2\right\}$.
\end{lemma}
\begin{proof}
By \eqref{5.19}, \eqref{5.16} becomes
\begin{align}
\notag&\rho_\mu\frac{1}{(1+w)^2}xw_{tt}+\left(P\left(\frac{\rho_\mu}{(1+w)^2(1+w+xw_x)}\right)\right)_x-(P(\rho_\mu))_x\frac{1}{(1+w)^4}\\
=&\nu\left(\frac{((x+xw)^2xw_t)_x}{(x+xw)^2(1+w+xw_x)}\right)_x.\label{7.46}
\end{align}
Multiplying \eqref{7.46} by $(x+xw)^2xw_t$, integrating from $0$ to $R_\mu$ and using  Taylor expansion, we have
\begin{align}
\notag&\frac{1}{2}\frac{d}{dt}\int_0^{R_\mu}  \left( x^4\rho_\mu w^2_t+x^4P'(\rho_\mu)\rho_\mu w^2_x+x^3((P(\rho_\mu))_x-3(P'(\rho_\mu))_x\rho_\mu)w^2+x^2Q_1(w,(xw)_x)\right)dx\\
=&-\int_0^{R_\mu}   \left(B(x,t)((x+xw)^2xw_t)_x+4\nu_1\left(\frac{xw_t}{x+xw}\right)_x(x+xw)^2xw_t\right)dx,\label{7.48}
\end{align}
where $B(x,t)$ is defined by \eqref{B}. Here, the Taylor rest term $Q_1(w,(xw)_x)$ is bounded by
\begin{align*}
Q_1(w,(xw)_x)\lesssim&\left(P(\rho_\mu)+\rho_\mu P'(\rho_\mu)+\frac{\rho_\mu^4}{\tilde{\rho}^4}\left|\tilde{\rho}^2P''(\tilde{\rho})-4\tilde{\rho}P'(\tilde{\rho})+6P(\tilde{\rho})\right|\right)\left(|w|^3+|(xw)_x|^3\right)\\
\lesssim&\epsilon_1 P(\rho_\mu)\left(|w|^2+|(xw)_x|^2\right),
\end{align*}
where $\tilde{\rho}=\rho_\mu+\rho_\mu\xi_1\left(\frac{1}{(1+w)^2((xw)_x+1)}-1\right)$ with some constant $\xi_1\in(0,1)$.

Integrating \eqref{7.48} from $0$ to $t$, by Lemma \ref{lemL}, there holds
\begin{align*}
&\int_0^{R_\mu}   x^4\rho_\mu w^2_tdx+\int_0^{R_\mu}\left(  P'(\rho_\mu )\rho_\mu x^2|(xw)_x|^2+x^2\rho_\mu |w|^2\right)dx\\
&+\sigma\int^t_0\left( \|xw_t\|^2_2+\|xw_t+x^2w_{tx}\|^2_2\right)ds\\
\lesssim&\left(\|x^2\rho_\mu^{\frac{1}{2}}w_t\|_2^2+\|x(P(\rho_\mu))^{\frac{1}{2}}(w,w+xw_x)\|_2^2\right)\Bigg|_{t=0},
\end{align*}
where $\sigma=\min\left\{\frac{2\nu_1}{3},\nu_2\right\}$.
 The proof is completed by \eqref{5.19}.
\end{proof}

\begin{lemma}\label{lem7.2}
For $0\leq t\leq T$, we have
\begin{align}
\notag&\sigma\left\|x\left(\frac{r}{x}-1,r_x-1\right)\right\|^2_2+\int^t_0\left\|x(P(\rho_\mu))^\frac{1}{2}\left(\frac{r}{x}-1,r_x-1\right)\right\|_2^2ds\\
\lesssim&\left(\left\|x\left(\frac{r}{x}-1,r_x-1\right)\right\|_2^2+\|x\rho_\mu^\frac{1}{2}v\|_2^2\right)\Bigg|_{t=0}\label{7.52}
\end{align}
and
\begin{align}
\notag&(1+t)\|x\rho_\mu^\frac{1}{2}v\|_2^2+(1+t)\left\|x(P(\rho_\mu))^\frac{1}{2}\left(\frac{r}{x}-1,r_x-1\right)\right\|_2^2\\
\notag&+(1+t)^\frac{2\gamma_1-2}{\gamma_1}\|r-x\|_2^2+\sigma\int^t_0(1+s)\|(v,xv_x)\|_2^2ds\\
\lesssim&\left(\left\|x\left(\frac{r}{x}-1,r_x-1\right)\right\|_2^2+\|x\rho_\mu^\frac{1}{2}v\|_2^2\right)\Bigg|_{t=0},\label{7.55}
\end{align}
where $\sigma=\min \left\{\frac{2\nu_1}{3},\nu_2\right\}$.
\end{lemma}
\begin{proof}
 To prove \eqref{7.52}, we  multiply \eqref{5.16} by $r^3-x^3$ or multiply \eqref{7.46} by $x^3((w+1)^2-1)$. We have
\begin{align*}
&\int_0^{R_\mu} \left(  P(\rho_\mu)\left(\frac{x^4}{r^4}(r^3-x^3)\right)_x-P\left(\frac{x^2\rho_\mu}{r^2r_x}\right)(r^3-x^3)_x\right)dx\\
=&-\int_0^{R_\mu}\left(B(x,t)(r^3-x^3)_x-4\nu_1\left(\frac{v}{r}\right)_x(r^3-x^3)\right)dx-\int_0^{R_\mu} x^3\rho_\mu v_t\left(\frac{r}{x}-\frac{x^2}{r^2}\right)dx.
\end{align*}
Using \eqref{5.19} and the Taylor expansion, there holds
\begin{align}
&\int_0^{R_\mu}  \left( P(\rho_\mu)\left(\frac{x^4}{r^4}(r^3-x^3)\right)_x-P\left(\frac{x^2\rho_\mu}{r^2r_x}\right)(r^3-x^3)_x\right)dx\nonumber\\
=&6\int_0^{R_\mu}\Big((-4P(\rho_\mu)+4\rho_\mu P'(\rho_\mu))x^2w^2+2x^2(-4P(\rho_\mu)+2\rho_\mu P'(\rho_\mu))w(w+xw_x)\nonumber\\
&+\rho_\mu P'(\rho_\mu)x^2(w+xw_x)^2\Big)dx+\int_0^{R_\mu} x^2  Q_2(w,(xw)_x)dx\nonumber\\
=&6\int_0^{R_\mu} \left(x^4P'(\rho_\mu)\rho_\mu w^2_x-(3P'(\rho_\mu)\rho_\mu-4P(\rho_\mu))_xx^3w^2\right)dx+\int_0^{R_\mu} x^2Q_2(w,(xw)_x)dx\nonumber\\
=&6\langle L_\mu(xw),(xw)\rangle_{X_\mu}+\int_0^{R_\mu} x^2  Q_2(w,(xw)_x)dx,\label{7.59}
\end{align}
where the Taylor rest term $Q_2(w,(xw)_x)$  can be bounded by
\begin{align*}
Q_2(w,(xw)_x)\lesssim& \left(P(\rho_\mu)+\rho_\mu P'(\rho_\mu)+\rho^2_\mu|P''(\tilde{\rho})| \right)\left(|w|^3+|(xw)_x|^3\right)\\
\lesssim& \epsilon_1P(\rho_\mu)\left(|w|^2+|(xw)_x|^2\right),
\end{align*}
where $\tilde{\rho}=\rho_\mu+\rho_\mu\xi_2\left(\frac{1}{(1+w)^2((xw)_x+1)}-1\right)$ for some constant $\xi_2\in(0,1)$.
Therefore,  \eqref{7.52} holds by Lemma \ref{lemL} and  the similar arguments in Lemma 3.6 of \cite{LTXZZH20162}.

Define
\begin{align*}
\eta(x,t):=  x^4\rho_\mu w^2_t+x^4P'(\rho_\mu)\rho_\mu w^2_x+x^3((P(\rho_\mu))_x-3(P'(\rho_\mu))_x\rho_\mu)w^2+x^2Q_2(w,(xw)_x).
\end{align*}
By the arguments in Lemma \ref{lemL}, we have
\begin{align*}
\int_0^{R_\mu}  \eta(x,t)dx \sim\int_0^{R_\mu}  x^2\rho_\mu|xw_t|^2dx+\int_0^{R_\mu}\left(  P'(\rho_\mu )\rho_\mu x^2|(xw)_x|^2+x^2\rho_\mu |w|^2\right)dx.
\end{align*}
Hence,
\begin{align*}
\int^t_0\int_0^{R_\mu}  \eta(x,s)dxds\lesssim\int^t_0\int_0^{R_\mu}   v^2dxds+\int^t_0\int_0^{R_\mu}  x^2P(\rho_\mu)\left(\left(\frac{r}{x}-1\right)^2+(r_x-1)^2\right)dxds.
\end{align*}
Multiplying \eqref{7.48} by $(1+t)$ and integrating with respect to time, we obtain \eqref{7.55} by the similar arguments in  Lemma 3.6 of \cite{LTXZZH20162}.
\end{proof}

\begin{lemma}
Let $r_2<R_\mu$ be close to $R_\mu$, there holds
\begin{align*}
\left|\frac{r}{x}-1\right|+|r_x-1|\lesssim\left(\|x\rho_\mu^{\frac{1}{2}}v\|_2+\left\|x\left(\frac{r}{x}-1,r_x-1\right)\right\|_2+\|r_x-1\|_{L^\infty[r_1,R_\mu]}\right)\Bigg|_{t=0}
\end{align*}
for $(x,t)\in [r_2,R_\mu]\times[0,T]$.
\end{lemma}
\begin{proof}
Consider $x\in [r_2,R_\mu]$ for $0<r_2<R_\mu$ close to $R_\mu$. Integrating \eqref{5.16} over $[x,R_\mu]$, we have
\begin{align*}
\int_x^{R_\mu}\rho_\mu\left(\frac{y}{r}\right)^2v_tdy-P\left(\frac{x^2\rho_\mu}{r^2r_x}\right)-\int^{R_\mu}_x\frac{y^4}{r^4}(P(\rho_\mu))_ydy=-B(x,t)+4\nu_1\int^{R_\mu}_x\left(\frac{v}{r}\right)_ydy.
\end{align*}
Similar to the arguments in Lemma 3.6 of \cite{LTXZZH20162}, the only small differences are
\begin{align*}
\mathfrak{u}:=\int^t_0P\left(\frac{x^2\rho_\mu}{r^2r_x}\right)ds\sim\int^t_0\left(\frac{x^2\rho_\mu}{r^2r_x}\right)^{\gamma_1} ds
\end{align*}
and
\begin{align*}
\mathfrak{u}_t=P\left(\frac{x^2\rho_\mu}{r^2r_x}\right)\sim \left(\frac{x^2\rho_\mu}{r^2r_x}\right)^{\gamma_1}.
\end{align*}
Fortunately, we also have
\begin{align*}
e^{\frac{\gamma_1}{\nu}\mathfrak{u}}\sim 1+\int^t_0 \frac{\gamma_1}{\nu}\left(\frac{x^2}{r^2}\frac{\rho_\mu}{r_x|_{t=0}}\right)^{\gamma_1} e^{-\frac{\gamma_1}{\nu}\mathfrak{L}}ds,
\end{align*}
where the meaning of $\mathfrak{L}$ is  the same as that in \cite{LTXZZH20162}. By Lemma \ref{lem7.1} and \ref{lem7.2}, it completes the proof.
\end{proof}

\begin{lemma}
There holds
\begin{align*}
&(1+t)\|(x\rho_\mu^\frac{1}{2}v_t,v,xv_x)(\cdot,t)\|^2_2+\int^t_0(1+s)\|(v_s,xv_{sx})(\cdot,s)\|^2_2ds\\
\lesssim&\left(\left\|\left(x\rho_\mu^\frac{1}{2}v_t,v,xv_x\right)\right\|_2^2+\left\|x\left(\frac{r}{x}-1,r_x-1\right)\right\|^2_2\right)\Bigg|_{t=0}.
\end{align*}
\end{lemma}
\begin{proof}
Note that
\begin{align}\label{7.70}
\rho_\mu x^2v_t+r^2\left(P\left(\frac{x^2\rho_\mu}{r^2r_x}\right)\right)_x-\frac{x^4}{r^2}(P(\rho_\mu))_x=r^2B_x+4\nu_1\left(\frac{v}{r}\right)_xr^2
\end{align}
and
\begin{align}
\notag&\rho_\mu x^2v_{tt}+2rv\left(P\left(\frac{x^2\rho_\mu}{r^2r_x}\right)\right)_x+r^2\left(P\left(\frac{x^2\rho_\mu}{r^2r_x}\right)\right)_{xt}+\frac{2x^4}{r^3}v(P(\rho_\mu))_x\\
=&r^2B_{xt}+2rvB_x+4\nu_1\left(\frac{v}{r}\right)_{xt}r^2+8\nu_1\left(\frac{v}{r}\right)_xrv.\label{7.72}
\end{align}

Multiplying \eqref{7.72} by $v_t$ and integrating from $0$ to $R_\mu$, we have
\begin{align*}
\frac{d}{dt}\int_0^{R_\mu}  \left(\frac{1}{2}\rho_\mu x^2|v_t|^2+\eta_1(x,t)dx\right)dx+\int_0^{R_\mu}  \left(B_t(r^2v_t)_x-4\nu_1r^2v_t\left(\frac{v}{r}\right)_{xt}\right)dx=:J_1+J_2,
\end{align*}
where
\begin{align*}
  \eta_1:=&\frac{1}{2}\int_0^{R_\mu} x^2\rho_\mu|v_t|^2dx+\frac{1}{2}\int_0^{R_\mu}   P'\left(\frac{x^2\rho_\mu}{r^2r_x}\right)\frac{x^2\rho_\mu}{r^2r_x}\frac{r^2}{r_x}|v_x|^2dx\\
&+\int_0^{R_\mu}   vv_x\left(4P'\left(\frac{x^2\rho_\mu}{r^2r_x}\right)\frac{x^2\rho_\mu}{rr_x}-4rP\left(\frac{x^2\rho_\mu}{r^2r_x}\right)-4\frac{x^4}{r^3}P(\rho_\mu)\right)dx\\
&+\int_0^{R_\mu}   v^2\left(4P'\left(\frac{x^2\rho_\mu}{r^2r_x}\right)\frac{x^2\rho_\mu}{r^2}-2r_xP\left(\frac{x^2\rho_\mu}{r^2r_x}\right)-2P(\rho_\mu)\left(\frac{x^4}{r^3}\right)_x\right)dx,\\
J_1:=&-\int_0^{R_\mu} \left(B(x,t)(2rvv_t)_x-8\nu_1(\frac{v}{r})_xrvv_t\right)dx,\\
J_2:=&\int_0^{R_\mu}   vv_x\left(4P'\left(\frac{x^2\rho_\mu}{r^2r_x}\right)\frac{x^2\rho_\mu}{rr_x}-4rP\left(\frac{x^2\rho_\mu}{r^2r_x}\right)-4\frac{x^4}{r^3}P(\rho_\mu)\right)_tdx\\
&+\int_0^{R_\mu}   v^2\left(4P'\left(\frac{x^2\rho_\mu}{r^2r_x}\right)\frac{x^2\rho_\mu}{r^2}-2r_xP\left(\frac{x^2\rho_\mu}{r^2r_x}\right)-2P(\rho_\mu)\left(\frac{x^4}{r^3}\right)_x\right)_tdx\\
&+\int_0^{R_\mu}   \left(P'\left(\frac{x^2\rho_\mu}{r^2r_x}\right)\frac{x^2\rho_\mu}{r^2r_x}\frac{r^2}{r_x}\right)_t|v_x|^2dx.
\end{align*}
Note that
\begin{align*}
&\langle L_{\mu}v,v\rangle_{X_\mu}\\
=&\int_0^{R_\mu} \left(  x^4P'(\rho_\mu)\rho_\mu|\left(\frac{v}{x}\right)_x|^2+x^3\left(4(P(\rho_\mu))_x-3(\rho_\mu P'(\rho_\mu))_x\right)\left(\frac{v}{x}\right)^2\right)dx\\
=&\int_0^{R_\mu} \left(  P'(\rho_\mu)\rho_\mu x^2v^2_x-(4P(\rho_\mu)-2P'(\rho_\mu)\rho_\mu)2xvv_x-(4P(\rho_\mu)-4P'(\rho_\mu)\rho_\mu)v^2\right)dx.
\end{align*}
Using Taylor expansion and Lemma \ref{lemL}, we have
\begin{align*}
\eta_1&\sim \frac{1}{2}\int_0^{R_\mu} x^2\rho_\mu|v_t|^2dx+\frac{1}{2}\langle L_\mu v,v\rangle_{X_\mu},\\
J_2&\lesssim\epsilon_1\int_0^{R_\mu} (x^2v^2_x+v^2)dx,
\end{align*}
for $\epsilon_1>0$ small enough.
The details of the rest proof are similar to Lemma 3.8 in \cite{LTXZZH20162} which completes the proof.
\end{proof}

Now, we derive the time improved decay estimates. By the similar arguments in Lemma 3.13 of \cite{LZ2019}, we have
\begin{align*}
\int_0^{R_\mu} \rho_\mu^{-\beta}\frac{\Phi''(\rho_\mu)}{x^2}|(x^2\rho_\mu u)_x|^2dx\gtrsim\int_0^{R_\mu}\rho_\mu^{-\beta}\rho_\mu|u|^2x^2dx
\end{align*}
for $0<\beta<\gamma_1-1$.
Hence, we obtain that the quadric form
$$\int_0^{R_\mu} \rho_\mu^{-\beta} \left(\Phi''(\rho_\mu)\frac{1}{x^2}|(x^3\rho_\mu w)_x|^2-4\pi x^4\rho_\mu^2w^2\right)dx$$
can absorb  the  term  $\int_0^{R_\mu} \rho_\mu^{-\beta} x^2Q_0(w,(xw)_x)dx$ with
$$Q_0(w,(xw)_x)\lesssim\epsilon_1 P(\rho_\mu )(|w|^2+|(xw)_x|^2)$$ for $|w|,|(xw)_x|\leq \epsilon_1$ small enough.

\begin{lemma}
For $\epsilon_1,\epsilon_2>0$ small enough and $\theta\in \left(0,\frac{2(\gamma_1-1)}{3\gamma_1}\right)$, there exists a constant $C_\theta$ independent of $t$ such that
\begin{align*}
&\|\rho_\mu^{\frac{\gamma_1\theta-2(\gamma_1-1)}{4}}(r-x,xr_x-x)\|_2^2+(1+t)^{\frac{\gamma_1-1}{\gamma_1}-\theta}\|xr_x-x\|_2^2+(1+t)^{\frac{3(\gamma_1-1)}{\gamma_1}-\theta}\|r-x\|^2_2\\
&+(1+t)^{\frac{2\gamma_1-1}{\gamma_1}-\theta}\left(\|(x\rho_\mu^\frac{1}{2}v_t,v,xv_x)\|^2_2+\|\rho_\mu^\frac{\gamma_1}{2}(r-x,xr_x-x)\|^2_2\right)\\
&+\int^t_0\left((1+s)^\frac{\gamma_1-1-\gamma_1\theta}{\gamma_1}\|\rho_\mu^\frac{\gamma_1}{2}(r-x,xr_x-x)\|^2_2+(1+s)^\frac{2\gamma_1-1-\theta\gamma_1}{\gamma_1}\|(v,xv_x,v_s,xv_{sx})\|_2^2\right)ds\\
&+\int^t_0\|\rho_\mu^\frac{\theta\gamma_1+2}{4}(r-x,xr_x-x)\|^2_2ds+\int^t_0(1+s)^\frac{2\gamma_1-1-\theta\gamma_1}{2\gamma_1}\|\rho_\mu^\frac{\gamma_1\theta-2(\gamma_1-1)}{4}(v,xv_x)\|^2_2ds\\
\leq& C_\theta\left(\|(v,xv_x,x\rho_\mu^\frac{1}{2}v_t)\|^2_2+\|r_x-1\|^2_\infty\right)\Bigg|_{t=0},\ \forall t\in[0,T].
\end{align*}
Moreover, for any $a\in(0,R_\mu)$ and $\theta\in \left(0,\frac{2(\gamma_1-1)}{3\gamma_1}\right)$, we obtain
\begin{align*}
&(1+t)^{\frac{2(\gamma_1-1)}{\gamma_1}-\theta}\left(\|r-x\|^2_{L^\infty[a,R_\mu]}+\|v\|^2_{L^\infty[a,R_\mu]}\right)\\
\lesssim&\left(\|(v,xv_x,x\rho_\mu^\frac{1}{2}v_t)\|_2+\|r_x-1\|^2_\infty\right)\Bigg|_{t=0},\ \forall t\in[0,T].
\end{align*}
\end{lemma}
\begin{proof}
Multiplying \eqref{7.70} by $\int^x_0 \rho_{\mu}^{-\beta}(r^3-y^3)_ydy~(\beta=\gamma_1-1-\frac{\gamma_1\theta}{2})$ and integrating  from $0$ to $R_\mu$, we have
\begin{align*}
&\int_0^{R_\mu} \rho_\mu^{-\beta}(\rho_\mu )\left(P(\rho_\mu )\left(\frac{x^4}{r^4}(r^3-x^3)\right)_x-P\left(\frac{\rho_\mu  x^2}{r^2r_x}\right)(r^3-x^3)_x\right)dx\\
&+\int_0^{R_\mu} \rho_\mu^{-\beta}\left(B(r^3-x^3)_x-4\nu_1\left(\frac{v}{r}\right)_x(r^3-x^3)\right)dx=:\sum_{i=1}^3 L_i,
\end{align*}
where
\begin{align*}
L_1:=&\int_0^{R_\mu} \rho_\mu\frac{x^2}{r^2}v_t\left(\int^x_0\rho_\mu^{-\beta}(r^3-y^3)_ydy\right)dx,\\
L_2:=&\int_0^{R_\mu} P(\rho_\mu)\left(\frac{x^4}{r^4}\right)_x\left(\rho_\mu^{-\beta}(r^3-x^3)-\int^x_0\rho_\mu^{-\beta}(r^3-y^3)_ydy\right)dx,\\
L_3:=&4\nu_1\int_0^{R_\mu}\left(\frac{v}{r}\right)_x\left(\int^x_0\rho_\mu^{-\beta}(r^3-y^3)_ydy-\rho_\mu^{-\beta}(r^3-x^3)\right)dx.
\end{align*}
Like \eqref{7.59}, we have
\begin{align*}
&\int_0^{R_\mu} \rho_\mu^{-\beta}(\rho_\mu )\left(P(\rho_\mu )\left(\frac{x^4}{r^4}(r^3-x^3)\right)_x-P\left(\frac{\rho_\mu  x^2}{r^2r_x}\right)(r^3-x^3)_x\right)dx\\
=&6\int_0^{R_\mu}   \rho_{\mu}^{-\beta}\Big(\rho_{\mu}P'(\rho_{\mu})x^4w_x^2+(-8P(\rho_{\mu})+6\rho_{\mu}P'(\rho_{\mu}))x^3ww_x\\
&+(-12P(\rho_{\mu})+9P'(\rho_{\mu})\rho_{\mu})x^2w^2\Big)dx
+\int_0^{R_\mu} \rho_\mu^{-\beta} x^2Q_2(w,(xw)_x)dx\\
=&6\int_0^{R_\mu}  \rho_{\mu}^{-\beta}\left(\frac{P'(\rho_{\mu})}{\rho_{\mu}}\frac{1}{x^2}|(x^3\rho_{\mu}w)_x|^2-4\pi x^4\rho_{\mu}^2w^2\right)dx\\
&-\int_0^{R_\mu} (\rho_{\mu}^{-\beta})_x\left(-4P(\rho_{\mu})+3\rho_{\mu}P'(\rho_{\mu})\right)x^3w^2dx\\
&+\int_0^{R_\mu}(\rho_{\mu}^{-\beta})_x\left(P'(\rho_{\mu})3x^3\rho_{\mu}w^2+P'(\rho_{\mu})(\rho_{\mu})_xx^4w^2\right)dx\\
&+\int_0^{R_\mu} \rho_\mu^{-\beta} x^2Q_2(w,(xw)_x)dx.
\end{align*}

For $r_1>0$ close to $0$ and $r_2<R_\mu$ close to $R_\mu$, there holds
\begin{align*}
&\int^{R_{\mu}}_{r_2} (\rho_{\mu}^{-\beta})_x(-4P(\rho_{\mu})+3\rho_{\mu}P'(\rho_{\mu}))x^3w^2dx
\lesssim\int^{R_{\mu}}_{r_2} (R-x)^{-\frac{\beta}{\gamma_1-1}-1}(R-x)^{\frac{\gamma_1}{\gamma_1-1}}x^3w^2dx\\
\lesssim&\int^{R_{\mu}}_{r_2} (R-x)^{\frac{1-\beta}{\gamma_1-1}}x^3w^2dx\lesssim \int^{R_{\mu}}_{r_2} x^2w^2dx,
\end{align*}
and
\begin{align*}
&\int^{R_{\mu}}_{r_2}(\rho_{\mu}^{-\beta})_x\left(P'(\rho_{\mu})3x^3\rho_{\mu}w^2+P'(\rho_{\mu})(\rho_{\mu})_xx^4w^2\right)dx\\
\lesssim&\int^{R_{\mu}}_{r_2}(R-x)^{-\frac{\beta}{\gamma_1-1}-1}((R-x)^{\frac{\gamma_1}{\gamma_1-1}}+(R-x)(R-x)^{\frac{1}{\gamma_1-1}-1})w^2dx\\
\lesssim&\int^{R_{\mu}}_{r_2}(R-x)^{\frac{1-\beta}{\gamma_1-1}-1}w^2dx\\
\lesssim&\int^{R_{\mu}}_{r_1}(R-x)^{\frac{1-\beta}{\gamma_1-1}+1}\left(|w|^2+|w_x|^2\right)dx.
\end{align*}
The last inequality holds by  $\frac{1-\beta}{\gamma_1-1}+1=\frac{\gamma_1-\beta}{\gamma_1-1}>1$.

Hence, by Lemma \ref{lem7.2}, we have
\begin{align*}
&\int_0^{R_\mu} x^2\rho_{\mu}^{-\beta}\left(\left(\frac{r}{x}-1\right)^2+(r_x-1)^2\right)dx+\int^t_0\int_0^{R_\mu}   x^2\rho_{\mu}^{\gamma-\beta}\left(\left(\frac{r}{x}-1\right)^2+(r_x-1)^2\right)dxds\\
\lesssim&\left(\|r_{x}-1\|^2_\infty+\left\|x(\frac{r}{x}-1,r_{x}-1)\right\|_2^2+\|x\rho_{\mu}^{\frac{1}{2}}v\|_2^2\right)\Bigg|_{t=0}+\sum_{i=1}^{3}\int^t_0|L_i|ds.
\end{align*}
The details of rest proof   are similar to the proof of Lemma 3.9 in \cite{LTXZZH20162}. We omit them.
\end{proof}

\subsubsection{High-order estimates}

Define
\begin{align*}
G:=\ln r_x+2\ln\left(\frac{r}{x}\right).
\end{align*}
Then, \eqref{5.16} becomes
\begin{align}
&\nu G_{xt}+P'\left(\frac{\rho_{\mu}x^2}{r^2r_x}\right)\frac{\rho_{\mu}x^2}{r^2r_x}G_x\nonumber\\
=&\frac{x^2}{r^2}\rho_{\mu}v_t-\left(\frac{P\left(\frac{\rho_{\mu}x^2}{r^2r_x}\right)}{P(\rho_{\mu})}-\left(\frac{x}{r}\right)^4\right)x\phi\rho_{\mu}+(\rho_{\mu})_x\left(P'\left(\frac{\rho_{\mu}x^2}{r^2r_x}\right)\frac{x^2}{r^2r_x}-\frac{P'(\rho_{\mu})}{P(\rho_{\mu})}P\left(\frac{\rho_{\mu}x^2}{r^2r_x}\right)\right),\label{7.115}
\end{align}
where
\begin{align*}
\phi:=\frac{1}{x^3}\int^x_04\pi\rho_\mu(s)s^2ds\in \left[\frac{M_\mu}{R_\mu^3},\frac{4\pi \mu}{3}\right],\ M_\mu:=\int_0^{R_\mu} 4\pi\rho_\mu(s)s^2ds,
\end{align*}
and
\begin{align*}
(P(\rho_\mu))_x=-x\phi\rho_\mu.
\end{align*}

The results of Lemma 3.10-3.17 in \cite{LTXZZH20162} also hold because of the estimates don't depend on the equations.

\begin{lemma}
For any $\theta\in\left(0,\frac{2(\gamma_1-1)}{3\gamma_1}\right)$ and $a\in(0,R_\mu)$, there holds
\begin{align*}
&\left\|\rho_\mu^{\gamma_1-\frac{1}{2}}\left(r_{xx},\left(\frac{r}{x}\right)_x\right)\right\|^2_2\lesssim\mathcal{E}(0),\\
&(1+t)^{\frac{\gamma_1-1}{\gamma_1}-\theta}\left\|\left(r_{xx},\left(\frac{r}{x}\right)_x\right)\right\|^2_{L^2(0,a)}\lesssim_{\theta,a}\left(\mathcal{E}(t)+\|r_x-1\|^2_\infty\right)|_{t=0},\\
&(1+t)^{\frac{\gamma_1-1}{\gamma_1}-\theta}\left\|\rho_\mu^\frac{1}{2}v_t\right\|^2_2+\int^t_0(1+s)^{\frac{\gamma_1-1}{\gamma_1}-\theta}\left\|\left(v_x,\frac{v}{x},v_{sx},\frac{v_s}{x}\right)\right\|^2_2ds\lesssim_\theta(\mathcal{E}(t)+\|r_x-1\|^2_\infty)|_{t=0},\\
&(1+t)^{\frac{\gamma_1-1}{\gamma_1}-\theta}\left\|\left(v_{xx},\left(\frac{v}{x}\right)_x\right)\right\|^2_{L^2(0,a)}\lesssim_{\theta,a}(\mathcal{E}(t)+\|r_x-1\|^2_\infty)|_{t=0},\ \forall t\in[0,T].
\end{align*}
Here,
\begin{align*}
\mathcal{E}(t):=&\|(r-x,xr_x-x)\|^2_2+\|(v,xv_x)\|^2_2+\|r_x-1\|^2_{L^\infty[r_2,R_\mu]}\\
&+\left\|\rho_\mu^{\gamma-\frac{1}{2}}\left(r_{xx},\left(\frac{r}{x}\right)_x\right)\right\|^2_2+\|\rho_\mu^\frac{1}{2}v_t\|^2_2
\end{align*}
for $r_2<R_\mu$ close to $R_\mu.$
\end{lemma}
\begin{proof}
Multiplying \eqref{7.115} by $\rho_\mu^{2\gamma_1-1}G_x$, integrating from $0$ to $R_\mu$ and using the Cauchy inequality,  we have
\begin{align}
&\frac{\nu}{2}\int_0^{R_\mu} \rho_\mu^{2\gamma_1-1}|G_x|^2dx+\frac{1}{2}\int_0^{R_\mu}  P'\left(\frac{\rho_{\mu}x^2}{r^2r_x}\right)\frac{\rho_{\mu}x^2}{r^2r_x}\rho_\mu^{2\gamma_1-1}|G_x|^2dx\nonumber\\
\lesssim& \int_0^{R_\mu} v^2_tdx+\int_0^{R_\mu} x^2\rho^{\gamma_1}_\mu\left(\left|\frac{r}{x}-1\right|^2+|r_x-1|^2\right)dx\nonumber\\
&+\int_0^{R_\mu} \left((\rho_{\mu})_x\left(P'\left(\frac{\rho_{\mu}x^2}{r^2r_x}\right)\frac{x^2}{r^2r_x}-\frac{P'\left(\rho_{\mu}\right)}{P(\rho_{\mu})}P\left(\frac{\rho_{\mu}x^2}{r^2r_x}\right)\right)\right)^2\rho_\mu^{\gamma_1-1}dx.\label{7.126}
\end{align}
We focus on the third term in \eqref{7.126}. Using Taylor expansion, there holds
\begin{align}\label{7.129}
&\int_0^{R_\mu} \left((\rho_{\mu})_x\left(P'\left(\frac{\rho_{\mu}x^2}{r^2r_x}\right)\frac{x^2}{r^2r_x}-\frac{P'(\rho_{\mu})}{P(\rho_{\mu})}P\left(\frac{\rho_{\mu}x^2}{r^2r_x}\right)\right)\right)^2\rho_\mu^{\gamma_1-1}dx\\
= &\int_0^{R_\mu}  ((\rho_{\mu})_x)^2 \rho_\mu^{\gamma_1-1}\Big((-2P''(\rho_{\mu})\rho_{\mu}-2P'(\rho_{\mu})+2\frac{P'^2(\rho_{\mu})}{P(\rho_{\mu})}\rho_\mu)(\frac{r}{x}-1)\nonumber\\
&\left(-P''(\rho_{\mu})\rho_{\mu}-P'(\rho_{\mu})+\frac{P'^2(\rho_{\mu})}{P(\rho_{\mu})}\rho_\mu\right)(r_x-1)+Q_3(\frac{r}{x}-1,r_x-1)\Big)^2dx.\nonumber
\end{align}
where  $Q_3(\frac{r}{x}-1,r_x-1)$ is bounded  by
\begin{align*}
&Q_3(\frac{r}{x}-1,r_x-1)\\
\lesssim&\left(P'''(\tilde{\rho})\rho^2_\mu+P''(\tilde{\rho})\rho_\mu+P'(\tilde{\rho})-\frac{P'(\rho_\mu)}{P(\rho_\mu)}(P''(\tilde{\rho})\rho^2_\mu+P'(\tilde{\rho})\rho_\mu)\right)\left(\frac{r}{x}-1\right)^2+(r_x-1)^2
\end{align*}
with $\tilde{\rho}=\frac{\rho_\mu}{\left(1+\xi_3(\frac{r}{x}-1)\right)^2\left(1+\xi_3(r_x-1)\right)}$ for some $\xi_3\in(0,1)$.

For $x$ close to $R_\mu$, we have
\begin{align*}
&-2P''(\rho_{\mu})\rho_{\mu}-2P'(\rho_{\mu})+2\frac{P'^2(\rho_{\mu})}{P(\rho_{\mu})}\rho_\mu\lesssim \rho_{\mu}^{\gamma_1-1}\\
&-P''(\rho_{\mu})\rho_{\mu}-P'(\rho_{\mu})+\frac{P'^2(\rho_{\mu})}{P(\rho_{\mu})}\rho_\mu\lesssim \rho_{\mu}^{\gamma_1-1}\\
&P'''(\tilde{\rho})\rho^2_\mu+P''(\tilde{\rho})\rho_\mu+P'(\tilde{\rho})-\frac{P'(\rho_\mu)}{P(\rho_\mu)}(P''(\tilde{\rho})\rho^2_\mu+P'(\tilde{\rho})\rho_\mu)\lesssim \rho_{\mu}^{\gamma_1-1}.
\end{align*}
Note that $(\rho_\mu)_x=\frac{x\phi\rho_\mu}{P'(\rho_\mu)}.$ Hence, for $\epsilon_1$ small enough, there holds
\begin{align*}
\eqref{7.129}\lesssim\int_0^{R_\mu} x^2\rho_{\mu}^{\gamma_1}\left(\left(\frac{r}{x}-1\right)^2+(r_x-1)^2\right)dx.
\end{align*}
The details of rest proof   are similar to the proof of Lemma 3.18 in \cite{LTXZZH20162}. We omit them.
\end{proof}

By the lemmas above and the local existence in the next subsection, we obtain the global existence and long time behavior of the strong solution by the similar arguments of Theorem 2.3 in \cite{LTXZZH20162} which is stated below.
\begin{theorem}\label{thm7.7}
Assume that  $n^-(L_{\mu})=0$ and $\frac{dM_\mu}{d\mu}\neq 0$. $P$ satisfies (P1), (P2) and the initial total mass is as the same as the mass of the Lane-Emden star. There exists a constant $\varepsilon>0$ small enough such that if $E(0)\leq \varepsilon$,  the problem \eqref{5.16} with boundary condition \eqref{7.31} and initial data $(r,v)|_{t=0}=(r_0(x),v_0(r_0(x)))$ admits a unique strong solution in $[0,R_\mu]\times[0,+\infty)$ with
\begin{align*}
E(t)\lesssim E(0),\quad \forall t\geq 0,
\end{align*}
where
\begin{align*}
E(t)=\|(r_x-1,v_x)\|^2_\infty+\left\|\rho_\mu^{\gamma_1-\frac{1}{2}}\left(r_{xx},\left(\frac{r}{x}\right)_x\right)\right\|^2_2+\|\rho_\mu^{\frac{1}{2}}v_t\|^2_2.
\end{align*}
\end{theorem}

The definition of strong solution to the problem \eqref{5.16} with boundary condition \eqref{7.31} is as the same as Definition 2.1 in \cite{LTXZZH20162}. Since the faster decay estimates and further regularity of the solution are similar to those in \cite{LTXZZH20162}, the asymptotic stability of the stationary solution also holds. We state the result which is similar to Theorem 2.4 in \cite{LTXZZH20162}.

\begin{theorem}
\label{thm:asymptotic sta}Under the assumptions in Theorem \ref{thm7.7}, there
exists a unique global strong solution $(\rho(r,t),u(r,t),R(t))$ to the free
boundary problem \eqref{1.4}-\eqref{1.8} with initial data $(\rho,u,R)|_{t=0}$. Moreover, for $0<\theta
<\frac{2(\gamma_1-1)}{3\gamma_1}$, there holds
\begin{align*}
&  \sup_{x\in\lbrack0,R_{\mu}]}|r(x,t)-x|\lesssim_{\theta}(1+t)^{-\frac
{\gamma_1-1}{\gamma_1}+\frac{\theta}{2}}(E(0))^{\frac{1}{2}},\\
&  \sup_{r\in\lbrack0,R(t)]}|u(r,t)|\lesssim_{\theta}(1+t)^{-\frac{3\gamma_1
-2}{4\gamma_1}+\frac{\theta}{2}}(E(0))^{\frac{1}{2}},\\
&  \sup_{r\in\lbrack0,R(t)]}\left|u_{r},\frac{u}{r}\right|\lesssim_{\theta}(1+t)^{-\frac
{\gamma_1-1}{2\gamma_1}+\frac{\theta}{2}}(E(0))^{\frac{1}{2}},\\
&  \sup_{x\in\lbrack0,R_{\mu}]}\left|\rho_{\mu}^{\frac{3\gamma_1-6}{4}}(\rho(r(x,t),t)-\rho_{\mu}(x))\right|\lesssim_{\theta}(1+t)^{-\frac{\gamma_1
-1}{2\gamma_1}+\frac{\theta}{2}}(E(0))^{\frac{1}{2}},\ \forall t\geq0.
\end{align*}

\end{theorem}

We consider white dwarf stars to be an important example.

\begin{remark}
For white dwarf stars with the equation of states $P_{\text{w}}\left(
\rho\right)  =Af\left(  x\right)  $ and $\rho=Bx^{3}$, where $A,B$ are two
constants and
\begin{align}
f\left(  x\right)   &  =x\left(  x^{2}+1\right)  ^{\frac{1}{2}}\left(
2x^{2}-3\right)  +3\ln\left(  x+\sqrt{1+x^{2}}\right)
  =8\int_{0}^{x}\frac{u^{4}}{\sqrt{1+u^{2}}}du,\label{defn-f-white dwarf}
\end{align}
it was shown in \cite{LZ2019} that $\frac{dM_\mu}{d\mu}  >0$ and
$n^{-}(L_{\mu})=0$ for any $\mu\in\left(  0,\infty\right)  $. Thus, by
Theorems \ref{thm7.7} and \ref{thm:asymptotic sta}, any non-rotating white
dwarf star is asymptotically stable under small radial perturbation.
\end{remark}

For more general equation of states with $\gamma_{1}\in\left(  \frac{4}%
{3},2\right)  $, the non-rotating stars are linearly stable up to the first
maximum of the total mass by the turning point principle. Thus, Theorem \ref{thm:asymptotic sta} implies  the
asymptotic stability of the non-rotating stars with center density up to the
first maximum point of the mass.

\subsubsection{Local existence}

In this part, we prove the local existence of strong solution to the problem \eqref{5.16}-\eqref{7.31} with initial data $(r^0,v^0)$ on a time interval $[0,T_*]$ for some $T_*>0$ by Galerkin method.

The finite difference scheme is defined as follows. Let $N>0$ be an integer and $h=\frac{R_\mu}{N}$. For $n=0,1,\ldots,N$, let $x_n=nh$, $\rho_n=\rho_\mu(x_n)$ and $v_n=v(x_n)$. Approximate \eqref{5.16} by the following Cauchy problem for $v_1,\ldots,v_{N-1}$ of ODE:
\begin{align}
&\rho_n\left(\frac{x_n}{r_n}\right)^2\frac{dv_n}{dt}+\frac{1}{h}\left(\left(B_n+4\lambda_1\frac{v_{n-1}}{r_{n-1}}\right)-\left(B_{n+1}+4\lambda_1\frac{v_n}{r_n}\right)\right)=q_n\left(\frac{x^4_n}{r^4_n}-1\right)\nonumber\\
&+\frac{1}{h}\left(\left(P(\rho_{n+1})-P\left(\frac{x^2_n\rho_{n+1}h}{r^2_n(r_{n+1}-r_n)}\right)\right)-\left(P(\rho_{n})-P\left(\frac{x^2_{n-1}\rho_{n}h}{r^2_{n-1}(r_{n}-r_{n-1})}\right)\right)\right),\label{7.138}\\
&v_n(t)|_{t=0}=v^0(x_n),\label{7.139}
\end{align}
where
\begin{align*}
&q_n=(P(\rho_\mu))_x(x_n),\ r_n(t)=r^0(x_n)+\int^t_0v_n(s)ds,\\
&B_n=\nu\frac{v_n-v_{n-1}}{r_n-r_{n-1}}+\left(2\nu_2-\frac{4\nu_1}{3}\right)\frac{v_{n-1}}{r_{n-1}}.
\end{align*}

To match the conditions \eqref{7.31}, the boundary  conditions are imposed by
\begin{align}
&v_0(t)=0,\nonumber\\
&B_N(t)=\nu\frac{v_N-v_{N-1}}{r_N-r_{N-1}}+\left(2\nu_2-\frac{4\nu_1}{3}\right)\frac{v_{N-1}}{r_{N-1}}=0,\label{7.144}\\
&r_0(t)=0.\nonumber
\end{align}
By \eqref{7.144} and
\begin{align}\label{7.146}
\frac{dr_n}{dt}=v_n,
\end{align}
we have
\begin{align*}
0=B_N(t)=\frac{d}{dt}\left(\nu\ln\left(\frac{r_N-r_{N-1}}{h}\right)+\left(2\nu_2-\frac{4\nu_1}{3}\right)\ln\left(\frac{r_{N-1}}{x_{N-1}}\right)\right).
\end{align*}
It concludes that
\begin{align*}
r_N(t)=r_{N-1}(t)+\left(r^0(x_N)-r^0(x_{N-1})\right)\left(\frac{r^0(x_{N-1})}{r_{N-1}}\right)^{\frac{2\nu_2}{\nu}-\frac{4\nu_1}{3\nu}}.
\end{align*}
By \eqref{7.146}, we can obtain $v_N(t)$. In all, the  ordinary differential equations   are closed.

The approximation of $E(t)$ is defined by
\begin{align*}
E_N(t)=&\max_{1\leq n\leq N}\left(\left|\frac{r_n-r_{n-1}}{h}-1\right|^2+\left|\frac{v_n-v_{n-1}}{h}\right|^2\right)+h\sum_{n=1}^{N-1}\rho_n\left|\frac{dv_{n}}{dt}\right|^2\\
&+h\sum_{n=1}^{N-1}\rho_n^{2\gamma_1-1}\left(\left|\frac{r_{n+1}-2r_n+r_{n-1}}{h^2}\right|^2+\left|\frac{1}{h}\left(\frac{r_n}{x_n}-\frac{r_{n-1}}{x_{n-1}}\right)\right|^2\right).
\end{align*}

\begin{lemma}
Under the assumptions in Theorem \ref{thm7.7}, let $E(0)<\infty$ and $r^0(0)=0$. Then, there exists $\varepsilon>0$  small enough independent of $N$ such that if $\|r_x^0-1\|_\infty< \varepsilon$ and $N$ is large enough, the problem \eqref{7.138}-\eqref{7.139} admits a unique solution $(r_n,v_n)$ on the time interval $[0,T_*]$ for some constant $T_*>0$ independent of $N$. satisfying
\begin{align*}
&E_N(t)\leq C'E_N(0)\leq 2C'E(0),\ \forall t\in[0,T_*],\\
&T_*\geq \min\left\{\frac{C''}{\sqrt{2C'E(0)}},\frac{1}{C'(1+2C'E(0))}\right\},
\end{align*}
where constants $C',C''>0$ independent of $N$.
\end{lemma}
\begin{proof}
From the classical ordinary differential equation theory, the Cauchy problem \eqref{7.138}-\eqref{7.139} has a solution on $[0,T_N]$. Without lose of generality, let $T_N>0$ be the maximum existence time.  Following the ideas in \cite{LTXZZH20162}, we need to prove  there exists a constant $C>0$ independent of $N$ and $t$ such that
\begin{align*}
E_N(t)\leq C\left(E_N(0)+\int^t_0\left(E_N(s)+E^2_N(s)\right)ds\right),\forall t\in[0,T],
\end{align*}
for some $T$ satisfying
\begin{align*}
T\sup_{s\in[0,T]}\sqrt{E_N(s)}\ll 1,\quad T<T_N.
\end{align*}

Compared to (A.19) with \cite{LTXZZH20162}, we have
\begin{align*}
\rho_n\left(\frac{x_n}{r_n}\right)^2\frac{d^2 v_n}{dt^2}+\frac{1}{h}\left(\left(\frac{d}{dt}B_n-\frac{d}{dt}B_{n+1}\right)+4\nu_1\left(\frac{\frac{dv_{n-1}}{dt}}{r_{n-1}}-\frac{\frac{dv_{n}}{dt}}{r_{n}}\right)\right)=\frac{1}{h}(P_{n+1}-P_{n})+e_n,
\end{align*}
where
\begin{align*}
P_n=P'\left(\frac{x^2_{n-1}\rho_n h}{r^2_{n-1}(r_{n}-r_{n-1})}\right)\frac{x^2_{n-1}\rho_n h}{r^2_{n-1}(r_{n}-r_{n-1})}\left(2\frac{v_{n-1}}{r_{n-1}}+\frac{v_n-v_{n-1}}{r_n-r_{n-1}}\right).
\end{align*}
Consider
\begin{align}\label{7.158}
&\sum_{n=1}^{N-1}\frac{1}{h}(P_{n+1}-P_n)r^2_{n-1}\frac{dv_{n}}{dt},\quad\sum_{n=1}^{N-1}\frac{1}{h}(P_{n+1}-P_n)\psi_n\frac{dv_{n}}{dt},\\
&\frac{P(\rho_{n+1})-P(\rho_n)}{h}-\frac{1}{h}\left(P\left(\frac{x^2_n\rho_{n+1}h}{r^2_n(r_{n+1}-r_n)}\right)-P\left(\frac{x^2_{n-1}\rho_{n}h}{r^2_{n-1}(r_{n}-r_{n-1})}\right)\right),\label{7.159}
\end{align}
where $\psi$ is truncation function defined on $[0,R_\mu]$ satisfying
\begin{align*}
\psi=1\ \textrm{on}\ \left[0,\frac{R_\mu}{4}\right],\ \psi=0\ \textrm{on}\ \left[\frac{R_\mu}{2},R_\mu\right],\ |\psi'|\leq \frac{8}{R_\mu}.
\end{align*}
Denote $\psi_n=\psi(x_n)$.

There holds
\begin{align*}
\sum_{n=1}^{N-1}\frac{1}{h}(P_{n+1}-P_n)r^2_{n-1}\frac{dv_{n}}{dt}
=&\sum_{n=1}^{N-1}\frac{1}{h}P_{n+1}\left(r^2_{n-1}\frac{dv_{n}}{dt}-r^2_n\frac{dv_{n+1}}{dt}\right)\\
=&\sum_{n=1}^{N-1}P_{n+1}\left(r_n^2\frac{\frac{dv_{n}}{dt}-\frac{dv_{n+1}}{dt}}{h}+\frac{r^2_{n-1}-r^2_n}{h}\frac{dv_{n}}{dt}\right)\\
\lesssim&\frac{1}{4}\sum_{n=1}^{N-1}\left(x^2_n\left(\frac{\frac{dv_{n+1}}{dt}-\frac{dv_{n}}{dt}}{h}\right)^2+\left(\frac{dv_{n}}{dt}\right)^2\right)+\sum_{n=1}^{N-1}P^2_{n+1}x^2_n
\end{align*}
and
\begin{align*}
\sum_{n=1}^{N-1}P^2_{n+1}x^2_n\lesssim \sum_{n=1}^{N-1}x^2_n(P'(\rho_n)\rho_n E_N(t))^2\lesssim E^2_N(t)
\end{align*}
provided that $\|r_x^0-1\|_{\infty}< \varepsilon\ll 1$.
 The estimate of the second term in \eqref{7.158} is obtained similarly.

To estimate the term \eqref{7.159}, we notice that
\begin{align*}
&\frac{P(\rho_{n+1})-P(\rho_n)}{h}-\frac{1}{h}\left(P\left(\frac{x^2_n\rho_{n+1}h}{r^2_n(r_{n+1}-r_n)}\right)-P\left(\frac{x^2_{n-1}\rho_{n}h}{r^2_{n-1}(r_{n}-r_{n-1})}\right)\right)\\
\lesssim&\frac{P'(\rho_{n+1})\rho_{n+1}}{h}\left(\frac{x^2_n}{r^2_n}\frac{h}{r_{n+1}-r_n}-1\right)+\frac{P'(\rho_{n})\rho_{n}}{h}\left(\frac{x^2_{n-1}}{r^2_{n-1}}\frac{h}{r_{n}-r_{n-1}}-1\right)\\
\lesssim&\frac{\rho_{n+1}}{h}\left(\left|\frac{r_n}{x_n}-1\right|+\left|\frac{r_{n+1}-r_n}{h}-1\right|\right)+\frac{\rho_{n}}{h}\left(\left|\frac{r_{n-1}}{x_{n-1}}-1\right|+\left|\frac{r_{n}-r_{n-1}}{h}-1\right|\right).
\end{align*}
The details of rest proof are similar to Lemma A.2 in \cite{LTXZZH20162} which we omit.
\end{proof}

Due to the above Lemma, we obtain the unique local strong solution of \eqref{5.16}-\eqref{7.31} with initial data $(r^0,v^0)$ by the similar arguments  in \cite{LTXZZH20162}.

\section{Nonlinear instability under the spherically
symmetric perturbation}\label{Nonlinearinstability}
In this section, we consider  the unstable case $\frac{dM_\mu}{d\mu}\neq 0$ and $n^-(L_{\mu})>0$. The nonlinear estimates are established by adopting the ideas  from \cite{JJTI2013}. However, our proof doesn't depend on the polytropic case of $\frac{6}{5}<\gamma<\frac{4}{3}$.  Fortunately, we can use a similar variational structure to deal with the maximum eigenvalue problem to obtain the linear estimates by $n^-(L_{\mu})>0$. To obtain the nonlinear energy estimates, we give an additional suitable regular assumption of the pressure. Then, a bootstrap argument  shows that all of energy estimates grow no faster than the linear growth rate. At last, one can choose a family of initial data for the nonlinear problem to show the Lyapunov instability.

\subsection{The linear unstable estimates}
The following theorem provides the linear instability of the non-rotating gaseous stars of the Navier-Stokes-Poisson system.

\begin{theorem}\label{thm6.1}
Assume that $\frac{dM_\mu}{d\mu}\neq 0$ and $n^-(L_{\mu})>0$.  Then, the non-rotating stars of the  Navier-Stokes-Possion system \eqref{1.4}-\eqref{1.5} is linearly unstable. The total algebraic multiplicity of  eigenvalues on the right half plane to the eigenvalue problem \eqref{3.7} is  as the same as  $n^-(L_{\mu})$. Moreover, we obtain
\begin{align}
&e^{-2\lambda t}\int_0^{R_\mu}\left( r^2\rho_\mu |u|^2+\frac{\nu_2}{r^2}|(r^2u)_r|^2+\frac{4\nu_1}{3}r^4\left|\left(\frac{u}{r}\right)_r\right|^2+\Phi''(\rho_\mu )r^2|\rho|^2\right)dr\nonumber\\
\lesssim&\int_0^{R_\mu} \Big(r^2\rho_\mu |u|^2+\frac{\nu_2}{r^2}|(r^2u)_r|^2+\frac{4\nu_1}{3}r^4\left|\left(\frac{u}{r}\right)_r\right|^2+\frac{\Phi''(\rho_\mu )}{r^2}\left|(r^2\rho_\mu u)_r\right|^2\nonumber\\
&+\Phi''(\rho_\mu )r^2|\rho|^2+r^2\rho_\mu |u_t|^2\Big)dr\Big|_{t=0},\label{6.3}
\end{align}
where $\lambda$ is the maximum of eigenvalues on the right half plane.
\end{theorem}
\begin{proof}
From Theorem \ref{nEPeqnNSP} or \ref{thm: KTC}, we obtain the linearly instability of the non-rotating stars. We only need to prove the estimate \eqref{6.3}. Adopting the ideas  from \cite{JJTI2013}, the arguments are as follows.
\end{proof}

\begin{lemma}
There holds
\begin{align}
&\frac{1}{2}\frac{d}{dt}\int_0^{R_\mu}   r^2\rho_\mu |u_t|^2dr+\frac{4\nu_1}{3}\int_0^{R_\mu}   r^4\left|\left(\frac{u_t}{r}\right)_r\right|^2dr+\nu_2\int_0^{R_\mu}   \frac{1}{r^2}|(r^2u_t)_r|^2dr\nonumber\\
&+\frac{1}{2}\frac{d}{dt}\int_0^{R_\mu}   \left(\frac{\Phi''(\rho_\mu )}{r^2}|(r^2\rho_\mu  u)_r|^2-4\pi r^2\rho_\mu ^2|u|^2\right)dr=0.\label{6.4}
\end{align}
\end{lemma}
\begin{proof}

Multiplying \eqref{3.4} by $r^2\rho_\mu u_t$ and integrating from $0$ to $R_\mu$, we obtain \eqref{6.4} by integrating by parts.
\end{proof}

\begin{lemma}\label{lem6.3}
Let $\lambda$ be the maximum of eigenvalues on the right half plane to the eigenvalue problem \eqref{3.7}. Then, there exists  a positive constant $m$ satisfying $m\geq 4\pi \sup_{[0,R_\mu]}\rho_\mu$ such that
\begin{align*}
\|u\|^2_{X_\mu}+\int^t_0\|u\|^2_Dds&\leq e^{2\lambda t}\|u(\cdot,0)\|_{X_\mu}^2+\frac{I_7}{2\lambda}(e^{2\lambda t}-1),\\
\int_0^{R_\mu}   \frac{\Phi''(\rho_\mu )}{r^2}|(r^2\rho_\mu  u)_r|^2dr
&\leq I_8+m\left(e^{2\lambda t}\|u(\cdot,0)\|_{X_\mu}^2+\frac{I_7}{2\lambda}(e^{2\lambda t}-1)\right),
\end{align*}
where
\begin{align*}
&\|u\|^2_D=\langle Du,u\rangle_{X_\mu}=\int_0^{R_\mu}  \left(\frac{\nu_2}{r^2}|(r^2u)_r|^2+\frac{4\nu_1}{3}r^4\left|\left(\frac{u}{r}\right)_r\right|^2\right)dr,\\
&I_7=\frac{2I_8}{\lambda}+2\int_0^{R_\mu}  \left(\frac{\nu_2}{r^2}|(r^2u(r,0))_r|^2+\frac{4\nu_1}{3}r^4\left|\left(\frac{u(r,0)}{r}\right)_r\right|^2\right)dr,\\
&I_8=\int_0^{R_\mu}   r^2\rho_\mu |u_t(r,0)|^2dr+\int_0^{R_\mu}  \frac{\Phi''(\rho_\mu )}{r^2}|(r^2\rho_\mu  u(r,0))_r|^2dr.
\end{align*}
\end{lemma}
\begin{proof}
Integrating \eqref{6.4} from $0$ to $t$, we have
\begin{align}
\notag&\int_0^{R_\mu}   r^2\rho_\mu |u_t|^2dr+\frac{8\nu_1}{3}\int^t_0\int_0^{R_\mu}   r^4\left|\left(\frac{u_t}{r}\right)_r\right|^2dr
+2\nu_2\int^t_0\int_0^{R_\mu}   \frac{1}{r^2}|(r^2u_t)_r|^2dr\\
\notag&+\int_0^{R_\mu}   \left(\frac{\Phi''(\rho_\mu )}{r^2}|(r^2\rho_\mu  u)_r|^2-4\pi r^2\rho_\mu ^2|u|^2\right)dr\\
=&\int_0^{R_\mu}   \left(r^2\rho_\mu |u_t|^2+\frac{\Phi''(\rho_\mu )}{r^2}|(r^2\rho_\mu  u)_r|^2-4\pi r^2\rho_\mu ^2|u|^2\right)dr\Bigg|_{t=0}.\label{6.16}
\end{align}
Because  $\lambda$ is the maximum of eigenvalues on the right half plane, by \eqref{3.7}, it obtains that
\begin{align}
\notag&\lambda^2\int_0^{R_\mu}  r^2\rho_\mu |u|^2dr+\lambda\int_0^{R_\mu}  \left(\frac{\nu_2}{r^2}|(r^2u)_r|^2+\frac{4\nu_1}{3}r^4\left|\left(\frac{u}{r}\right)_r\right|^2\right)dr\\
&+\int_0^{R_\mu}  \left(\frac{1}{r^2}\Phi''(\rho_\mu )|(r^2 \rho_\mu u)_r|^2-4\pi r^2\rho_\mu ^2|u|^2\right)dr\geq 0.\label{8.4}
\end{align}
Hence, there holds
\begin{align*}
&\int_0^{R_\mu}   r^2\rho_\mu |u_t|^2dr+\frac{8\nu_1}{3}\int^t_0\int_0^{R_\mu}   r^4\left|\left(\frac{u_t}{r}\right)_r\right|^2dr
+2\nu_2\int^t_0\int_0^{R_\mu}   \frac{1}{r^2}|(r^2u_t)_r|^2dr\\
\leq&\lambda^2\int_0^{R_\mu}  r^2\rho_\mu |u|dr+\lambda\int_0^{R_\mu}  \left(\frac{\nu_2}{r^2}|(r^2u)_r|^2+\frac{4\nu_1}{3}r^4\left|\left(\frac{u}{r}\right)_r\right|^2\right)dr\\
&+\left(\int_0^{R_\mu}   r^2\rho_\mu |u_t|^2dr+\int_0^{R_\mu}  \frac{\Phi''(\rho_\mu )}{r^2}|(r^2\rho_\mu  u)_r|^2dr\right)\Bigg|_{t=0}.
\end{align*}
Rewriting it, we obtain
\begin{align}\label{6.22}
\|u_t\|^2_{X_\mu}+2\int^t_0\|u_t\|_D^2ds\leq I_8+\lambda^2\|u\|_{X_\mu}^2+\lambda\|u\|^2_D.
\end{align}
Note that
\begin{align*}
\lambda\|u\|_{X_\mu}^2=\lambda\|u(\cdot,0)\|_{X_\mu}^2+2\lambda\int^t_0\langle u,u_t\rangle_{X_\mu} ds
\leq\lambda\|u(\cdot,0)\|_{X_\mu}^2+\int^t_0\|u_t\|^2_{X_\mu}ds+\lambda^2\int^t_0\|u\|_{X_\mu}^2ds
\end{align*}
and
\begin{align*}
\lambda\frac{d}{dt}\|u\|^2_{X_\mu}\leq \lambda^2\|u\|_{X_\mu}^2+\|u_t\|^2_{X_\mu}.
\end{align*}
We have
\begin{align*}
\frac{d}{dt}\|u\|^2_{X_\mu}+\|u\|^2_D\leq I_7+2\lambda\|u\|^2_{X_\mu}+2\lambda\int^t_0\|u\|^2_Dds.
\end{align*}
By Gronwall's inequality, it obtains that
\begin{align}\label{6.27}
\|u\|^2_{X_\mu}+\int^t_0\|u\|^2_Dds\leq e^{2\lambda t}\|u(\cdot,0)\|_{X_\mu}^2+\frac{I_7}{2\lambda}\left(e^{2\lambda t}-1\right).
\end{align}
By \eqref{6.22}, we also have
\begin{align*}
\frac{1}{\lambda}\|u_t\|^2_{X_\mu}+\|u\|^2_D\leq e^{2\lambda t}\left(2\lambda\|u(\cdot,0)\|_{X_\mu}^2+I_7\right).
\end{align*}
By \eqref{6.16} and \eqref{6.27}, there exists a constant $m\geq 4\pi \sup_{[0,R_\mu]}\rho_\mu$ such that
\begin{align*}
\int_0^{R_\mu}   \frac{\Phi''(\rho_\mu )}{r^2}|(r^2\rho_\mu  u)_r|^2dr
\leq& I_8+\int_0^{R_\mu}  4\pi r^2\rho_\mu ^2|u|^2dr\\
\leq&I_8+m\left(e^{2\lambda t}\|u(\cdot,0)\|_{X_\mu}^2+\frac{I_7}{2\lambda}(e^{2\lambda t}-1)\right).
\end{align*}
It completes the proof.
\end{proof}

Now, we prove \eqref{6.3}. By Lemma \ref{lem6.3}, we have
\begin{align}
\notag&\int_0^{R_\mu} \left(r^2\rho_\mu |u|^2+\frac{\nu_2}{r^2}|(r^2u)_r|^2+\frac{4\nu_1}{3}r^4\left|\left(\frac{u}{r}\right)_r\right|^2+\frac{\Phi''(\rho_\mu )}{r^2}|(r^2\rho_\mu u)_r|^2\right)dr\\
\notag\lesssim&e^{2\lambda t}\left[\int_0^{R_\mu} \left(r^2\rho_\mu |u|^2+\frac{\nu_2}{r^2}|(r^2u)_r|^2+\frac{4\nu_1}{3}r^4\left|\left(\frac{u}{r}\right)_r\right|^2+\frac{\Phi''(\rho_\mu )}{r^2}|(r^2\rho_\mu u)_r|^2+r^2\rho_\mu |u_t|^2\right)dr\right]_{t=0}\\
=:&e^{2\lambda t}I_9.\label{6.37}
\end{align}
By \eqref{1.4}, there holds
\begin{align*}
\int_0^{R_\mu}\Phi''(\rho_\mu )r^2|\rho_t|^2dr=\int_0^{R_\mu}\frac{\Phi''(\rho_\mu )}{r^2}|(r^2\rho_\mu  u)_r|^2dr.
\end{align*}
By $\partial_t\|u\|\leq \|\partial_t u\|$, we have
\begin{align*}
\partial_t\left(\int_0^{R_\mu}\Phi''(\rho_\mu )r^2|\rho|^2dr\right)^{\frac{1}{2}}\leq\left(\int_0^{R_\mu}\Phi''(\rho_\mu )r^2|\partial_t\rho|^2dr\right)^\frac{1}{2}\lesssim I_9^{\frac{1}{2}}e^{\lambda t}.
\end{align*}
Integrating it from $0$ to $t$, it yields that
\begin{align}\label{6.42}
\left(\int_0^{R_\mu}\Phi''(\rho_\mu )r^2|\rho|^2dr\right)^{\frac{1}{2}}\lesssim e^{\lambda t}\left(\int_0^{R_\mu}\Phi''(\rho_\mu )r^2|\rho(r,0)|^2dr+I_9\right)^{\frac{1}{2}}.
\end{align}
Combining with  \eqref{6.37} and \eqref{6.42}, it completes the proof.

\subsection{Formulation in Lagrangian mass coordinates}\label{sub5.2}
To compare with \cite{JJTI2013} conveniently, we  redefine the symbols of each variable in this part. Let $r=|y|=\sqrt{y_1^2+y_2^2+y_3^2}$ be the radius of the spherical coordinate system in 3-dimensional Euclidean space. We define the following Lagrangian mass coordinates
\begin{align*}
x(r,t):=\int^r_0 4\pi s^2\rho(s,t)ds=\int_{B(0,r)}\rho(y,t)dy.
\end{align*}
From $x_t(r,t)=\int^r_0 4\pi s^2\rho_t (s,t)ds=-\int^r_0(s^2\rho(s,t) u(s,t))_s ds=-r^2\rho(r,t)u(r,t)$, we have $x_t(R(t),t)=0$. It implies that the total mass $M=\int^{R(t)}_0 4\pi s^2\rho(s,t)ds=\int^{R_\mu}_0 4\pi s^2\rho_\mu(s)ds$ is preserved in
time. Then, the domain of $x$ is $[0,M]$.

The  stationary solution $(\rho_\mu(x),0,r_\mu(x))$ of \eqref{1.4}-\eqref{1.5} satisfies
\begin{align*}
4\pi r^2_\mu(x)( P(\rho_\mu))_x +\frac{x}{r^2_\mu(x)}=0,
\end{align*}
where
\begin{align}\label{rmu}
r_\mu(x)=\left(\frac{3}{4\pi}\int^x_0\frac{1}{\rho_\mu(y)}dy\right)^{\frac{1}{3}}.
\end{align}

Now, we linearize \eqref{1.4}-\eqref{1.5} in Lagrangian mass coordinates around the stationary solution $(\rho_\mu,0,r_\mu)$:
\begin{align}\label{8.12}
&\varrho_t+4\pi\rho_\mu^2(r_\mu^2\upsilon)_x= 0,\ x\in(0,M),\\
&\upsilon_t+4\pi r_\mu^2(P'(\rho_\mu)\varrho)_x+\frac{x}{\pi r^5_\mu}\int^x_0\frac{\varrho(y,t)}{\rho_\mu^2(y)}dy=16\pi^2 \nu r^2_\mu(\rho_\mu(r^2_\mu \upsilon)_x)_x,\ x\in(0,M),\label{8.122}
\end{align}
with linearized boundary conditions
\begin{align}
&\upsilon(0,t)=\varrho(M,t)=0,\label{8.152}\\
&P'(\rho_\mu)\varrho-\frac{4\nu_1}{3}\left(4\pi r^2_\mu\rho_\mu\upsilon_x-\frac{\upsilon}{r_\mu}\right)-\nu_2\left(4\pi r^2_\mu\rho_\mu \upsilon_x+\frac{2\upsilon}{r_\mu}\right)\Bigg|_{x=M}=0.\label{8.15}
\end{align}
The details of calculation are stated in the Part I of appendix.

\subsection{Nonlinear estimates for instability}

One can see the variational characterization \eqref{8.4} is equivalent to
\begin{align}
\notag&\lambda\int^M_0\left(\nu_2\rho_\mu|\zeta_x|^2+\frac{4\nu_1}{3}\rho_\mu\left|r^3_\mu\left(\frac{\zeta}{r^3_\mu}\right)_x\right|^2\right)dx+\frac{1}{2}\int^M_0\left(P'(\rho_\mu)\rho^2_\mu|\zeta_x|^2+\frac{(P(\rho_\mu))_x}{\pi r_\mu^3}|\zeta|^2\right)dx\nonumber\\
\geq& -\lambda^2\int^M_0\frac{|\zeta|^2}{16\pi^2r^4_\mu}dx\label{8.6}
\end{align}
for any $\sqrt{\rho_\mu}\zeta_x\in L^2(0,M)$ and $\frac{\zeta}{r^2\sqrt{\rho_\mu}}\in L^2(0,M)$ by setting $\zeta=\sqrt{4\pi}r^2_\mu u$, which generalizes (4) in Theorem 2.1 of \cite{JJTI2013}. 

By the unstable assumptions $\frac{dM_\mu}{d\mu}\neq 0$ and  $n^-(L)>0$, we obtain a growing mode solution to the linearized equations with fastest rate by adopting the variational analysis in \cite{JJTI2013}. We also obtain the  estimates which are similar to those in Theorem 3.2, 3.3 and 3.5 of \cite{JJTI2013}  by replacing $\gamma \rho_\mu^\gamma$, $\gamma K\rho_\mu ^{\gamma-1}$, $\gamma K\rho_\mu ^{\gamma+1}$ by $\frac{1}{K_1}P'(\rho_\mu )\rho_\mu $, $P'(\rho_\mu )$, $P'(\rho_\mu )\rho_\mu^2$ respectively. Hence, we have the following results.
\begin{theorem}
Under the assumptions $\frac{dM_\mu}{d\mu}\neq 0,$ $n^-(L)>0$ and (P1). There exist $\lambda>0$ satisfying \eqref{8.6} and $(\varrho(x),\upsilon(x))$ such that $(e^{\lambda t}\varrho(x),e^{\lambda t}\upsilon(x))$ solve the
linearized Navier-Stokes-Poisson system \eqref{8.12}-\eqref{8.15}. Moreover, by \eqref{8.6}, we know that $\lambda>0$  is  the largest possible growth rate. That is no solution to the linearized system  grows in time at a rate faster than $e^{\lambda t}.$
\end{theorem}

Next, we consider the nonlinear estimates. We additionally assume that\\
(P3) $P(s)\in C^5(0,+\infty)$ and there exists $K_3>0$ such that $$\lim_{s\rightarrow 0^+}|P^{(4)}(s)| s^{4-\gamma_1}=\lim_{s\rightarrow 0^+}|P^{(5)}(s)| s^{5-\gamma_1}=K_3.$$

  The system of the small perturbations $\sigma,v$ around the stationary solution  is rewritten as follows:
\begin{align}
&\frac{\rho_\mu}{\rho}\left(\frac{\sigma}{\rho_\mu}\right)_t+4\pi\rho(r^2v)_x=0,\label{8.190}\\
&v_t+4\pi r^2\left(P'(\rho_\mu)\sigma+\frac{1}{2}P''(a^*\rho_\mu)\sigma^2\right)_x+x\left(\frac{1}{r^2}-\frac{r^2}{r^4_\mu}\right)=16\pi^2\nu r^2(\rho(r^2v)_x)_x,\label{8.19}
\end{align}
where $a_*$ is the smooth bounded remainder such that
\begin{align*}
P(\rho)-P(\rho_\mu)=P'(\rho_\mu )\sigma+\frac{1}{2}P''(a^*\rho_\mu )\sigma^2
\end{align*}
 and $r(x,t)$ is determined by
\begin{align*}
r^3(x,t)=\frac{3}{4\pi}\int^x_0\frac{1}{\rho_\mu(y)+\sigma(y,t)}dy,\ r_t(x,t)=v(x,t).
\end{align*}
The boundary conditions are
\begin{align}
\nu_2\rho(r^2v)_x+\frac{4\nu_1}{3}\rho r^3\left(\frac{x}{r}\right)_x\Big|_{x=M}=0,\ \rho(M,t)=v(0,t)=0.\label{BC}
\end{align}

Define
\begin{align*}
&\mathscr{E}^{0,\sigma}=\frac{1}{2}\int^M_0\frac{P'(\rho_\mu )}{\left(1+\frac{\sigma}{\rho_\mu }\right)^2}\left|\frac{\sigma}{\rho_\mu }\right|^2dx,\\
&\mathscr{D}^{1,\sigma}=\int^M_0 16\pi^2r^4 P'(\rho_\mu )\rho_\mu \left|\left(\frac{\sigma}{\rho_\mu }\right)_x\right|^2dx,\\
&\mathscr{D}^{1,\sigma}_b=\int^M_0 16\pi^2r^2 P'(\rho_\mu ) \left|\left(\frac{\sigma}{\rho_\mu }\right)_x\right|^2dx,\\
&\mathscr{E}^2=\frac{1}{2}\int^M_0|v_t|^2dx+\frac{1}{2}\int^M_0\frac{P'(\rho_\mu )}{\left(1+\frac{\sigma}{\rho_\mu }\right)^2}\left|\left(\frac{\sigma}{\rho_\mu }\right)_t\right|^2dx,\\
&\mathscr{D}^4=4\pi\int^M_0 r^2 P'(\rho_\mu )\rho_\mu \rho\left|\left(r^4\left(\frac{\sigma}{\rho_\mu }\right)_x\right)_x\right|^2dx,\\
&\mathscr{E}^{2,\sigma}=\frac{1}{2}\int^M_0\frac{P'(\rho_\mu )}{\left(1+\frac{\sigma}{\rho_\mu }\right)^2}\left|\left(\frac{\sigma}{\rho_\mu }\right)_t\right|^2dx,\\
&\mathscr{E}^{1+2i}=\frac{1}{2}\int^M_0|\partial^i_t v|^2dx+\frac{1}{2}\int^M_0\frac{P'(\rho_\mu)}{\left(1+\frac{\sigma}{\rho_\mu}\right)^2}\left|\partial_t^i\left(\frac{\sigma}{\rho_\mu}\right)\right|^2dx,\ i=2,3,
\end{align*}
which generalize those in  \cite{JJTI2013}. There is no change with the other terms of (4-11)-(4-18) in \cite{JJTI2013}. Following the ideas in Section 4 of \cite{JJTI2013}, we obtain the same estimates like Lemma 4.1-4.9 in \cite{JJTI2013} by slight changes.
By adopting a continuity method in Section 5 of \cite{JJTI2013}, we also have the nonlinear instability of the Lane-Emden gaseous star modeled by the Navier-Stokes-Poisson system.
\begin{theorem}
\label{thm:nonlinear instability}Under the assumptions $\frac{dM_\mu}{d\mu}\neq 0,$ $n^{-}(L)>0$, (P1), (P2) and (P3), there exist $\varsigma_{0}>0,C>0$ and $0<\iota_{0}<\varsigma_{0}$
such that for any $0<\iota\leq\iota_{0}$, there exists a family of solutions
$(\sigma^{\iota},v^{\iota})$ to the Navier-Stokes-Poisson system \eqref{8.190}-\eqref{8.19} with boundary conditions \eqref{BC}
satisfying
\[
\sqrt{\mathscr{E}}(0)\leq C\iota,\quad\sup_{0\leq t\leq T^{\iota}}%
\sqrt{\mathscr{E}^{0}}(t)\geq\sup_{0\leq t\leq T^{\iota}}\sqrt
{\mathscr{E}^{0,\sigma^{\iota}}(T^{\iota})+\mathscr{E}^{0,v^{\iota}}(T^{\iota
})}\geq\varsigma_{0}.
\]
Here, $T^{\iota}$ is given by $T^{\iota}=\frac{1}{\lambda}\ln\frac{\varsigma_{0}%
}{\iota}$. $\mathscr{E}^{0,v^{\iota}}$ and $\mathscr{E}^{0,\sigma^{\iota}}$
are defined by
\begin{align*}
  \mathscr{E}^{0,v^{\iota}}=\frac{1}{2}\int_{0}^{M}|v|^{2}dx,\
  \mathscr{E}^{0,\sigma^{\iota}}=\frac{1}{2}\int_{0}^{M}\frac{P^{\prime}%
(\rho_{\mu})}{\left(1+\frac{\sigma}{\rho_{\mu}}\right)^2}\left|\frac{\sigma}{\rho_{\mu}}\right|^{2}dx.
\end{align*}

\end{theorem}

\section*{Appendix}\label{s6}
In this section, we consider the equivalence of  linearized system  \eqref{3.1}-\eqref{3.3} in different coordinates.
\subsection*{Part I: Comparing to \cite{JJTI2013}}

Using the Lagrangian mass coordinates formulation in Subsection \ref{sub5.2}, we set the unknown density  $\sigma(x,t)=\rho(r(x,t),t)$ and the velocity $v(x,t)=u(r(x,t),t)$.
Then, the linearized density $\delta\sigma$ and the linearized velocity $\delta v$ around the stationary solution $(\rho_\mu,0,r_\mu)$ satisfy
\begin{align*}
\delta\sigma(x,t)=&\frac{d}{d\epsilon}(\rho_\mu +\epsilon\delta\rho)(r_\mu(x)+\epsilon\delta r(x,t),t)\Big|_{\epsilon=0}=(\rho_\mu(r_\mu))_{r_\mu}\delta r+\delta\rho(r_\mu,t),\\
 \delta v(x,t)=&\delta u(r_\mu,t).
\end{align*}
Therefore, by  \eqref{3.1}, we have
\begin{align*}
\delta\sigma_t(x,t)
=&((\rho_\mu(r_\mu))_{r_\mu} \delta r)_t+\delta\rho_t(r_\mu,t)=(\rho_\mu(r_\mu))_{r_\mu}\delta v- \frac{1}{r^2}(r^2\rho_\mu \delta v)_r\Big|_{r=r_\mu}\\
=&-\rho_\mu \frac{1}{r^2}(r^2\delta v)_r\Big|_{r=r_\mu}.
\end{align*}
By \eqref{rmu}, we have $\partial_{r_\mu}=4\pi \rho_\mu  r_\mu^2\partial_x$. There holds
\begin{align*}
\delta\sigma_t=-4\pi\rho_\mu ^2(r^2_\mu\delta v)_x,
\end{align*}
which coincides with (2-1) in \cite{JJTI2013} for the polytropes.

By \eqref{3.2} and \eqref{3.3}, we have
\begin{align*}
\delta v_t&=\frac{1}{\rho_\mu }\left(\frac{1}{r^2}\left(\frac{4}{3}\nu_1+\nu_2\right)(r^2 u)_r\right)_r-\frac{1}{\rho_\mu }(P'(\rho_\mu)\delta\sigma)_r-\left(\int^r_0 \frac{4\pi}{y^2}\int_0^y\delta\sigma(s)s^2dsdy\right)_r\Big|_{r=r_\mu}\\
&=16\pi^2 r_\mu^2 \left(\rho_\mu \left(\frac{4}{3}\nu_1+\nu_2\right)_x(r_\mu^2 v)_x\right)-4\pi r_\mu^2(P'(\rho_\mu)\delta\sigma)_x-\frac{1}{2}\frac{x}{\pi r_\mu^5}\int_0^x \frac{\delta\sigma(y,t)}{\rho_\mu ^2(y)}dy,
\end{align*}
which coincides with (2-2) in \cite{JJTI2013}  for the polytropes.

The linearized boundary conditions
\begin{align*}
&\delta u(0,t)=0,\ \delta\rho(R(t),t)=0,\\
&-\frac{4\nu_1}{3}\left(\delta u_r -\frac{\delta u}{r}\right)-\nu_2\left( \delta u_r + \frac{2\delta u}{r}\right)\Bigg|_{r=R(t)}=0
\end{align*}
become
\begin{align*}
&\delta v(0,t)=0,\quad \delta\sigma(M,t)=0,\\
&-\frac{4\nu_1}{3}\left( 4\pi \rho_\mu  r_\mu^2 \delta v_x-\frac{\delta v}{r_\mu}\right)-\nu_2\left( 4\pi \rho_\mu  r_\mu^2\delta v_x + \frac{2\delta v}{r_\mu}\right)\Bigg|_{x=M}=0,
\end{align*}
which coincides with (2-4) in \cite{JJTI2013}  for the polytropes.

\subsection*{Part II: Comparing to \cite{LTXZZH20162}}


We adopt the Lagrangian particle trajectory formulation which can reduce the original free boundary problem to a problem on the fixed domain $x\in [0,R_\mu]$. In this part, let $x$ be the distance to the origin for the Lane-Emden solution. Set the Lagrangian variable $r(x,t)$ such that
\begin{align*}
&r_t(x,t)=u(r(x,t),t),\ \int^x_0\rho_\mu(s)s^2ds=\int^{r(x,t)}_0\rho(s,t)s^2ds,\ x\in (0,R_\mu),\\
&r(0,t)=0,\ r(R_\mu,t)=R(t).
\end{align*}

Denote
\begin{align*}
f(x,t)=\rho(r(x,t),t),\ v(x,t)=u(r(x,t),t).
\end{align*}
Then, the linearized perturbations of $f,v$ are
\begin{align}
\delta f(x,t)&=\frac{d}{d\epsilon}(\rho_\mu+\epsilon\delta\rho)(x+\epsilon\delta r)|_{\epsilon=0}=(\rho_\mu(x))_x\delta r+\delta\rho,\label{9.21}\\
\delta v(x,t)&=\frac{d}{d\epsilon}(0+\epsilon\delta u)(x+\epsilon\delta r)|_{\epsilon=0}=\delta u(x,t).\label{9.23}
\end{align}
By \eqref{3.1}, we get
\begin{align}
\notag\delta f_t(x,t)&=((\rho_\mu)_x\delta r+\delta\rho(x,t))_t=(\rho_\mu)_x\delta v-\frac{1}{x^2}(x^2\rho_\mu\delta v)_x\\
&=-\rho_\mu\frac{1}{x^2}(x^2\delta v)_x=-\rho_\mu\delta v_x-2\rho_\mu\frac{\delta v}{x}.\label{9.27}
\end{align}
From another point of view, by \eqref{5.15},
 we have
\begin{align}\label{5.21}
\delta f&=\frac{d}{d\epsilon}\frac{x^2\rho_\mu(x)}{(x+\epsilon\delta r)^2(x,t)(x+\epsilon\delta r)_x(x,t)}\Big|_{\epsilon=0}=-2\frac{\rho_\mu}{x}\delta r-\rho_\mu\delta r_x.
\end{align}
Then, there holds
\begin{align*}
\delta f_t&=-2\rho_\mu\frac{\delta r_t}{x}-\rho_\mu\delta r_{xt},
\end{align*}
which coincides with \eqref{9.27} because of $\delta r_t=\delta v$.

By \eqref{9.21} and \eqref{5.21}, we have
\begin{align*}
\delta\rho=\left(-2\frac{\rho_\mu}{x}-(\rho_\mu)_x\right)\delta r-\rho_\mu\delta r_x.
\end{align*}
Applying it by $\partial_t$ and omitting $\delta$, by \eqref{9.23},  there holds
\begin{align*}
\rho_t=\left(-2\frac{\rho_\mu}{x}-(\rho_\mu)_x\right)v-\rho_\mu v_x.
\end{align*}
That is
\begin{align}
\rho_t+\frac{1}{x^2}(x^2\rho_\mu v)_x=0,
\end{align}
which coincides with \eqref{3.1}.

Applying $\partial_t$ to \eqref{5.23} or (A.41a) in \cite{LTXZZH20162}, it obtains that
\begin{align}\label{9.35}
x\rho_\mu w_{ttt}-\left(3(\rho_\mu P'(\rho_\mu))_x-4(P(\rho_\mu))_x\right)w_t-x^{-3}(\rho_\mu P'(\rho_\mu)x^4 w_{xt})_x=\nu(xw_{xtt}+3w_{tt})_x.
\end{align}
Due to $u=xw_{t}$, we have
\begin{align}\label{9.36}
u_{tt}-\frac{1}{x\rho_\mu}\left(3(\rho_\mu P'(\rho_\mu))_x-4(P(\rho_\mu))_x\right)u-\frac{1}{x^3\rho_\mu}\left(\rho_\mu P'(\rho_\mu)x^4 \left(\frac{u}{x}\right)_{x}\right)_x
=\frac{\nu}{\rho_\mu}\left(x\left(\frac{u}{x}\right)_{xt}+3\left(\frac{u}{x}\right)_{t}\right)_x.
\end{align}
Since
\begin{align*}
(P(\rho_\mu))_x=-\frac{\rho_\mu}{x^2}\int^x_0 4\pi\rho_\mu s^2ds,
\end{align*}
we have
\begin{align*}
-4\pi \rho_\mu x^2=\left(\frac{P(\rho_\mu)_x x^2}{\rho_\mu}\right)_x.
\end{align*}
Hence, \eqref{9.36} coincides with \eqref{3.4}.

Meanwhile, multiplying \eqref{9.35} by $x^3w_{tt}$ and integrating from $0$ to $R_\mu$,  there holds
\begin{align}
\notag&\frac{1}{2}\frac{d}{dt}\int_0^{R_\mu} \left( x^4\rho_\mu|w_{tt}|^2+\rho_\mu P'(\rho_\mu)x^4|w_{xt}|^2-x^3(3(\rho_\mu P'(\rho_\mu))_x-4(P(\rho_\mu))_x)|w_t|^2\right)dx\\
&+\nu\int_0^{R_\mu} x^2(3w_{tt}+xw_{ttx})^2dx=0.\label{9.40}
\end{align}

Note that
\begin{align*}
&\langle D_\mu u_t,u_t\rangle_{X_\mu}=\nu\int_0^{R_\mu} x^2\left(3w_{tt}+xw_{ttx}\right)^2dx,\\
&\langle L_\mu u,u\rangle_{X_\mu}=\int_0^{R_\mu} \left(\rho_\mu P'(\rho_\mu)x^4|w_{xt}|^2-x^3(3(\rho_\mu P'(\rho_\mu))_x-4P(\rho_\mu)_x)|w_t|^2\right)dx.
\end{align*}
Hence, \eqref{9.40} is equivalent to
\begin{align*}
\frac{1}{2}\frac{d}{dt}\left(\langle u_t,u_t\rangle_{X_\mu}+\langle L_\mu u,u\rangle_{X_\mu}\right)+\langle D_\mu u_t,u_t\rangle_{X_\mu}=0.
\end{align*}

Moreover, by \eqref{9.36}, changing the variable $x$ by $r_\mu$ (in fact, from the arguments above, it cannot cause confusion), we have
\begin{align}
\notag&u_{tt}-\frac{1}{r_\mu\rho_\mu}\left(3(\rho_\mu P'(\rho_\mu))_{r_\mu}-4(P(\rho_\mu))_{r_\mu}\right)u-\frac{1}{{r_\mu}^3\rho_\mu}\left(\rho_\mu P'(\rho_\mu)r_\mu^4 \left(\frac{u}{r_\mu}\right)_{r_\mu}\right)_{r_\mu}\\
&=\frac{\nu}{\rho_\mu}\left(r_\mu\left(\frac{u}{r_\mu}\right)_{r_\mu t}+3\left(\frac{u}{r_\mu}\right)_{t}\right)_{r_\mu}.\label{9.45}
\end{align}
Since
\begin{align*}
r_\mu^3=\frac{3}{4\pi}\int_0^x\frac{1}{\rho_\mu(y) }dy,\ \partial_{r_\mu}=4\pi \rho_\mu  r_\mu^2\partial_x,\ u=-\frac{\phi}{r^2_\mu},\
 (P(\rho_\mu))_{r_\mu}=-\frac{\rho_\mu}{x^2}\int^{r_\mu}_0 4\pi\rho_\mu s^2ds,
\end{align*}
where $x$ denote the Lagrangian mass coordinates without abusing notations, \eqref{9.45} becomes
\begin{align*}
-\frac{\phi_{tt}}{16\pi^2r_\mu^4}=\frac{(P(\rho_\mu))_x}{\pi r^3_\mu}\phi-(\nu\rho_\mu\phi_{xt}+P'(\rho_\mu)\rho_\mu^2\phi_x)_x,
\end{align*}
which coincides with (3-1) in \cite{JJTI2013}  for the polytropes.

\section*{Acknowledgements}
Cheng's research was supported in part by the NSF of China
(Grant No.   12371191), the Scientific and Technological Project of Jilin Provinces Education Department (Grant No. JJKH20231123KJ).

\end{document}